\documentclass[a4paper]{article}
\usepackage[margin=25mm]{geometry}
\usepackage{amsmath}
\usepackage{amsfonts}
\usepackage{amssymb}
\usepackage{amsthm}
\usepackage{graphicx}
\usepackage{verbatim}
\usepackage{enumerate}
\usepackage[utf8]{inputenc}
\usepackage[all,cmtip]{xy}
\usepackage{float}
\usepackage{graphicx,wrapfig}
\usepackage{multicol}
\usepackage{subcaption}
\usepackage{bbm}
\usepackage{mathtools}
\newtheorem{theorem}{Theorem}
\newtheorem{prop}[theorem]{Proposition}

\newtheorem{lem}{Lemma}

\theoremstyle{definition}
\newtheorem{definition}{Definition}

\theoremstyle{definition}

\newcommand{\R}{\mathbb{R}}
\newcommand{\N}{\mathbb{N}}

\newcommand{\e}{\varepsilon}

\newcommand{\divv}{\text{div }}

\newcommand{\rv}{\rvert}
\newcommand{\lv}{\lvert}
\newcommand{\lV}{\lVert}
\newcommand{\rV}{\rVert}

\newlength{\xywd}
\newcommand{\xyrightarrow}[2][]{%
  \sbox{0}{$\scriptstyle#1$}%
  \xywd=\wd0
  \sbox{0}{$\scriptstyle#2$}%
  \ifdim\wd0>\xywd \xywd=\wd0 \fi
  \xymatrix@C\dimexpr\xywd+1em\relax{{}\ar[r]^{#2}_{#1}&{}}%
}

\newcommand{\xuparrow}[1]{%
  {\left\uparrow\vbox to #1{}\right.\kern-\nulldelimiterspace}
}

\immediate\write18{texcount -tex -sum  \jobname.tex > \jobname.wordcount.tex}

\providecommand{\keywords}[1]
{
  \small	
  \textbf{\textit{Keywords---}} #1
}

\title{A particle method for non-local advection-selection-mutation equations}
\author{Frank Ernesto Alvarez\thanks{ CEREMADE (CNRS UMR no. 7534), PSL University, Université Paris - Dauphine, Place du Maréchal De Lattre De Tassigny, Paris CEDEX 16, 75775, France. {\tt\small alvarez-borges@ceremade.dauphine.fr}} \space \thanks{ GMM, INSA Toulouse, 135 Avenue de Rangueil, Toulouse, 31000, France. {\tt\small alvarez-borg@insa-toulouse.fr}}
\space and 
Jules Guilberteau 
\thanks{ 
       Sorbonne Université, CNRS, Université Paris Cité, Inria, Laboratoire Jacques-Louis Lions (LJLL), F-75005 Paris, France. {\tt\small jules.guilberteau@sorbonne-universite.fr}}
}
\date{} 

\begin{document}
\maketitle

\begin{abstract}
The well-posedness of a non-local advection-selection-mutation problem deriving from adaptive dynamics models is shown for a wide family of initial data. A particle method is then developed, in order to approximate the solution of such problem by a regularised sum of weighted Dirac masses whose characteristics solve a suitably defined ODE system. The convergence of the particle method over any finite interval is shown and an explicit rate of convergence is given. Furthermore, we investigate the asymptotic-preserving properties of the method in large times, providing sufficient conditions for it to hold true as well as examples and counter-examples. Finally, we illustrate the method in two cases taken from the literature. 
\end{abstract} \hspace{10pt}


\section{Introduction}

\subsection*{Presentation of the model}
The goal of this paper is to develop a numerical method allowing to approximate the solutions of equations of the form 
\begin{align}
\begin{cases}
\partial_t v(t,x)+\nabla_x \cdot (a(t,x,I_av(t,x))v(t,x))=R(t,x,I_gv(t,x))v(t,x)+\displaystyle{\int\limits_{\R^d}}{m(t,x,y,I_dv(t,x))v(t,y)dy},\\
v\in \mathcal{C}([0,T],L^1(\mathbb{R}^d)),\\
v(0,\cdot)=v^0(\cdot)\in W^{1,1}(\R^d),
\end{cases}
\label{eq intro}
\end{align}
where  
\[ (I_lu)(t,x)=\int\limits_{\mathbb{R}^d}\psi_l(t,x,y)u(t,y)dy, \quad l=a,g,d \]
 are non-local terms and $a$, $R$, $m$ and $\psi_l$ are smooth functions.  

 This general formulation aims to bring together a wide family of PDE models typically used in the field of \textit{adaptive dynamics}. In this context, $x$ represents a phenotypic trait (usually simply called `phenotype' or `trait') which is a characteristic inherent to individuals, and  $x\mapsto v(t,x)$ represents the density of the studied population at time $t\geq 0$. One purpose of adaptive dynamics is to understand the combined effect of selection and mutations (which are usually assumed to be rare and small \cite{dieckmann1996dynamical, geritz1998evolutionarily, metz1995adaptive}) on living populations \cite{burger2000mathematical}. The literature concerning phenotype-structured equations is abundant \cite{barles2009concentration, coville2013convergence, perthame2006transport, diekmann2005dynamics, lorz2013populational, perthame2008dirac, jabin2011selection}.
 The model proposed in this paper (which includes, among others, the equations studied in \cite{lorenzi2020asymptotic, calsina2013asymptotics, bonnefon2015concentration, guilberteau2023long, friedman2009asymptotic, desvillettes2008selection}) takes into account 
 \begin{itemize}
 \item Selection and growth, via the term `$R(t,x,I_gv(t,x))v(t,x)$', where $R$ can be interpreted as the instantaneous growth rate, which depends on the trait $x$ and the whole population.  
 \item Mutation, via the term `$\int_{\R^d}{m(t,x,y,I_dv(t,x))v(t,y)dy}$', where the function $m$ can be seen as the probability density for a cell of trait $y$ to mutate into a cell of trait $x$.
 \item Advection, via the term `$\nabla_x \cdot (a(t,x,I_av(t,x))v(t,x))$'. This term models how the environment drives the individuals towards specific regions, as opposed to random mutations. Among others, this term can be used in order to model a cell differentiation phenomenon. 
 \end{itemize}
 Mutations can also be modelled through a second order differential operator such as in \cite{Tom} and  \cite{FAJCJC}. Laplacian-like terms can be approximated by integral operators, as shown in \cite{degond1989weighted}, which means that, after choosing an appropriate integral approximation, our analysis could be extended to deal with second order equations. The other two non local-terms ($I_a$ and $I_g$) allow to take into account the influence of the environment, created by the whole population, over the behaviour of the individuals \cite{yates2004combining, friedman2009asymptotic}, or competition between individuals \cite{perthame2006transport}. 

 The long time behaviour of models considering only one phenomenon among selection, advection and mutation is well-known: broadly speaking, it has been shown that considering selection alone, or advection alone, leads to concentration phenomena (towards a finite number of traits) \cite{lorenzi2020asymptotic, perthame2006transport, diperna1989ordinary}, meaning that the density converges to a sum of Dirac masses, while mutations by themselves have a smoothing effect \cite{gyllenberg2005impossibility}. Nevertheless, the combined effects of these terms remains unclear, and may lead to different and non-intuitive behaviours. As an example, considering both selection and advection can lead to convergence either to a Dirac mass or to a continuous function \cite{guilberteau2023long, cooney2022long}, and considering mutation and selection leads either to convergence to a non-smooth measure or to a continuous function \cite{bonnefon2015concentration}. Note that this model also includes the equation studied in \cite{barbarroux2016multi}, which was also approximated with a particle method.
  
Upon establishing the well-posedness of \eqref{eq intro}, this paper is concerned with the derivation of a particle method inspired by \cite{degond1989weighted}, the analysis of its convergence and asymptotic-preserving properties. However, we must emphasise the two main novelties with respect to that work: First, the use of non local terms, which as we will show, poses technical difficulties and affects the existence of smooth solutions in certain cases. Secondly, the study of the asymptotic preserving property, which guarantees that, under certain hypotheses, the long time behaviour of the solution is conserved. As will be seen, these equations are naturally posed in the space of Radon measures, making particle methods a natural tool to approximate then. Compared to finite volume or finite element methods, they are more easily implemented. Furthermore, a change in model leads to very few changes in the corresponding code, a clear advantage over other methods.  

\subsection*{Particle method}
Particle methods use ODE resolution in order to approximate the solution of PDEs. This makes them particularly easy to implement, as they only require a classical ODE solver. The main idea is to seek a sum of weighted Dirac masses, called particle solution,  which is denoted 
\begin{equation}
v^N(t)=\sum\limits_{i=1}^N{\alpha_i(t)\delta_{x_i(t)}},
\label{vN intro}
\end{equation}
where
the weights $\alpha_i$ and the points $x_i$ are solutions of a suitable ODE system.

In order to recover a smooth function close to the solution of the studied PDE, the particle solution needs to be regularised: this is usually done by means of a convolution with a so-called `cut-off function' $\varphi_{\e}$ which must satisfy some specific properties.  We denote this regular solution 
\begin{equation}
v^N_\e(t,x)=\sum\limits_{i=1}^N{\alpha_i(t)\varphi_\e(x-x_i(t))},
\label{vNeps intro}
\end{equation}
where the scaling parameter $\e$ is a function of $N$. 

This method is especially adapted for the linear advection equation `$\partial_t v(t,x)+\nabla \cdot \left(a(x)v(t,x)\right)=0$', 
but has been generalised to many other kinds of equations which  mostly come from physics \cite{hockney2021computer}, such as diffuson equations \cite{degond1990deterministic, russo1990deterministic, carrillo2019blob, leonard1980vortex, leonard1985computing}, advection-diffusion equations \cite{lucquin1992conservative, hermeline1989deterministic},  convection-diffusion equations \cite{degond1989weighted}, the Navier-Stokes equation \cite{cottet1990particle, choquin1989particles} or the Vlasov-Poisson equation \cite{raviart1985analysis, cottet1986particle}.

 We apply the particle method by following its main three steps, as described in \cite{chertock2017practical}:
\begin{enumerate}[(i)]
    \item \textbf{Particle approximation of the initial data.} This first step consists of approaching the initial condition of $v^0$ with a sum of weighted Dirac masses, $i.e.$ choosing $N\in \mathbb{N}$, $x_1^0, ..., x_N^0 \in \R^d$, $\alpha_1^0,..., \alpha_N^0\in \R$ such that 
    \[v^N_0:=\sum\limits_{i=1}^N{\alpha_i^0 \delta_{x_i^0} }\sim v^0, \]
    in the sense of Radon measures, which means that, for any $\phi \in \mathcal{C}^0_c(\R^d)$,
    \[\sum\limits_{i=1}^N{\alpha_i^0 \phi(x_i^0)}\underset{N\to +\infty}{\longrightarrow} \int_{\R^d}{\phi(x)v^0(x)dx}.\]
    Assuming that $v^0$ has a compact support, a canonical way of choosing these values is to choose a finite collection of subsets $\Omega^0_i\subset \mathrm{supp}\left( v^0 \right)$ satisfying
\begin{equation*}
    \Omega^0_i\cap\Omega^0_j=\emptyset,\mbox{ if }i\neq j,\mbox{ and }\bigcup\limits_{i\in \{1,..., N\}}\Omega^0_i=\mathrm{supp}\left( v^0 \right), 
\end{equation*}
and to take, for any $i\in \{1,..., N\}$
\[x_i^0\in \Omega^0_i,\quad w_i^0=\lvert  \Omega^0_i \rvert, \quad \nu_i^0=v^0(x_i^0)\quad \alpha_i=\nu_i w_i.\]

\item \textbf{Time-evolution of the particles.} By using a weak formulation of the PDE, we determine the ODE satisfied by the positions (denoted $x_i$), the volumes (denoted $w_i$) and the weights (denoted $\nu_i$) associated to each of the $N$ particles, with initial conditions ($x_i^0$, $w_i^0$, $\nu_i^0$) given at the previous step. The exact ODE, and the way the particles are correlated with each other depends on the complexity of the PDE. In the case where the advection term is local ($a(t,x,I)=a(t,x)$), the positions and the volumes satisfy
\begin{equation}
\dot x_i(t)=a(t,x_i(t)), \quad \dot w_i(t)=\nabla_x \cdot(a(t,x_i(t)))w_i(t)
\label{dot x dot w}
\end{equation}
and the formula for the $\nu_i$ depends on the selection and the mutation terms. The method described in the core of this article also allows to use this method in the case of a non-local advection, which modifies the ODE satisfied by the positions of the particles. Full formulas are given in Section \ref{TPM}.

Rewriting $\alpha_i=\nu_i w_i$, with the volumes $w_i$ which satisfy \eqref{dot x dot w} is required for the approximation of the different integral terms. Indeed, by Liouville's formula \cite{hartman2002ordinary}, for any $f$ smooth enough, 
\[\int_{\R^d}{f(x)dx}\sim \sum\limits_{i=1}^N{f(x_i(t))w_i(t)}. \]

Formally, $v$ and $v^N$ (as defined by \eqref{vN intro}) are both solution of \eqref{eq intro} in the weak sense, with $v^N(0)\sim v^0$, which implies that, assuming that the parameters of the PDE are smooth enough, for any given time $T>0$, $v^N(T)\sim v(T, \cdot)$. 

\item \textbf{Regularisation.} In order to transform the discrete measure $v^N$ into a smooth function, we use a regularisation process based on convolution, which writes as a sum, shown in \eqref{vNeps intro}, since the convoluted measure is a sum of Dirac masses. The function 
$$\varphi_\varepsilon:=\frac{1}{\varepsilon}\varphi\left(\frac{\cdot}{\varepsilon}\right),$$
depending on a parameter $\varepsilon>0$, is a scaling of the so-called \textit{cut-off function} $\varphi$, which must satisfy some regularity and symmetry properties (which we specify in Section \ref{ConvFT}). The choice of $\varepsilon$, which intimately depends on the choice of $N$, is intricate: if $\e$ is too large, then the solution is `over-regularised', and the scheme loses its accuracy. Conversely, if $\varepsilon$ is to small, then some of the particles will be neglected, and the scheme does not converge towards the solution. Choosing the optimal $\e$ as a function of $N$ is thus not a trivial question, and it is possible in some cases to optimise the convergence rate by improving this regularisation step \cite{cohen2000optimal, hou1990convergence}.
\end{enumerate}

\subsection*{Main results}
\paragraph{Well-posedness.} We first prove that problem \eqref{eq intro} is well-posed, \textit{i.e.} that for any family of parameters satisfying some regularity properties, defined in  $\mathcal{C}([0,+\infty),L^1(\mathbb{R}^d))$. The proof heavily relies on the use of the characteristic curves $X_u(t,y)$, solution to the equation
\begin{align*}
\begin{cases}
\dot X_u(t,y)=a(t,X_u(t,y),(I_au)(t,X_u(t,y)))=:\mathcal{A}_u(t,X_u(t,y)) , \quad t\in [0,T],\\
X_u(0,y)=y, 
\end{cases}
\end{align*}
for $y\in\R^d$ and $u\in \mathcal{C}([0,T],L^1(\R^d))$. The main difference with respect to the approach taken for similar problems, as in \cite{degond1989weighted}, is the need for continuity results for $X_u(t,y)$, not only with respect to the trait variable $y$, but with respect to $u$ as well. The required results are stated in Section \ref{CharRes} and proved in Appendix \ref{AnnA}.\\  

The well posedness of the problem for smooth initial data is proved in Section \ref{SmoothData} using a standard fixed point argument. Moreover, in Section \ref{GenIniCond}, we consider a more general family of initial conditions and we prove that the regularity of $v(t, \cdot)$ is linked to that of $v^0$: more precisely, if $v^0\in W^{k, \infty}(\R^d)$ with compact support,  then $v\in \mathcal{C}([0,+\infty), W^{k-1,1}(\R^d))$. This result can be improved if the advection is local, \textit{i.e.} $a(t,x,I)=a(t,x)$: In this case, for any initial condition in $W^{k, 1}(\R^d)$, the solution $v$ is in $\mathcal{C}([0,+\infty), W^{k,1}(\R^d))$.\\

\paragraph{Particle method: definition and well-posedness.} In Section \ref{TPM}, we define the particle method corresponding to this PDE, by deriving the ODE system satisfied by the particles ($x_i, w_i$ and $\nu_i$). For the non-local case, the equation we obtain is a coupled system with infinitely many equations and unknowns. We generalise some classical results from the Cauchy-Lipschitz theory in order to deal with this problem, and prove that the ODE system is well posed.\\

\paragraph{Convergence of the particle method.} Section \ref{Cvgn} is structured as follows: in Section \ref{ConvFT} we prove the following estimate, detailed in Theorem \ref{Theo6}:
\[\lVert v-v_\e^h \rVert_{L^1(\R^d)}\leqslant C \bigl(\e^r+\big(\frac{h}{\e}\big)^{\kappa}+h^{\kappa}\bigr)\lVert v^0 \rVert_{W^{\mu,1}(\R^d)},\quad \mbox{ for all } \; 0\leqslant t\leqslant T,\]
with $h=\frac{1}{N}$, $C>0$, $r\in \mathbb{N}^*$ (which depends on the regularity of $\varphi$), $\kappa\in \mathbb{N}^*$ (which depends on $k$ and the local or non-local nature of $a$) and $\mu\in \mathbb{N}^* $ (which depends on $\kappa$ and $r$). This result ensures the convergence of the particle method in finite time, upon choosing $\varepsilon$ as an suitable function of $h$ which allows to obtain a convergence in $(\e(h)^r+\big(\frac{h}{\e(h)}\big)^{\kappa}+h^{\kappa})$ which yields a convergence rate of $h^{\frac{\kappa r}{\kappa+r}}$ after choosing the optima value of $\varepsilon(h)\sim h^{\frac{\kappa }{\kappa+r}}$.  Nevertheless, in general this scheme is not asymptotic preserving. Therefore, in subsection \ref{conv infinite time}, we show examples for which the scheme is asymptotic preserving, and others for which it is not. In general, the asymptotic behaviour of a solution is preserved when it converges to a sum of Dirac masses, and is not when it converges to a smooth solution. 

\subsection*{Perspectives and open problems}
Although a loss of regularity appears to be taking place when advection is non local, we do not know whether such a loss of regularity does happen in certain cases. The construction of such an example or, on the contrary, the improvement of our results in order to prove that, in fact, no regularity is lost could be a first way to extend our work.\\
Another open problem is the optimisation of the order of convergence for the numerical solution, improving upon the order $\frac{\kappa r}{\kappa +r}$ obtained in the present work. The approach from \cite{cohen2000optimal}, where the local averages are viewed as point values of an approximation of the solution, and the regularisation of the solution at time $t >0$ is performed by interpolation rather than convolution, could be a suitable choice.\\
Lastly, as mentioned before, another direction could be the extension of our results in order to deal with second order equations, as done in \cite{degond1989weighted}, where the Laplacian operator is approximated by an appropriate sequence of mutation kernels.

\section{The problem}\label{PDE}
For $T>0$ and $k\in\mathbb{N}$ we consider the functions
\begin{align}
    (t,x,I)\mapsto a(t,x,I)&\in W^{1,\infty}\left([0,T],(W^{k+1,\infty}(\mathbb{R}^{d+1}))^d\right),\label{aetR1}\\
    (t,x,I)\mapsto R(t,x,I)&\in \mathcal{C}\left([0,T]\times \R^d_x, W^{k+1,\infty}_{loc}(\R_I) \right) \cap \mathcal{C}\left([0,T]\times \R_I, W^{k+1,\infty}(\R^d_x)\right).\label{aetR2}
\end{align}
We consider as well $(t,x,y,I)\mapsto m(t,x,y,I)$ such that 
\[
0\leqslant m\in\mathcal{C}\left([0,T]\times \R_x^d \times \R_y^d, W^{k+1,\infty}(\R_I) \right) \cap \mathcal{C}\left([0, T]\times \R_x^d \times \R_I, L^{\infty}(\R^d_y) \right) \cap \mathcal{C}\left([0, T] \times \R_y^d\times \R_I, \mathcal{C}^{k}_c(\R^d_x)\right),
\]
which is globally Lipschitz with respect to the non local variables and uniformly compactly supported with respect to the $x$ variable. That is, we suppose $m$ to satisfy the following hypotheses:
\begin{itemize}
    \item There exists $\mu>0$ such that
    \begin{equation}
        \sup\limits_{t,x,y}\sum\limits_{i=1}^k\sum\limits_{|\alpha|= i}\sum\limits_{j=1}^k|\partial^{\alpha}_x\partial^j_Im(t,x,y,I)-\partial^{\alpha}_x\partial^j_Im(t,x,y,J)|\leqslant \mu|I-J|.\label{mLip}
    \end{equation}
    \item There exists a compact set $\mathcal{K}$ such that the function
    \begin{equation}
        M(x):=\sup\limits_{t,y,I}\sum\limits_{i=1}^k\sum\limits_{|\alpha|= i}\sum\limits_{j=1}^k|\partial^{\alpha}_x\partial^j_Im(t,x,y,I)|,\label{bigMdef}
    \end{equation}
    satisfies
    \begin{equation}
        \mbox{supp } M(x)\subset \mathcal{K}.\label{suppM}
    \end{equation}
    Furthermore, we assume that
    \begin{align}
        &\|M\|_{L^{\infty}(\R^d)}\leqslant \overline{M}<\infty\label{bigMM}.
    \end{align}
    From a modelling point of view, assuming $m$ to be uniformly compactly supported reflects the fact that some traits are realistically out of reach for a given population, and that, in general, mutations are rare and small.
\end{itemize}
We remark that hypotheses \eqref{suppM} and \eqref{bigMM} imply that
\begin{equation}
    \|M\|_{L^{1}(\R^d)}\leqslant |\mathcal{K}|\overline{M}=:K<\infty.\label{bigMK}
\end{equation}
For all functions $u\in \mathcal{C}([0,T],L^1(\mathbb{R}^d))$ we consider the linear mappings $I_a$, $I_g$ and $I_d$ which satisfy, for all $t\in [0, T]$, $x\in \R^d$,
\begin{align}
    (I_au)(t,x)&:=\int\limits_{\mathbb{R}^d}\psi_a(t,x,y)u(t,y)dy,\label{Ia}\\
    (I_gu)(t,x)&:=\int\limits_{\mathbb{R}^d}\psi_g(t,x,y)u(t,y)dy,\label{Ig}\\
     (I_du)(t,x)&:=\int\limits_{\mathbb{R}^d}\psi_d(t,x,y)u(t,y)dy,\label{Id}
\end{align}
where
\begin{align}
    \psi_a&\in W^{1,\infty}\left([0,T]\times \R^d_x, L^\infty(\R^d_y)\right) \cap \mathcal{C}\left([0,T]\times \R^d_y, W^{k+1,\infty}(\R^d_x)\right) ,\label{psia}\\
   0<\underline{\psi_g}\leqslant \psi_g&\in\mathcal{C}\left([0,T]\times \R^d_x, L^\infty(\R^d_y)\right) \cap \mathcal{C}\left([0,T]\times \R^d_y,  W^{k+1,\infty}(\R^d)\right),\label{psig}\\
    \psi_d&\in\mathcal{C}\left([0,T]\times \R^d_x, L^\infty(\R^d_y)\right) \cap \mathcal{C}\left([0,T]\times \R^d_y,  W^{k+1,\infty}(\R^d)\right),\label{psid}
\end{align}
for a certain $\underline{\psi_g}>0$. We remark that $I_au,I_gu,I_du\in\mathcal{C}([0,T],L^\infty(\mathbb{R}^d))$. The functions $\psi_a$ and $\psi_d$ do not need to be positive, reflecting this way how different traits have different impacts (which are not always beneficial) on the environment, and ultimately on the population itself. On the other hand, $\psi_g$ has to be bounded away from zero. This hypothesis reflects the fact that, at least for the growth term, all interactions between individuals are of the same type. These interactions may be interpreted as either strictly competitive or strictly cooperative. In particular, this means that only ``very'' non-local dependence with respect to u are allowed. This excludes partial densities, for instance on part of the traits.\\
Lastly, we assume that there exist non-negative constants $I^*$ and $r^*$ such that, for all $t\in [0, T]$, $x\in \R^d$, and $I\geqslant I^*$, 
\begin{equation}
    R(t,x,I)+K<-r^*\label{GvsM},
\end{equation}
uniformly on $t$ and $x$. It is somewhat natural to assume that $R$ is negative for a large population: for example, if a carrying capacity is assumed to exist, and the population size is approaching such value, then the growth rate will inevitably drop to levels where no amount of mutations will be able to compensate for it.  \\
For a given function $v^0\in L^1(\mathbb{R}^d)$, we will study the existence and uniqueness of solution for the problem
\begin{align}
\begin{cases}
\partial_t v(t,x)+\nabla_x \cdot (a(t,x,I_av(t,x))v(t,x))=R(t,x,I_gv(t,x))v(t,x)+\displaystyle{\int\limits_{\R^d}}{m(t,x,y,I_dv(t,x))v(t,y)dy},\\
v\in \mathcal{C}([0,T),L^1(\mathbb{R}^d)),\\
v(0,\cdot)=v^0(\cdot).
\end{cases}\label{Pv}
\end{align}
We will show that, under additional hypotheses, either over $a$ or $v^0$, we can guarantee the well-posedness of this problem. In particular, we provide the results regarding the cases of local advection ($\partial_Ia=0$) and non-local advection ($\partial_Ia\neq 0$). We will see that this distinction directly affects the set of initial data $v^0$ for which the existence of solutions is guaranteed.
\subsection{Some bounds over the characteristics}\label{CharRes}
Consider $a$ satisfying \eqref{aetR1} and $\psi_a$ satisfying \eqref{psia} for some $k\geqslant 1$.
For all $y\in\R^d$ and $u\in \mathcal{C}([0,T],L^1(\R^d))$ we define the characteristic curve $ t\mapsto X_u(t,y)$ as the unique solution the following ODE
\begin{align}
\begin{cases}
\dot X_u(t,y)=a(t,X_u(t,y),(I_au)(t,X_u(t,y)))=:\mathcal{A}_u(t,X_u(t,y)) , \quad t\in [0,T],\\
X_u(0,y)=y,
\end{cases}\label{CharL}
\end{align}
where $(I_au)(t,x)$ is defined in \eqref{Ia}. Since the function $\psi_a$ belongs to $W^{1,\infty}\left([0,T]\times \R^d_x, L^\infty(\R^d_y)\right)$, then $(I_au)(t,x)$ belongs to $W^{1,\infty}\left([0,T]\times \R^d_x \right)\subset \mathcal{C}^{0,1}([0,T]\times \R^d_x)$. The regularity of $a$ then implies that $\mathcal{A}_u(t,x)$ is a Lipschitz function with respect to the $x$ variable, uniformly with respect to $t$, guaranteeing this way the global existence of solution for \eqref{CharL}.

For all $u\in \mathcal{C}([0,T],L^1(\R^d))$, we define the norms
\[
\|u\|_1:=\|u\|_{L^1([0,T]\times \R^d)}=\int\limits_0^T\int\limits_{\R^d}|u(t,x)|dxdt\mbox{ and }
\|u\|:=\sup\limits_{t\in[0,T]}\|u(t,\cdot)\|_{L^1(\R^d)}.
\]
We present some results involving the characteristics. The proofs for such results are given in Appendix \ref{AnnA}.\\

The first property we describe is the continuity of the family of characteristics with respect to the spatial variables and the function $u$.
\begin{lem}\label{ChardX2}
Let $\psi_a$ satisfy \eqref{psia} and $a$ satisfy \eqref{aetR1} for $k= 0$. Consider $y_1,y_2\in\R^d$ and $u_1,u_2\in \mathcal{C}([0,T],L^1(\R^d))$. Then there exists a positive constant $C(T,\|u_1\|,\|u_2\|)$, such that the solutions $X_{u_1}$ and $X_{u_2}$ of \eqref{CharL} satisfy for any $t\in [0,T]$,
\[
 \sum\limits_{j=1}^{d}|X^j_{u_1}(t,y_1)-X^j_{u_2}(t,y_2)|\leqslant  C(T,\|u_1\|,\|u_2\|)\left(|y_1-y_2|+\|u_1-u_2\|_1\right).
\]
\end{lem}
Secondly, we claim that the spatial derivatives of the characteristics remain bounded by a constant only depending on $T$ and $\|u\|$. We also claim that the spatial derivatives are continuous with respect to the spatial variables and the function $u$.
\begin{lem}\label{ChardX}
Let $a$ satisfy \eqref{aetR1} and $\psi_a$ satisfy \eqref{psia} for some $k\geqslant 1$. Consider $u\in \mathcal{C}([0,T],L^1(\R^d))$. Then there exists a positive constant $C(T,\|u\|)$, such that the solution $X_u(t,y)$ of \eqref{CharL} satisfies, for all $t\in [0,T]$ and $y\in \R^d$, 
\begin{equation}
     \sum\limits_{i=1}^{k}\sum\limits_{|\alpha|\leqslant i}\sum\limits_{j=1}^{d}|\partial^{\alpha}_{y}X^{j}_u(t,y)|\leqslant C(T,\|u\|)\label{bpX} .
\end{equation}
 Furthermore, for any two points $y_1,y_2\in\R^d$, and any two functions $u_1,u_2\in \mathcal{C}([0,T],L^1(\R^d))$, there exists a positive constant $C_2(T,\|u_1\|,\|u_2\|)$, such that the solutions $X_{u_1}$ and $X_{u_2}$ of \eqref{CharL} satisfy for any $t\in [0,T]$
\begin{equation}
    \sum\limits_{i=0}^{k}\sum\limits_{|\alpha|\leqslant i}\sum\limits_{j=1}^{d}|\partial^{\alpha}_{y}X^j_{u_1}(t,y_1)-\partial^{\alpha}_{y}X^j_{u_2}(t,y_2)|\leqslant C(T,\|u_1\|,\|u_2\|)(|y_1-y_2|+\|u_1-u_2\|_1).\label{ContpX}
\end{equation}

\end{lem}
Since for all $t\in [0,T]$, $y\mapsto X_u(t,y)$ is a $\mathcal{C}^1$-diffeomorphism from $\R^d$ onto itself, we may define its inverse as the function satisfying $X_u(t,X^{-1}_u(t,x))=x$ for all $(t,x)\in [0,T]\times \R^d$. We have the following results for $X^{-1}_u$.
\begin{lem}\label{ChardXinv}
Let $a$ satisfy \eqref{aetR1} and $\psi_a$ satisfy \eqref{psia} for some $k\geqslant 1$. Consider $u\in \mathcal{C}([0,T],L^1(\R^d))$. Then there exists a positive constant $\widetilde{C}(T,\|u\|)$, such that the inverse of the solution $X_u(t,y)$ of \eqref{CharL} satisfies, for all $t\in [0,T]$, $x\in \R^d$
\begin{equation}
    \sum\limits_{i=1}^{k}\sum\limits_{|\alpha|\leqslant i}\sum\limits_{j=1}^{d}|\partial^{\alpha}_{x}\left(X^{-1}_{u}\right)^j(t,x)|\leqslant \widetilde{C}(T,\|u\|). \label{bpXinv}
\end{equation}

\end{lem}
\begin{lem}\label{ChardXinv2}
Let $a$ satisfy \eqref{aetR1} and  $\psi_a$ satisfy \eqref{psia} for some $k\geqslant 1$. Consider any two functions $u_1,u_2\in \mathcal{C}([0,T],L^1(\R^d))$. Then there exists a positive constant $\widetilde{C}(T,\|u_1\|,\|u_2\|)$, which satisfies $\lim\limits_{T\rightarrow 0}\widetilde{C}(T,\|u_1\|,\|u_2\|)=0$ and such that the inverses of the solutions $X_{u_1}$ and $X_{u_2}$ of \eqref{CharL} satisfy, for all $t\in [0,T]$ and $x\in \R^d$
\begin{equation}
    \sum\limits_{i=0}^{k-1}\sum\limits_{|\alpha|\leqslant i}\sum\limits_{j=1}^{d}|\left(\partial^{\alpha}_{x}X^{-1}_{u_1}\right)^j(t,x)-\partial^{\alpha}_{x}\left(X^{-1}_{u_2}\right)^j(t,x)|\leqslant \widetilde{C}(T,\|u_1\|,\|u_2\|)\|u_1-u_2\|_1.\label{ContXinv}
\end{equation}

\end{lem}
We remark that thanks to the relation $\|u\|_1\leqslant T\|u\|$, the relations \eqref{ContpX} and \eqref{ContXinv} also hold true when replacing $\|u_1-u_2\|_1$ by $\|u_1-u_2\|$.  
Lastly, we give a result regarding the regularity of $X_u(t,x)$ with respect to $t$.
\begin{lem}\label{CompClos}
    Let $a$ satisfy \eqref{aetR1} and  $\psi_a$ satisfy \eqref{psia} for some $k\geqslant 1$. Consider  $u\in \mathcal{C}^1([0,T_1)\times \R^d)$ such that $\sup\limits_{t\in [0,T_1)}(\|u(t,\cdot)\|_{L^1(R^d)}+\|\partial_tu(t,\cdot)\|_{L^1(R^d)})<+\infty$. Then $X_u(t,y)\in \mathcal{C}^1([0,T_1],\mathcal{C}^{k}(\R^d))$. As a consequence, $X^{-1}_u(t,x)\in \mathcal{C}^1([0,T_1],\mathcal{C}^{k}(\R^d))$.
\end{lem}
\begin{proof}
    Thanks to Lemma \ref{ChardX}, we know that under these hypotheses, for all $T<T_1$, $X_u(t,y)$ exists and belongs to $\mathcal{C}^1([0,T],\mathcal{C}^{k}(\R^d))$. Consider $0<t_1,t_2<T_1$, then
    \begin{align*}
        \sum\limits_{i=0}^{k}\sum\limits_{|\alpha|\leqslant i}\sum\limits_{j=1}^{d}|\partial^{\alpha}_yX^j_u(t_1,y)-\partial^{\alpha}_yX^j_u(t_2,y)|&=\sum\limits_{i=0}^{k}\sum\limits_{|\alpha|\leqslant i}\sum\limits_{j=1}^{d}|\int_{t_1}^{t_2}\partial^{\alpha}_y\dot{X}^j_u(s,y)ds|\\
        &=\sum\limits_{i=0}^{k}\sum\limits_{|\alpha|\leqslant i}\sum\limits_{j=1}^{d}|\int_{t_1}^{t_2}\partial^{\alpha}_ya_j(s,X_u(s,y),(I_au)(s,X_u(s,y)))ds|.
    \end{align*}
    Thanks to the regularity of $a$ and $\psi_a$, the bounds given in Lemma \ref{ChardX} for the derivatives of $X_u(t,y)$ and the uniform bound for $\|u\|_{L^1(\R^d)}$ we conclude that there exists a positive constant such that
\[
\sum\limits_{i=0}^{k}\sum\limits_{|\alpha|\leqslant i}\sum\limits_{j=1}^{d}|\partial^{\alpha}_yX^j_u(t_1,y)-\partial^{\alpha}_yX^j_u(t_2,y)|\leqslant C|t_1-t_2|.
\]
Similarly, we have
\begin{align*}
        \sum\limits_{i=0}^{k}\sum\limits_{|\alpha|\leqslant i}\sum\limits_{j=1}^{d}|\partial^{\alpha}_y\dot{X}^j_u(t_1,y)-\partial^{\alpha}_y\dot{X}^j_u(t_2,y)|&=\sum\limits_{i=0}^{k}\sum\limits_{|\alpha|\leqslant i}\sum\limits_{j=1}^{d}|\partial^{\alpha}_y(\mathcal{A}_u)_j(t_1,y)-\partial^{\alpha}_y(\mathcal{A}_u)_j(t_2,y)|,
    \end{align*}
    where 
\[
(\mathcal{A}_u)_j(t,y):=a_j(t,X_u(t,y),(I_au)(t_1,X_u(t,y))).
\]

Again, the regularity up to order $k+1$ of the involved coefficients allow us to conclude that
\[
\sum\limits_{i=0}^{k}\sum\limits_{|\alpha|\leqslant i}\sum\limits_{j=1}^{d}|\partial^{\alpha}_y\dot{X}^j_u(t_1,y)-\dot{X}^j_u(t_1,y)|\leqslant  C|t_1-t_2|.
\]
We have shown that $X_u(t,\cdot)$ is a Cauchy sequence in $\mathcal{C}^{k}(\R^d)$ when $t$ goes to $T_M$, therefore, there exists $X^*(x)\in\mathcal{C}^{k}(\R^d)$ and $Y^*(x)\in\mathcal{C}^{k}(\R^d)$ such that
\[
\lim\limits_{t\rightarrow T_M}\|X_u(t,\cdot)-X^*\|_{\mathcal{C}^{k}(\R^d)}+\|\dot{X}_u(t,\cdot)-Y^*\|_{\mathcal{C}^{k}(\R^d)},
\]
which is the desired result.
\end{proof}

\subsection{Existence of solution for smooth initial data}\label{SmoothData}

We first provide the proof of existence and uniqueness of solution for problem \eqref{Pv} when the initial condition $v^0$ is a smooth enough function. We still assume hypotheses \eqref{aetR1} through \eqref{GvsM} hold. For a smooth initial condition $v^0$ we denote by \emph{classical solution} any function $v\in\mathcal{C}^1([0,T]\times\R^d)$ which satisfies problem \eqref{Pv}.\\
The following \emph{a priori} estimate will allow us to guarantee the global existence of a classical solution given that such solution exists over a certain interval $[0,T_1]$.
\begin{lem}\label{Priori}
    Let $v^0\in\mathcal{C}^1_c(\R^d)$ and $T_1>0$ be such that a classical solution $v\in \mathcal{C}^1([0,T]\times\R^d)$ exists for problem \eqref{Pv} for all $T<T_1$, which is positive and has compact support with respect to the $x$ variable. Then, such solution satisfies the estimate
    \begin{equation}
        \sup\limits_{t\in[0,T_1]}\|v(t, \cdot)\|_{L^1(\R^d)}\leqslant \max\{ \|v^0\|_{L^1(\R^d)},\frac{I^*}{\underline{\psi_g}}\}.\label{BL1}
    \end{equation}
\end{lem}
\begin{proof}
    Let $v$ be the aforementioned positive solution. Denoting $\rho(t):=\|v(t, \cdot)\|_{L^1(\R^d)}$, we see directly from equation \eqref{Pv} that
    \[
\dot{\rho}(t)=\int\limits_{\R^d}\left(R(t,y,(I_gv)(t,y))+\int_{\R^d}m(t,x,y,(I_dv)(t,x))dx\right)v(t,y)dy.
\]
If there exists $t$ such that $\rho(t)>\frac{I^*}{\underline{\psi_g}}$, then $(I_gv)(t,y)\geqslant \underline{\psi_g}\rho(t)> I^*$, which allows us to use hypothesis \eqref{GvsM} in order to conclude
\[
\dot{\rho}(t)<-r^*\rho(t)<0.
\]
This way, we see that either $\rho(t)$ is smaller than $\frac{I^*}{\underline{\psi_g}}$ or $\rho(t)$ is decreasing, which in turn implies the bound \eqref{BL1}.
\end{proof}
A fixed point argument together with estimate \eqref{BL1} will allow us to conclude the existence of solution for problem \eqref{Pv} for smooth initial conditions.
\begin{theorem}\label{T1}
Consider $k\geqslant 1$ and $T>0$. Under hypotheses \eqref{psia} through \eqref{GvsM}, for all non-negative functions $v^0\in \mathcal{C}^k_c(\R^d)$, there exists a unique non-negative classical solution $v\in \mathcal{C}^1([0,T],\mathcal{C}^{k}_c(\R^d))$ to problem \eqref{Pv}. Furthermore, such solution satisfies
\begin{align}
    \sup\limits_{t\in[0,T]}\|v(t, \cdot)\|_{L^1(\R^d)}\leqslant \max\{ \|v^0\|_{L^1(\R^d)},\frac{I^*}{\underline{\psi_g}}\},\label{BL12}
    \\
   \sup\limits_{t\in[0,T]}\|v(t,\cdot)\|_{W^{k,1}(\R^d)}\leqslant C_T\|v^0\|_{W^{k,1}(\R^d)}.\label{BWm1}
\end{align}
\end{theorem}
\begin{proof}
Consider $v^0\in \mathcal{C}^k_c(\R^d)$. For $t\geq 0$ and the function $M$ introduced in \eqref{bigMdef} we define $r_t:=2\|a\|_{L^{\infty}}t$, $B_{r_t}$ the open ball centred at $0$ and of radius $r_t$,
\[
O_t:=\mbox{supp }(v^0)\cup\mbox{supp }(M)+B_{r_t},
\]
and for $\alpha>1$ we define $\overline{\rho}_{\alpha}:=\max\{ \alpha\|v^0\|_{L^1(\R^d)},\frac{I^*}{\underline{\psi_g}}\}$. We consider
\[
u\in D^T_{\alpha}:=\left\{u\in \mathcal{C}([0,T]\times \R^d):u\geqslant 0,\mbox{supp }(u(t,\cdot))\subset O_t,\int_{\R^d}u(t,x)dx\leqslant \overline{\rho}_{\alpha},\ \forall t\in[0,T]\right\}.
\]
We denote as $\Phi u$ the mapping defined by $v=\Phi u$, where $v$ is the solution of
\begin{align*}
\begin{cases}
\partial_t v(t,x)+\nabla_x \cdot (a(t,x,(I_au)(t,x))v(t,x))-R(t,x,(I_gu)(t,x))v(t,x)=\int_{\R^d}{m(t,x,y,(I_du)(t,x))u(t,y)dy}\\
v(0,\cdot)=v^0(\cdot).
\end{cases}
\end{align*}
Let us denote, for any $y\in \R^d$, as $X_u(\cdot, y)$ the unique solution of 
\begin{align*}
\begin{cases}
\dot X_u(t,y)=a(t,X_u(t,y),(I_au)(t,X_u(t,y)))=:\mathcal{A}_u(t,X_u(t,y)) \quad t\geq 0,\\
X(0,y)=y.
\end{cases}
\end{align*}
As stated before, for any $t\in [0,T]$,  $x\mapsto X_u(t,x)$ is a $\mathcal{C}^1$-diffeomorphism, therefore, for all $t\geq 0$ and all $x\in\R^d$ there exists a unique $y\in\R^d$ such that $x=X_u(t,y)$, which we denote $y=X^{-1}_u(t,x)$.\\
We see that, 
\begin{align*}
\frac{d}{dt}v(t,X_u(t,y))=&\big[ R(t, X_u(t,y), (I_gu)(t, X_u(t,y)))-\divv \mathcal{A}_u(t,X_u(t,y))\big] v(t,X_u(t,y))\\ 
&+\int_{\R^d}{m(t, X_u(t,y), z, (I_du)(t, X_u(t,y)))u(t, z)dz},
\end{align*}
and thus, denoting
\[
\mathcal{G}_u(t,y):=R(t, X_u(t,y), (I_gu)(t, X_u(t,y)))-\divv \mathcal{A}_u(t,X_u(t,y)),
\]
we get
\begin{align*}
\Phi u (t, X_u(t,y))=& v(t,X_u(t,y))\\
=&v^0(y)\exp\biggl(\int_0^t\mathcal{G}_u(s,y)ds\biggr)\\
&+\int_0^t\int_{\R^d}m(s, X_u(s,y), z, (I_du)(s, X_u(s,y)))u(s, z)dz\exp\left(\int_s^t\mathcal{G}_u(\tau,y)d\tau\right)ds
\end{align*}
or, equivalently
\begin{align*}
\Phi u (t, x)=& v(t,x)\\
=&v^0(X^{-1}_u(t,x))\exp\biggl(\int_0^t\mathcal{G}_u(s,X^{-1}_u(t,x))ds\biggr)\\
&+\int_0^t\int_{\R^d}m(s, X_u(s,X^{-1}_u(t,x)), z, (I_du)(s, X_u(s,X^{-1}_u(t,x))))u(s, z)dz\exp\left(\int_s^t\mathcal{G}_u(\tau,X^{-1}_u(t,x))d\tau\right)ds.
\end{align*}
The solution $v$ is thus non-negative, according to the non-negativity of $v^0$, $u$ and $m$. Thanks to Lemma \ref{ChardXinv}, $v$ belongs to $\mathcal{C}^1([0,T],\mathcal{C}^{\kappa}(\R^d))\subset \mathcal{C}([0,T]\times \R^d)$. Furthermore, the fact that for all $t\in [0,T]$, $|X(t,y)-y|\leqslant \|a\|_{L^{\infty}}t$ implies that $\mbox{supp } (v(t,\cdot))\subset O_t$.\\
Additionally, directly from its definition, we see that
\[
\mathcal{G}_u(s,X^{-1}_u(t,x))\leqslant \gamma:=\|R\|_{L^{\infty}_{t,x,I}}+\|a\|_{W^{1,\infty}_xL^{\infty}_{t,I}}+\|a\|_{W^{1,\infty}_IL^{\infty}_{t,x}}\|\psi_a\|_{W^{1,\infty}_xL^{\infty}_y}\overline{\rho}_{\alpha},
\]
and consequently, for all $u\in D^T_{\alpha}$, we have
\begin{align*}
    \|\Phi u(t,\cdot))\|_{L^1(\R^d)}\leqslant& e^{\gamma T}\left(\int\limits_{\R^d}v^0(X^{-1}_u(t,x))dx\right.\\
    &\left.+\int_0^t\int_{\R^d}\int_{\R^d}m(s, X_u(s,X^{-1}_u(t,x)), z, (I_du)(s, X_u(s,X^{-1}_u(t,x))))u(s, z)dzdxds\right)\\
    \leqslant& e^{\gamma T}\left(\int\limits_{\R^d}v^0(X^{-1}_u(t,x))dx+\int_0^t\int_{\R^d}M(s, X_u(s,X^{-1}_u(t,x)))dx\int_{\R^d}u(s, z)dzds\right).
\end{align*}
Making the changes of variables $y=X^{-1}_u(t,x)$ and $y=X_u(s,X^{-1}_u(t,x))$ respectively on each of the integrals on the last expression, recalling that, according to Liouville's formula
\begin{align*}
    |J_{X_u(t,y)}|&=e^{\int_0^t\divv \mathcal{A}_u(s,X_u(t,y))ds},\\
    |J_{X_u(t,X^{-1}_u(s,y))}|&=e^{\int_0^t\divv \mathcal{A}_u(\tau,X_u(t,X^{-1}_u(s,y)))d\tau-\int_0^s\divv \mathcal{A}_u(\tau,y)d\tau},
\end{align*} 
and using the hypotheses over $a$ and $m$ we obtain that for all $t\in [0,T]$, 
\begin{align*}
    \|\Phi u(t, \cdot)\|_{L^1(\R^d)}
    &\leqslant e^{(\gamma+2\tilde{a}) T}\left(\|v^0\|_{L^1(\R^d)}+KT\|u\|_{L^1(\R^d)}\right),
\end{align*}
where $\tilde{a}:=\|a\|_{W^{1,\infty}_xL^{\infty}_{t,I}}+\|a\|_{W^{1,\infty}_IL^{\infty}_{t,x}}\|\psi_a\|_{W^{1,\infty}_xL^{\infty}_y}\overline{\rho}_{\alpha}$. Finally, using the hypothesis over $v^0$ and $u$ we see that
\[
\|\Phi u(t, \cdot )\|_{L^1(\R^d)}\leqslant e^{(\gamma+2\tilde{a}) T}\left(\frac{1}{\alpha} +KT\right)\overline{\rho}_{\alpha}.
\]
Thanks to the condition $\alpha>1$ there exists $T_{\alpha}$ (only depending on $\alpha$ and on the coefficients of the problem) such that $\|\Phi u\|_{L^1(\R^d)}\leqslant \overline{\rho}_{\alpha}$ for all $t\in[0,T_{\alpha}]$. In other words, $\Phi :D^{T_{\alpha}}_{\alpha}\rightarrow D^{T_{\alpha}}_{\alpha}$.\\

We now claim that the mapping $\Phi u$ is a contraction on some $D^{T_1}_{\alpha}$, $0<T_1\leqslant T_{\alpha}$, with respect to the usual norm in $\mathcal{C}([0,T]\times \R^d)$. For two functions $u_1,u_2\in D^{T_{\alpha}}_{\alpha}$ and any $t\in[0,T_{\alpha}]$, thanks to Lemma \ref{ChardX2} and Lemma \ref{ChardXinv2}, we have
\begin{align*}
    |\Phi u_1-\Phi u_2|\leqslant& Ce^{\gamma t}|X^{-1}_{u_1}(t,x)-X^{-1}_{u_2}(t,x)|\\
    \leqslant& Ce^{\gamma t}\|u_1-u_2\|_1\\
    \leqslant& Ce^{\gamma t}|O_t|t\|u_1-u_2\|_{\mathcal{C}([0,t]\times \R^d)}.
\end{align*}
Clearly, for $t=T_1$ small enough, $\Phi u$ is a contraction, and therefore, thanks to the Banach fixed point theorem, there exists a unique $v\in D^{T_1}_{\alpha}$ such that $\Phi v=v$. Such $v$ is a solution of \eqref{Pv} over $[0,T_1]$. Furthermore, directly from the relation $v=\Phi v$ we see that $v\in\mathcal{C}^1([0,T],\mathcal{C}^{\kappa}_c(\R^d))$.\\ 

Let us now assume that there exists $T_M$, a finite maximal time such that a solution exists in $B^{T}_{\alpha}$ for all $T<T_M$, and let $v$ be such solution. Directly from Lemma \ref{Priori}, the solution $v$ satisfies estimate \eqref{BL1} over $[0,T_M)$. Furthermore, thanks to the relation $v=\Phi v$, we are able to show that 
\[\sup\limits_{[0,T_M)}(\|v(t,\cdot)\|_{L^1(\R^d)}+\|\dot{v}(t,\cdot)\|_{L^1(\R^d)})<+\infty.
\]
From Lemma \ref{CompClos} we get then that $X^{-1}_v(t,x)\in\mathcal{C}^1([0,T_M],\mathcal{C}^k(\R^d))$, and by composition of functions, so is $v=\Phi v$. We can then iterate the previous ideas using $v(T_M)\in\mathcal{C}^k(\R^d)$ as a starting point in order to obtain the existence of solution over a certain interval $[T_M,T_M+\delta)$, contradicting this way the maximal character of $T_M$. Hence, there exists a classical solution of \eqref{Pv}
 for all $t>0$.\\

The uniqueness on $\mathcal{C}([0,T],\mathcal{C}^{\kappa}_c(\R^d))$ comes from the fact that every other solution on $D^{\infty}_{\alpha}$ will coincide with $v$ at least over a small interval $(0,t_0)$ and then, by continuity, the same would hold for all $t$.\\

To obtain the $W^{k,1}(\R^d)$ estimates, we differentiate the relation $v=\Phi v$ and notice that for all multi-index $\beta$ such that $|\beta|\leqslant k$
\begin{align*}
    \partial^{\beta}_xv=&\exp\biggl(\int_0^t\mathcal{G}_v(s,X^{-1}_v(t,x))ds\biggr)\sum\limits_{|\gamma|\leqslant|\beta|}\partial^{\gamma}_xv^0(X^{-1}_v(t,x))F^{\gamma}_1(t,x)\\
    &+\int_0^t\int_{\R^d}F_2(s,x,y)v(s,y)dy\exp\left(\int_s^t\mathcal{G}_v(\tau,x)d\tau\right)ds
\end{align*}
where the functions $F^{\gamma}_1$ and $F_2$ are combinations of sums and multiplications of the derivatives up to order $|\beta|$ of $X^{-1}_v$, $R$, $m$, $I_g$ and $I_d$. Taking absolute values, integrating over $\R^d$ and using the boundedness of all the involved coefficients we arrive at
\[
\|v(t,\cdot)\|_{W^{k,1}(\R^d)}\leqslant C^1_T\|v^0\|_{W^{k,1}(\R^d)}+C^2_T\int_0^t\|v(s, \cdot)\|_{W^{k,1}(\R^d)}ds
\]
and thanks to Gr\"onwall's lemma we get \eqref{BWm1}.
\end{proof}
\subsection{Existence of solution for more general initial data}\label{GenIniCond}
Depending on whether $\partial_Ia=0$ or $\partial_Ia\neq 0$, we will have a different class of initial data for which we are able to guarantee existence of solution for problem \eqref{Pv}. Furthermore, the regularity of such solution might also be affected.\\
We first prove that,  when $\partial_Ia\neq 0$, a solution exists (in a sense that will be defined below) for any initial condition $v^0\in  W^{k,\infty}(\R^d)$, with compact support. However, we do not prove that the regularity of the solution is preserved over time, even if we did not manage to highlight  the existence of cases where a loss of regularity is observed. Secondly, we will show that, when $\partial_Ia= 0$, not only the set of initial conditions for which we can claim existence of solution is more general ($v^0\in  W^{k,1}(\R^d)$), but the regularity of such solution is preserved for all $t>0$.\\ 
We introduce the definition of \emph{weak solution} for problem \eqref{Pv}. We say that $v$ is a \emph{weak solution} of problem \eqref{Pv} associated to $v^0\in L^p(\R^d)$ if
\[
v\in L^{\infty}([0,T],L^p(\R^d)),
\]
and it satisfies the equation in the following weak sense
\[
\int_0^T\int_{\R^d}vL^*_v\varphi dxdt=\int_{\R^d}v^0\varphi dx,
\]
for any $\varphi\in \mathcal{C}^1_c([0,T)\times \R^d)$, where we define the operator $L^*_v$ by
\[
L^*_v\varphi(t,x)=-\partial_t\varphi(t,x)-a(t,x,(I_av)(t,x))\cdot\nabla \varphi(t,x)-R(t,x,(I_gv)(t,x))\varphi(t,x)-\int_{\R^d}m(t,y,x,(I_dv)(t,y))\varphi(t,y)dy, 
\]
for all $t\in [0,T]$, $y\in \R^d$. 
We remark that a \emph{classical solution} is always a \emph{weak solution}.

\begin{theorem}\label{T2}
Under hypotheses \eqref{aetR1} through \eqref{GvsM}, for all $k\geqslant 1$  and any non-negative functions $v^0\in W^{k,\infty}(\R^d)$ with compact support, there exists a unique non-negative weak solution $v\in \mathcal{C}([0,T],\mathcal{C}^{k-1}_c(\R^d))$ to problem \eqref{Pv}. Furthermore, such solution satisfies
\begin{align}
    \sup\limits_{t\in[0,T]}\|v(t,\cdot)\|_{L^1(\R^d)}\leqslant \max\{ \|v^0\|_{L^1(\R^d)},\frac{I^*}{\underline{\psi_g}}\},\label{BL1g}\\
    \sup\limits_{t\in[0,T]}\|v(t, \cdot)\|_{W^{k-1,1}(\R^d)}\leqslant C_T\|v^0\|_{W^{k-1,1}(\R^d)},\label{BWm1g}
\end{align}
and, for $k\geqslant 2$, $v\in \mathcal{C}^1([0,T],\mathcal{C}^{k-1}_c(\R^d))$.
\end{theorem}
\begin{proof}
     Directly from Morrey's inequality, we get the relation $v^0\in W^{k,\infty}(\R^d)\subset \mathcal{C}^{k-1,1}(\R^d)$. If $k\geqslant 2$, we are able to apply Theorem \ref{T1} in order to get the desired result.\\
     Consider now $k=1$. The compact support of $v^0$ and the $W^{1,\infty}(\R^d)$ regularity imply that $v^0\in W^{k,1}(\R^d)$. This means that, there exists a sequence of compactly supported functions $v^0_{\varepsilon}\in \mathcal{C}^1_c(\R^d)$ such that 
\begin{gather*}
    \mbox{supp }(v^0)_{\varepsilon}\subset \mbox{ supp }(v^0),\\
    \|v^0_{\varepsilon}\|_{W^{1,\infty}(\R^d)}\leqslant \|v^0\|_{W^{1,\infty}(\R^d)},\\
    \lim\limits_{\varepsilon\rightarrow 0}\|v^0-v^0_{\varepsilon} \|_{W^{1,1}(\R^d)}=0.
\end{gather*}
We denote as $v_{\varepsilon}$ the solution of problem \eqref{Pv} associated to $v^0_{\varepsilon}$, and we claim that $v_{\varepsilon}$ is a Cauchy sequence in $\mathcal{C}([0,T],L^1(\R^d))$.\\
We recall that for all $\varepsilon$,
\[
\underset{t\in [0,T]}{\sup}\;\|v_{\varepsilon}\|_{L^1(\R^d)}\leqslant \overline{\rho}_{\varepsilon}:= \max\{ \|v^0_{\varepsilon}\|_{L^1(\R^d)},\frac{I^*}{\underline{\psi_g}}\}.
\]
Furthermore, the equality $v_{\varepsilon}=\Phi v_{\varepsilon}$ holds true, where $\Phi$ was defined on the proof of Theorem \ref{T1}. Consequently, for $\varepsilon_1,\varepsilon_2>0$ we have the relation
\begin{align*}
    \Delta V_{\varepsilon_1\varepsilon_2}:=&v_{\varepsilon_1}-v_{\varepsilon_2}\\
    =&\Phi v_{\varepsilon_1}-\Phi v_{\varepsilon_2}\\
    =&\biggl( v^0_{\varepsilon_1}(X^{-1}_{v_{\varepsilon_1}}(t,x))-v^0_{\varepsilon_2}(X^{-1}_{v_{\varepsilon_2}}(t,x))\biggr)\exp\biggl(\int_0^t\mathcal{G}_1(\tau,t,x)d\tau\biggr)-v^0_{\varepsilon_2}(X^{-1}_{v_{\varepsilon_2}}(t,x))\biggl(\Delta E(0,t,x)\biggr)\\
    &+\int_0^t\int_{\R^d}\biggl(\mathcal{M}_{1}(s,t,x,z)-\mathcal{M}_{2}(s,t,x,z)\biggr)v_{\varepsilon_1}(s, z)dz\exp\left(\int_s^t\mathcal{G}_1(\tau,t,x)d\tau\right)ds\\
    &+\int_0^t\int_{\R^d}\mathcal{M}_{2}(s,t,x,z)\biggl(\Delta V_{\varepsilon_1\varepsilon_2}(s, z)\biggr)dz\exp\left(\int_s^t\mathcal{G}_1(\tau,t,x)d\tau\right)ds\\
    &+\int_0^t\int_{\R^d}\mathcal{M}_{2}(s,t,x,z)v_{\varepsilon_j}(s, z)dz\biggl(\Delta E(s,t,x)\biggr)ds,
\end{align*}
where
\begin{align*}
\mathcal{G}_i(\tau,t,x)&=\mathcal{G}_{v_{\varepsilon_i}}(\tau,X^{-1}_{v_{\varepsilon_i}}(t,x)),\\
\Delta E(s,t,x)&:=\exp\left(\int_s^t\mathcal{G}_1(\tau,t,x)d\tau\right)-\exp\left(\int_s^t\mathcal{G}_2(\tau,t,x)d\tau\right),\\
\mathcal{M}_{i}(s,t,x,z)&:=m(s, X_{v_{\varepsilon_i}}(s,X^{-1}_{v_{\varepsilon_i}}(t,x)), z, (I_dv_{\varepsilon_i})(s, X_{v_{\varepsilon_i}}(s,X^{-1}_{v_{\varepsilon_i}}(t,x)))).
\end{align*}
We write
\begin{equation}
   \begin{matrix}
    v^0_{\varepsilon_1}(X^{-1}_{v_{\varepsilon_1}}(t,x))-v^0_{\varepsilon_2}(X^{-1}_{v_{\varepsilon_2}}(t,x))&=&
        v^0_{\varepsilon_1}(X^{-1}_{v_{\varepsilon_1}}(t,x))-v^0_{\varepsilon_2}(X^{-1}_{v_{\varepsilon_1}}(t,x))\\
        &&\\
        &&+v^0_{\varepsilon_2}(X^{-1}_{v_{\varepsilon_1}}(t,x))-v^0_{\varepsilon_2}(X^{-1}_{v_{\varepsilon_2}}(t,x)).
    \end{matrix}\label{Deltau0}
\end{equation}
Thanks to the change of variables $y=X^{-1}_{v_{\varepsilon_1}}(t,x)$ and the relations
\begin{align*}
    |J_{X_{v_{\varepsilon_1}}(t,y)}|=&e^{\int_0^t\divv \mathcal{A}_{v_{\varepsilon_1}}(s,X_{v_{\varepsilon_1}}(t,y))ds},\\
    \mathcal{G}(s,X^{-1}_{v_{\varepsilon_1}}(t,x))\leqslant &\gamma:=\|R\|_{L^{\infty}_{t,x,I}}+\|a\|_{W^{1,\infty}_xL^{\infty}_{t,I}}+\|a\|_{W^{1,\infty}_IL^{\infty}_{t,x}}\|\psi_a\|_{W^{1,\infty}_xL^{\infty}_y}\overline{\rho},\\
    \divv \mathcal{A}_{v_{\varepsilon_1}}(s,X_{v_{\varepsilon_1}}(t,y))\leqslant &\tilde{a}:=\|a\|_{W^{1,\infty}_xL^{\infty}_{t,I}}+\|a\|_{W^{1,\infty}_IL^{\infty}_{t,x}}\|\psi_a\|_{W^{1,\infty}_xL^{\infty}_y}\overline{\rho},
\end{align*}
we conclude that 
\begin{align}
    &\int\limits_{\R^d}\biggl(v^0_{\varepsilon_1}(X^{-1}_{v_{\varepsilon_1}}(t,x))-v^0_{\varepsilon_2}(X^{-1}_{v_{\varepsilon_1}}(t,x))\biggr)\exp\biggl(\int_0^t\mathcal{G}_1(\tau,t,x)d\tau\biggr)dx\nonumber\\
    =&\int\limits_{\R^d}\biggl(v^0_{\varepsilon_1}(y)-v^0_{\varepsilon_2}(y)\biggr)\exp\biggl(\int_0^t\mathcal{G}_{v_{\varepsilon_i}}(\tau,y)d\tau\biggr)|J^{-1}_{X_{v_{\varepsilon_1}}(t,y)}|dy\nonumber\\
    \leqslant &e^{(\gamma+\tilde{\alpha})T}\|v^0_{\varepsilon_1}-v^0_{\varepsilon_2}\|_{L^1(\R)}.\label{CS1}
\end{align}
On the other hand, the compactness of the support of $v^0_{\varepsilon_2}$, together with the relation 
\[
|X_{v_{\varepsilon}}(t,y)-y|\leqslant \|a\|_{L^{\infty}}t,
\]
implies that
\[
v^0_{\varepsilon_2}(X^{-1}_{v_{\varepsilon_1}}(t,x))-v^0_{\varepsilon_2}(X^{-1}_{v_{\varepsilon_2}}(t,x))=0, \mbox{ for all }x\not\in O_t:=\mbox{supp }v^0+B_{r_t},
\]
where $B_{r_t}$ is the ball of radius $\|a\|_{L^{\infty}}t$. Hence, we obtain
\begin{align}
    &\int\limits_{\R^d}\biggl(v^0_{\varepsilon_2}(X^{-1}_{v_{\varepsilon_1}}(t,x))-v^0_{\varepsilon_2}(X^{-1}_{v_{\varepsilon_2}}(t,x))\biggr)\exp\biggl(\int_0^t\mathcal{G}_1(\tau,t,x)d\tau\biggr)dx\nonumber\\
    \leqslant &|O_t|\|v^0_{\varepsilon_2}\|_{W^{1,\infty}(\R^d)}e^{\gamma T}|X^{-1}_{v_{\varepsilon_1}}(t,x)-X^{-1}_{v_{\varepsilon_2}}(t,x)|\nonumber\\
    \leqslant& \widetilde{C}_T\|v_{\varepsilon_1}-v_{\varepsilon_2}\|_1,\label{CS2}
\end{align}
where we have used \eqref{ContXinv} on the second line\footnote{This term is responsible for the possible loss of regularity for $t>0$: In order to prove that $v_{\varepsilon}$ is a Cauchy sequence in $\mathcal{C}([0,T],W^{1,1}(\R^d))$, we would need a $W^{2,\infty}(\R^d)$ estimate over $v^0_{\varepsilon}$, which we do not have.}.\\
From the definition of $\mathcal{G}_u(t,x)$, we observe that
\begin{align}
    |\Delta E(s,t,x)|\leqslant &e^{\gamma T}\int_s^t |\mathcal{G}_1(\tau,t,x)-\mathcal{G}_2(\tau,t,x)|d\tau\nonumber\\
    \leqslant & \widetilde{C}_T\|v_{\varepsilon_1}-v_{\varepsilon_2}\|_1,\label{CS3}
\end{align}
where $\widetilde{C}_T$ depends on $\overline{\rho}$, $T$, the derivatives of $R$, $a$ and $\psi_a$, and on the constant appearing in \eqref{ContXinv}. \\
Using the change of variables $y=X_{v_{\varepsilon_2}}(s,X^{-1}_{v_{\varepsilon_2}}(t,x))$ and recalling that 
\begin{align*}
    |J_{X_{v_{\varepsilon_2}}(t,X^{-1}_{v_{\varepsilon_2}}(s,x))}|&=e^{\int_0^t\divv \mathcal{A}_u(\tau,X_{v_{\varepsilon_2}}(t,X^{-1}_{v_{\varepsilon_2}}(s,x)))d\tau-\int_0^s\divv \mathcal{A}_{v_{\varepsilon_2}}(\tau,x)d\tau},
\end{align*} 
we see that
\begin{align*}
    \int\limits_{\R^d}\mathcal{M}_{2}(s,t,x,z)dx&=\int\limits_{\R^d}m(s, y, z, (I_dv_{\varepsilon_2})(s, y))|J_{X_{v_{\varepsilon_2}}(t,X^{-1}_{v_{\varepsilon_2}}(s,y))}|^{-1}dy\\
    &\leqslant e^{2\tilde{a}T}\int\limits_{\R^d}\sup\limits_{s,z,I} m_{\varepsilon_i}(s, y, z, I)dy\leqslant e^{2\tilde{a}T}|\mathcal{K}|\overline{M}.
\end{align*}
Therefore, we have the bounds
\begin{align}
    \int\limits_{\R^d}v^0(X^{-1}_{v_{\varepsilon_2}}(t,x))|\Delta E(0,t,x)|dx&\leqslant e^{\tilde{\alpha}T}\|v^0\|_{L^{1}(\R^d)}\widetilde{C}_T\|v_{\varepsilon_1}-v_{\varepsilon_2}\|_1,\label{CS4}\\
    \int\limits_0^t\int\limits_{\R^d}\int\limits_{\R^d}\mathcal{M}_{2}(s,t,x,z)\biggl(\Delta V_{\varepsilon_1\varepsilon_2}(s, z)\biggr)dz\exp\left(\int\limits_s^t\mathcal{G}_1(\tau,t,x)d\tau\right)dxds&\leqslant e^{(\gamma+2\tilde{a})T}|\mathcal{K}|\overline{M}\|v_{\varepsilon_1}-v_{\varepsilon_2}\|_1,\label{CS5}\\
    \int\limits_0^t\int\limits_{\R^d}\int\limits_{\R^d}\mathcal{M}_{2}(t,x,z)v_{\varepsilon_j}(s, z)dz\biggl(|\Delta E(s,t,x)|\biggr)dxds&\leqslant e^{2\tilde{a}T}|\mathcal{K}|\overline{M} \overline{\rho}T\widetilde{C}_T\|v_{\varepsilon_1}-v_{\varepsilon_2}\|_1.\label{CS6}
\end{align}
The function $m(t,x,z,I)$ having a compact support on the $x$ variable, leads to
\[
\mathcal{M}_{1}(s,t,x,z)-\mathcal{M}_{2}(s,t,x,z)=0, \mbox{ for all }x\not\in \mathcal{K}+B_{2r_t}.
\]
On the other hand, the function $m$ being differentiable and Lipschitz, implies that, for all $x\in \mathcal{K}+B_{2r_t}$
\begin{align*}
   | \mathcal{M}_{1}(s,t,x,z)-\mathcal{M}_{2}(s,t,x,z)|\leqslant & \|m\|_{W^{1,\infty}_x}|X_{v_{\varepsilon_1}}(s,X^{-1}_{v_{\varepsilon_1}}(t,x))-X_{v_{\varepsilon_2}}(s,X^{-1}_{v_{\varepsilon_2}}(t,x))|\\
   &+\mu|(I_dv_{\varepsilon_1})(s, X_{v_{\varepsilon_1}}(s,X^{-1}_{v_{\varepsilon_1}}(t,x)))-(I_dv_{\varepsilon_2})(s, X_{v_{\varepsilon_2}}(s,X^{-1}_{v_{\varepsilon_2}}(t,x)))|\\
   \leqslant& (\|m\|_{W^{1,\infty}_x}+\|\psi_d\|_{W^{1,\infty}_x}\overline{\rho})|X_{v_{\varepsilon_1}}(s,X^{-1}_{v_{\varepsilon_1}}(t,x))-X_{v_{\varepsilon_2}}(s,X^{-1}_{v_{\varepsilon_2}}(t,x))|\\
   &+\|\psi_d\|_{L^{\infty}}\|v_{\varepsilon_1}-v_{\varepsilon_2}\|_{L^1(\R^d)}.
\end{align*}
Using first Lemma \ref{ChardX2} and then Lemma \ref{ChardXinv2}, we conclude that there exists a constant $\widetilde{C}_T$ such that
\begin{align*}
   | \mathcal{M}_{1}(s,t,x,z)-\mathcal{M}_{2}(s,t,x,z)|\leqslant & \widetilde{C}_T\biggl(|X^{-1}_{v_{\varepsilon_1}}(t,x)-X^{-1}_{v_{\varepsilon_2}}(t,x)|+\|v_{\varepsilon_1}-v_{\varepsilon_2}\|_{1}+\|v_{\varepsilon_1}-v_{\varepsilon_2}\|_{L^1(\R^d)}\biggr)\\
   \leqslant &\widetilde{C}_T\biggl(\|v_{\varepsilon_1}-v_{\varepsilon_2}\|_{1}+\|v_{\varepsilon_1}-v_{\varepsilon_2}\|_{L^1(\R^d)}\biggr).
\end{align*}
Therefore, we have the bound
\begin{align}
    &\int\limits_0^t\int\limits_{\R^d}\int\limits_{\R^d}\biggl(\mathcal{M}_{1}(s,t,x,z)-\mathcal{M}_{2}(s,t,x,z)\biggr)v_{\varepsilon_1}(s, z)dz\exp\left(\int_s^t\mathcal{G}_1(\tau,t,x)d\tau\right)dxds\nonumber\\
    \leqslant &|\mathcal{K}+B_{2r_T}|\overline{\rho}\widetilde{C}_Te^{\gamma T}\|v_{\varepsilon_1}-v_{\varepsilon_2}\|_{1}.\label{CS7}
\end{align}
Putting together the bounds \eqref{CS1} through \eqref{CS7}, we get
\begin{align*}
    \|v_{\varepsilon_1}-v_{\varepsilon_2}\|_{L^1(\R^d)}\leqslant \widetilde{C}_T\left( \|v^0_{\varepsilon_1}-v^0_{\varepsilon_2}\|_{L^1(\R)}+\int_0^T\|v_{\varepsilon_1}-v_{\varepsilon_2}\|_{L^1(\R^d)}ds\right).
\end{align*}
Thanks to Gr\"onwall's lemma, we have then the relation
\[
\sup\limits_{t\in[0,T]}\|v_{\varepsilon_1}-v_{\varepsilon_2}\|_{L^1(\R^d)}\leqslant \widetilde{C}_T\|v^0_{\varepsilon_1}-v^0_{\varepsilon_2}\|_{L^1(\R^d)}
\]
for some $\widetilde{C}_T$ independent of $\varepsilon_1$ and $\varepsilon_2$, which proves that, up to the extraction of a sub-sequence, $v_{\varepsilon}$ is a Cauchy sequence in $\mathcal{C}([0,T],L^{1}(\R^d))$. Therefore, there exists $v\in \mathcal{C}([0,T],L^1(\R^d))$ such that 
\[
\lim\limits_{\varepsilon\rightarrow 0}\sup\limits_{t\in [0,T]}\|v_{\varepsilon}-v\|_{L^{1}(\R^d)}=0.
\]
Furthermore, such function satisfies the bounds \eqref{BL1} and \eqref{BWm1}.\\
We claim now that the sequence $L^*_{v_{\varepsilon}}\varphi$ converges to $L^*_{v}\varphi$ in $L^{\infty}([0,T]\times\R^d)$ for all $\varphi\in\mathcal{C}^1_c([0,T)\times \R^d) $. This is a direct consequence of the relation
\[
|L^*v_{\varepsilon}\varphi-L^*v_{\varepsilon}\varphi|\leqslant( L_r\|\psi_g\|_{L^{\infty}}+\mu\|\psi_d\|_{L^{\infty}})\|\varphi\|_{L^{\infty}(\R^d)}\|v^{\varepsilon}-v\|_{L^1(\R^d)}.
\]
In order to conclude, we recall that all classical solutions are weak solutions, and therefore, for all $\varepsilon>0$
\[
\int_0^T\int_{\R^d}v_{\varepsilon}L^*_{v_{\varepsilon}}\varphi dxdt=\int_{\R^d}v^0_{\varepsilon}\varphi dx,
\]
and taking the limit when $\varepsilon$ goes to $0$ we see that $v$ is a weak solution of problem \eqref{Pv}.
\end{proof}
Considering initial data with compact support might be enough in order to model most of the biological scenarios found in nature. However, the hypothesis $v^0\in W^{k,\infty}(\R^d)$ might be too restrictive for some real life scenarios. Furthermore, the study of the problem when more general initial conditions are present, is of theoretical interest. We show below that, when $\partial_I a=0$, a solution exists for any initial data $v^0\in W^{k,1}(\R^d)$, $k\geqslant 1$.
\begin{theorem}\label{T3}
Under hypothesis \eqref{psia} through \eqref{GvsM}, if $\partial_Ia=0$, for all non-negative functions $v^0\in W^{k,1}(\R^d)$, there exists a unique non-negative weak solution $v\in \mathcal{C}([0,T],W^{k,1}(\R^d))$ of problem \eqref{Pv}. Furthermore, such a solution satisfies
\begin{align}
    \sup\limits_{t\in[0,T]}\|v\|_{L^1(\R^d)}\leqslant \max\{ \|v^0\|_{L^1(\R^d)},\frac{I^*}{\underline{\psi_g}}\},\label{BL1g2}\\
    \sup\limits_{t\in[0,T]}\|v\|_{W^{k,\infty}(\R^d)}\leqslant C_T\|v^0\|_{W^{k,1}(\R^d)}.\label{BWm1g2}
\end{align}
\end{theorem}
\begin{proof}
    As in the proof for $k=1$ when $\partial_I a\neq 0$, we can approximate any function $v^0\in W^{k,1}$ by a smooth, compactly supported sequence of functions $v^0_{\varepsilon}$. The same arguments as in the previous proof will show that $v_{\varepsilon}$, the sequence of solutions associated to $v^0_{\varepsilon}$, is a Cauchy sequence in $\mathcal{C}([0,T],L^1(\R^d))$. Furthermore, given that the second term in \eqref{Deltau0}, which is responsible for the possible loss of regularity in the previous case, is equal $0$ when $\partial_I a=0$, we show that $v_{\varepsilon}$ is a Cauchy sequence in $\mathcal{C}([0,T],W^{k,1}(\R^d))$ as well. We prove as in Theorem \ref{T2} that the limit of $v_{\varepsilon}$ is the required weak solution.
\end{proof}
Given that the regularity of the solution varies depending on whether $\partial_I a=0$ or $\partial_I a\neq 0$, and that such regularity will be of importance in the upcoming sections, we define the parameter
\begin{equation*}
    \kappa:=\left\{
    \begin{matrix}
    k-1,&\mbox{ if }\partial_Ia\neq 0,\\
    &\\
    k,&\mbox{ if }\partial_Ia= 0,
    \end{matrix}
    \right.
\end{equation*}
which encompasses the information over said regularity.\\

We remark that if we had $\partial_I m=0$, we might obtain existence for a larger class of mutation functions $m$. For conciseness, we will however not consider such cases in the present work.

\section{Particle Method}\label{TPM}
The particle method basically consists in searching for an approximate solution of problem \eqref{Pv} which is a sum of weighted Dirac masses.\\
Throughout the following section we suppose
\begin{align}
    \psi_a&\in \mathcal{C}\left([0,T] \times \R_x^d, W^{1,\infty}(\R^d) \right) \cap \mathcal{C} \left( [0,T] \times \R_y, \mathcal{C}^{2}(\R^d_x)\cap W^{2,\infty}(\R^d_x) \right)                     \label{psiasmoothW},\\
   0<\underline{\psi_g}\leqslant \psi_g&\in \mathcal{C}\left([0,T]\times \R_x^d, W^{1,\infty}(\R_y^d)\right)\cap \mathcal{C}\left( [0,T]\times \R_y^d, \mathcal{C}^1(\R_x^d)\cap W^{1,\infty}(\R_x^d)  \right)     \label{psigsmoothW}\\
    \psi_d&\in \mathcal{C}\left([0,T]\times \R^d_x, W^{1,\infty}(\R_y^d)\right)\cap \mathcal{C}\left([0,T]\times \R^d_y, \mathcal{C}^1(\R_x^d)\cap W^{1,\infty}(\R_x^d). \right)\label{psidsmoothW}
\end{align}
Notice that, unlike the set of hypotheses \eqref{psia}-\eqref{psid}, we have imposed $W^{1,\infty}(\R^d)$ regularity for the $y$ variable, which is needed in order to approximate the integral terms by sums over a countable set.\\
Consider as well 
\begin{equation}
0 \leqslant m \in \mathcal{C}\left([0,T]\times \R^d_x \times \R^d_y, W^{1,\infty}(\R_I) \right) \cap \mathcal{C}\left([0,T]\times \R^d_x \times \R_I, W^{1,\infty}(\R^d_y) \right) \cap \mathcal{C}\left([0,T] \times \R^d_y \times \R_I, \mathcal{C}^{1}_c(\R_x^d)\right) 
\label{mPM}    
\end{equation}
satisfying hypotheses \eqref{mLip} through \eqref{GvsM}.\\
For $h>0$, consider a countable set of indices $\mathcal{J}_h\in\mathbb{Z}^d$, points $x^0_i\in\R^d$ and weights $w^0_i$ for $i\in \mathcal{J}_h$. The weights $w^0_i$ can be regarded as the respective masses of a collection of subsets $\Omega^0_i\subset\R^d$ satisfying
\begin{equation}
    \Omega^0_i\cap\Omega^0_j=\emptyset,\mbox{ if }i\neq j,\mbox{ and }\bigcup\limits_{i\in\mathcal{J}_h}\Omega^0_i=\R^d.\label{omega0k}
\end{equation}
For example, we may choose the $\Omega^0_i$ as the set of all non intersecting cubes of side length equal $h$ having the points $hi$ as centers , $i\in \mathbb{Z}^d$. This way, $w^0_i=h^d$, with each of the $x^0_i$ being a point in $\Omega^0_j$. In general we assume that there exist positive constants $c$ and $C$ such that
\begin{align}
    &ch\leqslant|x^0_i-x^0_j|\leqslant Ch,\ \forall i\neq j,\label{x0}\\
    &ch^d\leqslant w^0_i\leqslant Ch^d,\ \forall i\in\mathcal{J}_h .\label{w0}
\end{align}
Following \cite{degond1989weighted}, the particle method then consists in looking for a measure $\nu_h$ of the form
\[
\nu_h(t)=\sum\limits_{i\in\mathcal{J}_h}\nu_i(t)w_i(t)\delta_{x_i(t)},
\]
where $(\nu:=\{v_i(t)\}_{i\in\mathcal{J}_h},w:=\{w_i(t)\}_{i\in\mathcal{J}_h},\overline{x}:=\{x_i(t)\}_{i\in\mathcal{J}_h})$, is the solution of the following system
\begin{equation}
    \left\{
    \begin{matrix*}[l]
    \dot{x}_i(t) &=& A_{\nu,w}(t,x_i),\\
    &&\\
    \dot w_i(t)&=&\divv A_{\nu,w}(t,x_i(t))w_i(t),\\
    &&\\
    \dot \nu_i(t)&=&\big(-\divv A_{\nu,w}(t,x_i(t))+R(t,x_i(t),I_g(t,x_i(t),\nu,w))\big)\nu_i(t)\\
    &&\\
    &&+\sum\limits_{j\in\mathcal{J}_h}w_j(t)\nu_j(t)m(t,x_i(t),x_j(t),I_d(t,x_i(t),\nu,w)),\\
    &&\\
    x_i(0)&=&x^0_i,\ w_i(0)\ =\ w^0_i,\ \nu_i(0)\ =\ v^0(x^0_i),
    \end{matrix*}
    \right.\label{PM}
\end{equation}
where
\[
A_{\nu,w}(t,x)=a(t,x,I_a(t,x,\nu,w)),
\]
and
\begin{align*}
    I_l(t,x,\nu,w)&:=\sum\limits_{j\in\mathcal{J}_h}\nu_j(t)w_j(t)\psi_l(t,x,x_j(t)),
\end{align*}
with $l\in\{a,g,d\}$.\\
In what follows we assume that $h$ and $x^0_k$ are chosen in such a way that
\begin{equation}
    \|v^0\|_{1,h}:=\sum\limits_{i\in\mathcal{J}_h}v^0(x^0_i)w^0_i<\infty.\label{u0h}
\end{equation}
We define the subset of indices
\[
\mathcal{J}^m_h:=\{i\in \mathcal{J}_h:x^0_i\in \mbox{supp }m+B_{\|a\|_{L^{\infty}}T}\},
\]
where $B_{\|a\|_{L^{\infty}}T}$ is the ball of radius $\|a\|_{L^{\infty}}T$. The compact support of $m$ implies that $|\mathcal{J}^m_h|<\infty$.\\

For a positive value of $h$ and a set of indexes $\mathcal{J}_h$ we define the functional spaces
\begin{align*}
    \mathcal{l}^1(\mathcal{J}_h)&:=\big\{u=\{u_i\}_{i\in\mathcal{J}_h}:\sum\limits_{i\in\mathcal{J}_h}|u_i|< +\infty\big\},\\
    \mathcal{l}^{\infty}(\mathcal{J}_h)&:=\big\{w=\{w_i\}_{i\in\mathcal{J}_h}:\sup\limits_{i\in \mathcal{J}_h}|w_i|< +\infty\big\}.
\end{align*}
We equip these spaces with the norms
\[
\|u\|_{\mathcal{l}^1}:=\sum\limits_{i\in\mathcal{J}_h}|u_i|\mbox{ and }\|w\|_{\mathcal{l}^{\infty}}:=\sup\limits_{i\in \mathcal{J}_h}|w_i|
\]
respectively. It is clear that for all $u\in \mathcal{l}^1(\mathcal{J}_h)$ and $v\in \mathcal{l}^{\infty}(\mathcal{J}_h)$, then $uv\in \mathcal{l}^1(\mathcal{J}_h)$.
For $T>0$ we define as well the spaces
\[
X^T_{h}:=\mathcal{C}([0,T],\mathcal{l}^1(\mathcal{J}_h)) \mbox{ and }Y^T_{h}:=\mathcal{C}([0,T],\mathcal{l}^{\infty}(\mathcal{J}_h)),
\]
equipped with the norms
\[
\|u\|_{1,h}:=\sup\limits_{t\in[0,T]}\|u(t)\|_{\mathcal{l}^1}\mbox{ and }\|w\|_{\infty,h}:=\sup\limits_{t\in[0,T]}\|w(t)\|_{\mathcal{l}^{\infty}}.
\]
Problem \eqref{PM} is a strongly coupled system of ODEs, with an infinite number of unknowns and equations. In some cases the system becomes uncoupled (for example if $\partial_I a=0$) or with a finite number of equations and unknowns (for example if $v^0$, $a$ and $m$ have compact support), however, for the sake of generality, we present below the proof of existence of solution in the general case, and later discuss briefly these particular scenarios.\\
We start by giving two results that will be of great use for the proof of existence  for problem \eqref{PM}. First, we deal with the existence of solution for a simpler system of infinite equations with infinitely many unknowns:
 \begin{lem}\label{CLinf}
 Consider $a\in\mathcal{C}([0,T],(W^{1,\infty}(\R^{d+1}) ))^d$, $u\in X^T_h$, $w\in Y^T_h$ and $\psi_a$ satisfying hypothesis \eqref{psiasmoothW}. Then there exists a unique family of functions $x:=\{x_i\}_{i\in\mathcal{J}_h}$, $x_i\in\mathcal{C}^1([0,T])$ for all $i\in\mathcal{J}_h $ which is solution of the system of equations
 \begin{equation}
     \dot{x_i}(t)=A_{u,w}(t,x_i), \quad t\in [0,T], \ x_i(0)=x^0_i. \label{inftot}
 \end{equation}
 \end{lem}
 When $\partial_I a=0$, system \eqref{inftot} becomes uncoupled, each individual equation has a solution, thanks to the classic Cauchy-Lipschitz theory. The proof of the general case is given in Appendix \ref{AnnB}.\\
 The second auxiliary result comes from approximation theory, and it will also be of great use in Section \ref{Cvgn}:
 \begin{lem}\label{AT}
\begin{align*}
\forall \varphi \in W^{k, 1}(\R^d), \quad 
    \bigg\lvert \int_{\R^d}{\varphi(x)dx}-\underset{i\in\mathcal{J}_h}{\sum}{w_i(t)\varphi(x_i(t))}\bigg\rvert \leqslant Ch^k\lVert \varphi \rVert_{k,1}, 
\end{align*}
where $C$ is a constant which depends on $a$, $\psi_a$, $\|vw\|_{1,h}$ and $T$. 
\label{error integral-sum}
\end{lem}
This result is a direct corollary of Lemma 8 in \cite{mas1987particle}. More details regarding its proof are given in Appendix \ref{AppTheo}.\\
From now on, we suppose $h$ to be small enough so that for any $t\in [0, T]$, 
\begin{equation}
    \underset{i\in\mathcal{J}_h}{\sum}{w_i(t)m(t,x_i(t),y,I_d(t,x_i(t),\nu,w))}<K+\frac{r^*}{2},\label{mdisc}
\end{equation}
where the values of $K$ and $r^*$ are given in \eqref{bigMK} and \eqref{GvsM} respectively. Such a choice is always possible thanks to Lemma \ref{AT}.
 \begin{theorem}
Under hypothesis \eqref{aetR1} through \eqref{GvsM} and \eqref{psiasmoothW} through \eqref{w0}, for all $T>0$ and all non-negative initial data $v^0\in \mathcal{l}^1(\mathcal{J}_h,\Omega^0)$ there exists a unique solution $x_i\in\mathcal{C}^1([0,T])$, for all $ i\in\mathcal{J}_h$, $w:=\{w_i(\cdot)\}_{i\in\mathcal{J}_h}\in \mathcal{C}([0,T],\mathcal{l}^{\infty}(\mathcal{J}_h))$ and $0\leqslant \nu:=\{\nu_i(\cdot)\}_{i\in\mathcal{J}_h}\in \mathcal{C}([0,T],\mathcal{l}^{1}(\mathcal{J}_h))$ of problem \eqref{PM}. Furthermore, there exist positive constants $c_T$ and $C_T$ such that the solution satisfies, for all $t\in [0,T]$
\begin{gather}
    c_Th\leqslant|x_i(t)-x_j(t)|\leqslant C_Th,\ \forall i,j\in \mathcal{J}_h,\ i\neq j,\label{xt}\\
    c_Th^d\leqslant w_i(t)\leqslant C_Th^d,\ \forall i\in \mathcal{J}_h,\label{wt}\\
    \|\nu w\|_{1,h}\leqslant\max\{\|v^0h^d\|_{\mathcal{l}^1},\frac{I^*}{\underline{\psi_g}}\} \label{vhbound}.
\end{gather}

\label{theo bounds particles}
\end{theorem}
\begin{proof} 
Consider $v^0\in \mathcal{l}^1(\mathcal{J}_h,\Omega^0)$, satisfying $v^0\geqslant0$. Consider as well $\alpha>1$, and define 
\begin{align*}
     \overline{\rho}_{\alpha}&:=\max\{\alpha h^d \|v^0\|_{\mathcal{l}^1},\frac{I^*}{\underline{\psi_g}}\},\\
     \tilde{a}&:=\|a\|_{W^{1,\infty}_{x}L^{\infty}_{I,t}}+\|a\|_{W^{1,\infty}_{I}L^{\infty}_{t,x}}\|\psi_a\|_{W^{1,\infty}_{x}}C_{\alpha}.
\end{align*}

For $T>0$ we define the set
\[
D^T_{\alpha}:=\{(u, w) \in X^{T}_h\times Y^{T}_h: \|u w\|_{1,h}\leqslant \overline{\rho}_{\alpha}, \forall t\in [0, T], u(t)\geqslant 0, w(t)\geqslant 0,  h^de^{-\tilde{a}t}\leqslant w_k(t)\leqslant h^de^{\tilde{a}t}\}.
\]
For any $(u,w)\in D^T_{\alpha}$ we introduce the problem, for $t\in [0,T]$, 
\begin{equation}
    \left\{
    \begin{matrix*}[l]
    \dot x_i(t)&=&A_{u,w}(t,x_i),\\
    &&\\
    \dot \omega_i(t)&=&\divv A_{u,w}(t,x_i(t))\omega_i(t),\\
    &&\\
    \dot \nu_i(t)&=&\big(-\divv A_{u,w}(t,x_i(t))+R(t,x_i(t),I_g(t,x_i(t),u,w))\big)\nu_i(t)\\
    &&\\
    &&+\sum\limits_{j\in\mathcal{J}_h}\omega_j(t)u_j(t)m(t,x_i(t),x_j(t),I_d(t,x_i(t),u,w)),   \\
    &&\\
    x_i(0)&=&x^0_i,\ \omega_i(0)\ =\ w^0_i,\ \nu_i(0)\ =\ v^0(x^0_i).
    \end{matrix*}
    \right.\label{Oh}
\end{equation}
We denote $(\nu,\omega)=\Phi( u,w)$.

For each pair $(u,w)$, the existence and uniqueness of $x_i$ is immediate from Lemma \ref{CLinf}. Furthermore, for all values of $i$, we have the following explicit expression for $\omega_i$
\[
\omega_{i}(t)=w^0_ie^{\int_0^t\divv A_{u,w}(s,x_i(s))ds},
\]
which satisfies, for any $t\in [0,T]$
\[
h^d e^{-\tilde{a}t}\leqslant \omega_{i}(t)\leqslant h^d e^{\tilde{a}t}.
\]
On the other hand, for all $(u,w)\in D^{T}_{\alpha}$ and all values of $i$, the right-hand side of the differential equation in \eqref{Oh} is well defined, as we have for all $t\in [0,T]$
\begin{align*}
  \sum\limits_{j\in\mathcal{J}_h}w_j(t)u_j(t)m(t,x_i(t),x_j(t),I_d(t,x_i(t),u,w))&\leqslant\overline{M}\sum\limits_{j\in\mathcal{J}_h}w_j(t)u_j(t)\leqslant \overline{M} \overline{\rho}_{\alpha}.
\end{align*}
Therefore, the expression for $\nu_i$ is given by
\begin{equation}
    \nu_i(t)=v^0(x^0_i)e^{ \int_0^t\mathcal{G}_i(s)ds}+\int\limits_0^t\sum\limits_{j\in\mathcal{J}_h}w_j(s)u_j(s)m(s,x_i(s),x_j(s),I_d(s,x_i(s),u,w))e^{\int_s^t\mathcal{G}_i(\tau)d\tau}ds,\label{vk}
\end{equation}
where
\[
\mathcal{G}_i(t):=-\divv A_{u,w}(s,x_i(s))+R(s,x_i(s),I_g(s,x_i(s),u(s),w(s))),
\]
satisfies
\[
\underset{t\in [0,T]}{\sup}\,|\mathcal{G}_i(t)|\leqslant \gamma:=\|R\|_{L^{\infty}_{t,x,I}}+\|a\|_{W^{1,\infty}_xL^{\infty}_{t,I}}+\|a\|_{W^{1,\infty}_IL^{\infty}_{t,x}}\|\psi_a\|_{W^{1,\infty}_xL^{\infty}_y}\overline{\rho}_{\alpha}.
\]
The positiveness of $\nu$ is immediate from the positiveness of $v^0$ and $m$. 

Furthermore, given that $|x_i(t)-x^0_i|\leqslant \|a\|_{L^{\infty}}T$, for all $k\in \mathcal{J}_h$ and $t\in[0,T]$, we have 
\[
m(t,x_i(s),y,I)=0,
\]
for all $i\not\in \mathcal{J}^m_h$, $t\in[0,T]$, $y\in \R^d$ and $I\in \R$. As a result, multiplying \eqref{vk} by $\omega_i(t)$ for each $i$ and adding for all values of $i$, we obtain
\begin{align*}
    \sum\limits_{i\in\mathcal{J}_h}\nu_i(t)\omega_i(t)&\leqslant e^{(\gamma+\tilde{\alpha})T}\left(\sum\limits_{i\in\mathcal{J}_h}v^0(x^0_i)h^d+h^d\int\limits_0^t\sum\limits_{i\in\mathcal{J}_h}\sum\limits_{j\in\mathcal{J}_h}w_j(s)u_j(s)m(s,x_i(s),x_j(s),I_d(s,x_i(s),u,w))ds\right)\\
    &=e^{(\gamma+\tilde{\alpha})T}\left(\sum\limits_{i\in\mathcal{J}_h}v^0(x^0_i)h^d+h^d\int\limits_0^t\sum\limits_{i\in\mathcal{J}^m_h}\sum\limits_{j\in\mathcal{J}_h}w_j(s)u_j(s)m(s,x_i(s),x_j(s),I_d(s,x_i(s),u,w))ds\right)\\
    &\leqslant e^{(\gamma+\tilde{\alpha})T}\left(\sum\limits_{i\in\mathcal{J}_h}v^0(x^0_i)h^d+h^d|\mathcal{J}^m_h|\overline{M}\int\limits_0^t\sum\limits_{j\in\mathcal{J}_h}w_j(s)u_j(s)ds\right)\\
    &\leqslant e^{(\gamma+\tilde{\alpha})T}(\frac{1}{\alpha}+TK_h)\overline{\rho}_{\alpha},
\end{align*}
where $K_h:=h^d|\mathcal{J}^m_h|\overline{M}$\footnote{Notice that $h^d|\mathcal{J}^m_h|\approx |\mbox{supp }m+B_{\|a\|_{L^{\infty}}T}|$}. Thanks to the condition $\alpha>1$ there exists $T_{\alpha}$ (only depending on $\alpha$ and on the coefficients of the problem) such that $\Phi:D^{T}_{\alpha}\rightarrow D^{T}_{\alpha}$, for all $T\leqslant T_{\alpha}$.\\
We now prove that there exists $T\in(0,T_{\alpha})$ such that $\Phi$ is a contraction over $D^{T}_{\alpha}$.\\
{\it Step 1: Bounds over $x=\{x_i\}_{i\in\mathcal{J}_h}$}\\
Let $(u^1,w^1)$ and $(u^2,w^2)$ be two pairs in $D^{T}_{\alpha}$, and let $x^1,x^2$ be the respective solutions of
\[
\dot{x^j_i}=A_{u^j,w^j}(t,x^j_i).
\]
By following the same ideas as in the proof of Lemma \ref{ChardX2} (see Appendix \ref{AnnA}), we obtain that for all $t\in [0,T]$, 
\[
\|x^1(t)-x^2(t)\|_{\infty,h}\leqslant C(T,h)\left(\|w^1-w^2\|_{\infty,h}+\|u^1-u^2\|_{1,h}\right),
\]
where the constant $C(T,h)$ satisfies $\lim\limits_{T\rightarrow 0}C(T,h)=0$.
\\ 
{\it Step 2: Bounds over $\omega=\{\omega_i\}_{i\in\mathcal{J}_h}$}\\
From the expression for $\omega$, we get, for all $t\in [0,T]$
\begin{align*}
    |\omega^1_i(t)-\omega^2_i(t)|\leqslant& h^dTe^{\tilde{a}T}\left(\|a\|_{W^{2,\infty}_{x,I}}(1+|\partial_xI_a(t,x^2_i,u^2,w^2)|)\|x^1-x^2\|_{\infty,h}\right.\\
    &+\|a\|_{W^{2,\infty}_{x,I}}(1+|\partial_xI_a(t,x^2_i,u^2,w^2)|)|I_a(t,x^1_i,u^1,w^1)-I_a(t,x^2_i,u^2,w^2)|\\
    &\left.+\|a\|_{W^{1,\infty}_I}|\partial_xI_a(t,x^1_i,u^1,w^1)-\partial_xI_a(t,x^2_i,u^2,w^2)|\right).
\end{align*}
On the other hand we have
\begin{align*}
    |\partial_xI_a(t,x^2_i,u^2,w^2)| \leqslant& \|\psi_a\|_{W^{1,\infty}_x}\overline{\rho}_{\alpha},\\
    |I_a(t,x^1_i,u^1,w^1)-I_a(t,x^2_i,u^2,w^2)|\leqslant & \overline{\rho}_{\alpha}\|\psi_a\|_{W^{1,\infty}_{x,y}}\|x^1-x^2\|_{\infty,h},\\
    &+\|\psi_a\|_{L^{\infty}_{x,y}}e^{\tilde{a}T}(h^d\|u^1-u^2\|_{1,\infty}+\frac{\overline{\rho}_{\alpha}}{h^d}\|w^1-w^2\|_{\infty,h}),
\end{align*}
and
\begin{align*}
    |\partial_xI_a(t,x^1_i,u^1,w^1)-\partial_xI_a(t,x^2_i,u^2,w^2)|\leqslant &\overline{\rho}_{\alpha}\|\psi_a\|_{W^{2,\infty}_{x,y}}\|x^1-x^2\|_{\infty,h}\\
    &+\|\psi_a\|_{W^{1,\infty}_{x}}e^{\tilde{a}T}(h^d\|u^1-u^2\|_{1,\infty}+\frac{\overline{\rho}_{\alpha}}{h^d}\|w^1-w^2\|_{\infty,h}).
\end{align*}
In conclusion, there exists a constant $C(T,h)$, satisfying $\lim\limits_{T\rightarrow 0}C(T,h)=0$ such that
\[
\|\omega^1-\omega^2\|_{\infty,h}\leqslant C(T,h)\left(\|w^1-w^2\|_{\infty,h}+\|u^1-u^2\|_{1,h}\right).
\]
{\it Step 3: Bounds over $\nu=\{\nu_i\}_{i\in\mathcal{J}_h}$}\\
Using the expression for $\nu$, the regularity of $m$ and bounds similar to those used for $\omega$, we see that there exists a constant $C(T,h)$ satisfying $\lim\limits_{T\rightarrow 0}C(T,h)=0$ such that
\begin{align*}
    \|\nu^1-\nu^2\|_{1,h}\leqslant C(T,h)\left(\|w^1-w^2\|_{\infty,h}+\|u^1-u^2\|_{1,h}\right).
\end{align*}
Consequently, there exists a constant $C(T,h)$ satisfying $\lim\limits_{T\rightarrow 0}C(T,h)=0$, such that
\[
\|\omega^1-\omega^2\|_{\infty,h}+\|\nu^1-\nu^2\|_{1,h}\leqslant C(T,h)\left(\|w^1-w^2\|_{\infty,h}+\|u^1-u^2\|_{1,h}\right),
\]
which implies that, for $0<T_1\leqslant T_{\alpha}$ small enough, $\Phi$ is a contraction over $D^{T_1}_{\alpha}$, and therefore it has a unique fixed point. Such fixed point is a solution of problem \eqref{PM} over $[0,T_1)$.\\
We now claim that the solution exists for $T$ arbitrary, and furthermore, it satisfies the relation \eqref{vhbound}. Let $T_f$ be the maximal time of existence of solution. Suppose that there exists $t_0\in(0,T_f]$ such that $\|\nu w\|_{1,h}> \overline{\rho}_{\alpha}$. This implies that there exist $\delta\geqslant 0$ and $t^*>0$ such that for a certain finite subset of $\mathcal{J}_h$, that we denote as $\mathcal{K}_h$, the following statements are true:
\begin{align*}
    \sum\limits_{i\in \mathcal{K}_h}\nu_i(t)w_i(t)\leqslant \overline{\rho}_{\alpha},\ \forall\  t\in[t^*-\delta,t^*],\\
    \sum\limits_{i\in \mathcal{K}_h}\nu_i(t)w_i(t)> \overline{\rho}_{\alpha},\ \forall\ t\in(t^*,t^*+\delta].
\end{align*}
This implies the existence of $t_1\in [t^*,t^*+\delta]$ such that the following properties are satisfied simultaneously
\begin{equation}
   \sum\limits_{i\in \mathcal{K}_h}\nu_i(t_1)w_i(t_1)\geqslant \overline{\rho}_{\alpha}\mbox{ and }\left(\sum\limits_{i\in \mathcal{K}_h}\nu_iw_i\right)'(t_1)\geqslant 0.\label{Piht1} 
\end{equation}
Multiplying the equation satisfied by $\nu_i(t)$ by $w_i(t)$ we get the relation
\begin{align*}
    \dot \nu_i(t)w_i(t)=& \big(-\divv A_{\nu,w}(t,x_i(t))+R(t,x_i(t),I_g(t,x_i(t),\nu,w))\big) \nu_i(t)w_i(t)\\
    &+w_i(t)\sum\limits_{j\in\mathcal{J}_h}w_j(t)\nu_j(t)m(t,x_i(t),x_j(t),I_g(t,x_i(t),\nu,w)),
\end{align*}
while, directly from the equation for $w_i(t)$ we deduce
\[
\nu_i(t)\dot w_i(t)= \divv A_{\nu,w}(t,x_i(t)) \nu_i(t)w_i(t).
\]
Therefore, adding both relations for $i\in \mathcal{K}_h$ and using \eqref{mdisc}, we get
\begin{align*}
    \left(\sum\limits_{i\in \mathcal{K}_h}\nu_i w_i\right)'(t)=& \left(\sum\limits_{i\in \mathcal{K}_h}R(t,x_i(t),I_g(t,x_i(t),\nu,w)) \nu_i(t)w_i(t)\right)\\
    &+\sum\limits_{j\in\mathcal{J}}w_j(t)\nu_j(t)\sum_{i\in \mathcal{K}_h\cap \mathcal{J}^m_h }m(t,x_i(t),x_j(t),I_d(t,x_i(t),\nu,w))w_i(t)\\
    \leqslant&\left(\sum\limits_{i\in \mathcal{K}_h}R(t,x_i(t),I_g(t,x_i(t))) \nu_i(t)w_i(t)\right)+(K+\frac{r^*}{2})\sum\limits_{j\in\mathcal{K}_h}w_j(t)\nu_j(t).
\end{align*}
Given that $\|\nu(t_1)w(t_1)\|_{\mathcal{l}^1}\geqslant \overline{\rho}_{\alpha}\geqslant \frac{I^*}{\underline{\psi}_g}$, then $I_g(t_1,x_i(t_1),\nu(t_1),w(t_1))\geqslant I^*$, and consequently
\[R(t_1,x_i(t_1),I_g(t_1,x_i(t_1),\nu(t_1),w(t_1)))<-r^*-K,\nonumber
\]
for all values of $i$, which in turn implies that
\begin{align*}
    \left(\sum\limits_{i\in \mathcal{K}_h}\nu_iw_i\right)'(t_1)\leqslant -\frac{r^2}{2}\sum\limits_{i\in \mathcal{K}_h}\nu_i(t_1)w_i(t_1)<0,
\end{align*}
which contradicts \eqref{Piht1}. Therefore, $\|\nu w\|_{1,h}\leqslant  \overline{\rho}_{\alpha}$ for all values of $\alpha>1$ and for all $t\in (0, T_f)$. We can then iterate the arguments used to prove existence of a solution, and conclude that the solution can be extended to any interval $[0,T]$. As $\|\nu w\|_{1,h}\leqslant \overline{\rho}_{\alpha}$ for all $t$, independently of $\alpha$, taking the limit when $\alpha$ goes to $1$, we obtain \eqref{vhbound}.
\end{proof}

\section{Convergence of the numerical solution towards a weak solution}\label{Cvgn}
We study now the conditions under which a solution of problem \eqref{Oh} converges towards a solution of problem \eqref{Pv}, in a certain sense that will be defined later. We split our analysis in two cases: first, the study of convergence on a finite interval of time $[0,T]$. We will see that for any $T>0$, the solution obtained through the particle method converges towards the solution of the PDE \eqref{Pv}. However, the speed of convergence might be affected by the value of $T$. The second case we study is the asymptotic  proximity of both solutions when $t$ goes to $\infty$. This is a far more complex and interesting issue, and we show different examples exposing some of the behaviours that can be observed.\\ 

Directly from the study of existence of solutions for each problem, we notice that the sets of hypotheses we have used, do not coincide. We give a set of hypotheses which simultaneously guarantees the existence of solution for both problems, while taking into account the distinction of cases involved in the definition of $\kappa$.\\
For a certain $T>0$ and $k>0$, we consider the functions
\begin{align}
    \psi_a&\in\mathcal{C}([0,T]\times \R_x^d, W^{1,\infty}(\R_y^d)) \cap  \mathcal{C}\left([0,T]\times \R_y^d, \mathcal{C}^{k+1}(\R^d_x)\cap W^{k+1,\infty}(\R^d_x) \right),\label{psiaf}\\
   0<\underline{\psi_g}\leqslant \psi_g&\in \mathcal{C}\left([0,T]\times \R_x, W^{1,\infty}(\R_y^d)\right) \cap \mathcal{C}\left([0,T]\times \R_y, \mathcal{C}^\kappa(\R^d_x)\cap W^{\kappa,\infty}(\R^d_x)\right),\label{psigf}\\
    \psi_d&\in \mathcal{C}\left([0,T]\times \R^d_x, W^{1,\infty}(\R^d_y)\right) \cap   \mathcal{C}\left([0,T]\times \R^d_y, \mathcal{C}^\kappa(\R^d_x)\cap W^{\kappa,\infty}(\R^d_x) \right).\label{psidf}
\end{align}
As in Section \ref{PDE} we introduce
\begin{align}
    a&\in \mathcal{C}([0,T],W^{k+1,\infty}(\mathbb{R}^{d+1})),\label{aetR1f}\\
    R&\in \mathcal{C}\left([0,T]\times \R_x, W^{\kappa,\infty}_{loc}(\R^d_y)\right) \cap  \mathcal{C}\left([0,T]\times \R_y, \mathcal{C}^\kappa(\R^d_x)\cap W^{\kappa,\infty}(\R^d_x))\right) \label{aetR2f}
\end{align}
We consider as well 
\begin{equation}
    0\leqslant m\in \mathcal{C}\left([0,T]\times \R^d_x \times \R^d_y, W^{\kappa,\infty}(\R_I)\right) \cap \mathcal{C}\left([0,T]\times \R^d_x \times \R^d_I, W^{\kappa,\infty}(\R_y^d)\right) \cap \mathcal{C}\left([0,T]\times \R^d_y \times \R^d_I, \mathcal{C}^{\kappa}_c(\R^d_x) \right).\label{mf}
\end{equation}
satisfying hypothesis \eqref{mLip} through \eqref{GvsM}.\\
Finally, we consider $v^0\in W^{k,1}(\R^d)$ if $\partial_Ia=0$ and $v^0\in W^{k,1}(\R^d)\cap W^{k,\infty}(\R^d)$ with compact support otherwise.
\subsection{Convergence on a finite time interval}\label{ConvFT}
We recall that the function $v$ represents the solution of problem \eqref{Pv} while $x_i$, $w_i$ and $\nu_i$, $i\in \mathcal{J}_h$ represents that of problem \eqref{Oh}. We recall as well that
\[\max\{\|v\|,\|\nu w\|_{1,h}\}\leqslant\max\{\|v^0\|_{L^1(\R^d)},\|v^0h^d\|_{\mathcal{l}^1},\frac{I^*}{\underline{\psi_g}}\}=:\overline{\rho}.\nonumber\]

Let $\e>0$, $r\in\R$ and $\varphi\in \mathcal{C}_c(\R^d)$ satisfy the following conditions
\begin{align}
\int_{\R^d}{\varphi(x)dx}&=1, \label{cut off function  integrable}\\
\int_{\R^d}{x^\alpha \varphi(x)dx}&=0, \quad \forall \; \alpha \in \N^n, \lv \alpha \rv \leqslant r-1. \label{cut off function moment}
\end{align}
. We define, for all $t\in [0, T]$, $x\in \R^d$
\begin{enumerate}
    \item[{\it i)}] 
    \begin{equation}
    v^h(t,x)=\sum\limits_{i\in\mathcal{J}_h}{\nu_i(t)w_i(t)\delta(x-x_i(t))},\label{vh}
    \end{equation}
    a time dependent measure obtained as a sum of weighted Dirac deltas at $x_i(t)$,
    \item[{\it ii)}] 
    \begin{equation}
        v_\e^h(t,x)=(v^h(t)*\varphi_\e)(x)=\sum\limits_{i\in\mathcal{J}_h}{\nu_i(t)w_i(t)\varphi_\varepsilon(x-x_i(t))},\label{veh}
    \end{equation} 
    a regular function obtained as the space convolution of  $v^h(t,x)$ and $\varphi_\e(x)$, where
    \[\varphi_\e:=\frac{1}{\e^d}\varphi(\frac{\cdot }{\e}).\]
\end{enumerate}
We also introduce the following operator, for any function $v\in L^\infty(0,T; L^1(\R^d))$: 
\[
\biggl(\Pi_\e^h(t)v\biggr)(x)=\sum\limits_{i\in\mathcal{J}_h}{w_i(t)v(t,x_i(t))\varphi_\e(x-x_i(t))}\nonumber.
\]
We recall a direct corollary of the \textbf{Theorem 3} in \cite{mas1987particle}:
\begin{prop}\label{varphi}
Let $k$, $r$ be two integers, with $k>d$,  and let us assume that $a\in L^{\infty}\left(0,T;  W^{k+1, \infty}(\R^d)^d\right)$, and that $\varphi \in \mathcal{C}_c^1(\R^d)\cap W^{k+1, 1}(\R^d)$ satisfies conditions \eqref{cut off function  integrable} and \eqref{cut off function moment}. Then, for any $p\in [1, +\infty]$, there exists $C=C(T)>0$ such that, for any $u\in W^{\mu, 1}(\R^d)$ ($\mu= \max(r,k)$), 
\begin{align*}
    \lV u-{\Pi}_\e^h(t)u\rV_{L^p(\R^d)}\leqslant C \big(\e^r\lV u\rV_{W^{r,p}(\R^d)}+\big(\frac{\e}{h}\big)^k \lV u\rV_{W^{k,p}(\R^d)}\big).
\end{align*}
\label{u-pi u}
\end{prop}
We seek to prove the following approximation result between $v_\e^h$ and $v$, the solution of problem \eqref{Pv}.
\begin{theorem}\label{Theo6}
 Assume that hypotheses \eqref{psiaf} through \eqref{mf} are satisfied, and that $\varphi \in \mathcal{C}_c^1(\R^d)\cap W^{k+1, 1}(\R^d)$ satisfies \eqref{cut off function  integrable} and \eqref{cut off function moment}. 
Then, there exists $C=C(T, a, R, m,\overline{\rho})>0$, a positive constant which depends on $T$, $a$, $R$, $m$ and $\overline{\rho}$ such that 
\[\lVert v-v_\e^h \rVert_{L^1(\R^d)}\leqslant C \bigl(\e^r+\big(\frac{h}{\e}\big)^{\kappa}+h^{\kappa}\bigr)\lV v^0 \rV_{W^{\mu,1}(\R^d)},\quad \forall \; 0\leqslant t\leqslant T,   \]
where $\mu= \max(r, \kappa)$. 
\label{theo estimation }
\end{theorem}
The proof of Theorem \ref{Theo6} strongly relies on Proposition \ref{varphi} and the following result:
\begin{prop}\label{Prox}
Under hypotheses \eqref{psiaf} through \eqref{mf}, there exists a constant $C_T>0$, depending only on $T$ and on the parameters of problems \eqref{Pv} and \eqref{Oh}, such that their respective solutions satisfy, for all $t\in [0, T]$, 
\begin{equation}
    \sum\limits_{i\in\mathcal{J}_h}|v(t,x_i(t))-\nu_i(t)|w_i(t)\leqslant C_Th^{k-1}\lV v^0 \rV_{W^{\mu,1}(\R^d)}.\label{ConvRes}
\end{equation}
\label{v-v_k}
\end{prop}
\begin{proof}
Consider $\beta_{\varepsilon}$ as in \eqref{renorm}.  We define $e=\{e_i(\cdot)\}_{i\in\mathcal{J}}$ where for all $i\in \mathcal{J}$ and $t\in [0,T]$, 
\begin{align*}
    e_i(t)&:=v(t,x_i(t))-\nu_i(t),\\
    e_{\varepsilon,i}(t)&:=\beta_{\varepsilon}(e_i(t))w_i(t),
\end{align*}
and compute
\begin{equation}
    \dot e_{\varepsilon,i}(t)=\beta'_{\varepsilon}(e_i(t))\dot e_i(t)w_i(t)+\beta_{\varepsilon}(e_i(t)) \dot w_i(t).\label{eprime}
\end{equation}
We recall that
\begin{align*}
    \dot e_i(t)=&\biggl(G_{\nu,w}(t,x_i(t))-\mathcal{G}_{v}(t,x_i(t))\biggr)\nu_i(t)+\biggl(A_{\nu,w}(t,x_i(t))-\mathcal{A}_v(t,x_i(t))\biggr)\nabla v(t,x_i(t))\\
    &-\mathcal{G}_{v}(t,x_i(t))e_i(t)+\Delta\mathcal{M}(t,x_i(t)),
\end{align*}
where
\begin{align*}
    G_{\nu,w}(t,x_i(t))&:=\divv A_{\nu,w}(t,x_i(t))-R(t,x_i(t),I_g(t,x_i(t),\nu,w)),\\
    \mathcal{G}_{v}(t,x_i(t))&:=
    \divv \mathcal{A}_v(t,x_i(t))+R(t,x_i(t),(I_gv)(t,x_i(t))),\\
   \Delta\mathcal{M}(t,x_i(t))&:=\int\limits_{\R^d}{m(t,x_i(t),y,(I_dv)(t,x_i))v(t,y)dy}-\sum\limits_{j\in\mathcal{J}_h}w_j(t)\nu_j(t)m(t,x_i(t),x_j(t),I_d(t,x_i(t),\nu,w)).
\end{align*}
The functions $a$ and $R$ being Lipschitz, there exists a constant $C$ depending on the parameters of the problem and the value $\overline{\rho}$, such that
\begin{align*}
    \biggl|G_{\nu,w}(t,x_i(t))-\mathcal{G}_{v}(t,x_i(t))\biggr|\leqslant& C\biggl(|(I_av)(t,x_i(t))-I_a(t,x_i(t),\nu,w)|\\
    &+|\partial_x(I_av)(t,x_i(t))-\partial_xI_a(t,x_i(t),\nu,w)|
    \\&+|(I_gv)(t,x_i(t))-I_g(t,x_i(t),\nu,w)|\biggr).
\end{align*}
Notice that
\begin{align*}
    |(I_av)(t,x_i(t))-I_a(t,x_i(t),\nu,w)|=&\biggl|\int\limits_{\R^d}\psi_a(t,x_i(t),y)v(y)dy-\sum\limits_{j\in\mathcal{J}_h}\psi_a(t,x_i(t),x_j(t))\nu_j(t)w_j(t)\biggr|\\
    \leqslant &\biggl|\int\limits_{\R^d}\psi_a(t,x_i(t),y)v(y)dy-\sum\limits_{j\in\mathcal{J}_h}\psi_a(t,x_i(t),x_j(t))v(t,x_j(t))w_j(t)\biggr|\\
    &+\biggl|\sum\limits_{j\in\mathcal{J}_h}\psi_a(t,x_i(t),x_j(t))\biggl(v(t,x_j(t)-\nu_j(t))\biggr)w_j(t)\biggr|\\
    \leqslant &C\biggl(h^{\kappa}\|v\|_{W^{\kappa,1}(\R^d)}+\sum\limits_{j\in\mathcal{J}_h}|e_j(t)|w_j(t)\biggr),
\end{align*}
where in the last line we have used Lemma \ref{AT} and the $W^{\kappa,1}(\R^d)$ regularity of $v$. Similar results are true for $|\partial_x(I_av)(t,x_i(t))-\partial_xI_a(t,x_i(t),\nu,w)|$ and $|(I_gv)(t,x_i(t))-I_g(t,x_i(t),\nu,w)|$. In conclusion, thanks to \eqref{BWm1}, there exists a constant $C_T$, only depending on $T$, the parameters of the problem and the value $\overline{\rho}$, such that
\begin{align}
    \biggl|G_{\nu,w}(t,x_i(t))-\mathcal{G}_{v}(t,x_i(t))\biggr|\leqslant C_T\biggl(h^{\kappa}\|v^0\|_{W^{\kappa,1}(\R^d)}+\sum\limits_{j\in\mathcal{J}_h}|e_j(t)|w_j(t)\biggr).\label{Prox1}
\end{align}
Again, using the Lipschitz regularity of $a$, we see that
\begin{align}
    \biggl|A_{\nu,w}(t,x_i(t))-\mathcal{A}_v(t,x_i(t))\biggr|\leqslant & C|(I_av)(t,x_i(t))-I_a(t,x_i(t),\nu,w)|\nonumber\\
    \leqslant &C_T\biggl(h^{\kappa}\|v^0\|_{W^{\kappa,1}(\R^d)}+\sum\limits_{j\in\mathcal{J}_h}|e_j(t)|w_j(t)\biggr).\label{Prox2}
\end{align}
The boundedness of $a$ and $R$ implies the existence of a constant $\overline{G}$ such that for all $i\in \mathcal{J}$ and $t\in [0,T]$, 
\begin{equation}
    |\mathcal{G}_{v}(t,x_i(t))|\leqslant \overline{G}.\label{Prox3}
\end{equation}
We recall from the previous section that 
\[
\Delta\mathcal{M}(t,x_i(t))=0,\ \forall i\not\in \mathcal{J}^m_h,
\]
where the set of indexes $\mathcal{J}^m_h$ has a finite number of elements, which depends on $T$. On the other hand, for those $i\in\mathcal{J}^m_h $, we have
\begin{align*}
&|\Delta\mathcal{M}(t,x_i(t))|\\
    \leqslant &\biggl|\int\limits_{\R^d}{m(t,x_i(t),y,(I_dv)(t,x_i(t)))v(t,y)dy}-\sum\limits_{j\in\mathcal{J}_h}w_j(t)v(t,x_j(t))m(t,x_i(t),x_j(t),(I_dv)(t,x_i(t)))\biggr|\\
    &+\sum\limits_{j\in\mathcal{J}_h}w_j(t)v(t,x_j(t))\biggl|m(t,x_i(t),x_j(t),(I_dv)(t,x_i(t)))-m(t,x_i(t),x_j(t),I_d(t,x_i(t),\nu,w))\biggr|\\
    &+\sum\limits_{j\in\mathcal{J}_h}w_j(t)|e_j(t)|m(t,x_i(t),x_j(t),(I_dv)(t,x_i(t)))\\
    \leqslant &C\biggl(h^{\kappa}\|v\|_{W^{\kappa,1}(\R^d)}+\mu|(I_dv)(t,x_i(t))-I_d(t,x_i(t),\nu,w)|\sum\limits_{j\in\mathcal{J}_h}w_j(t)v(t,x_j(t))+\overline{M}\sum\limits_{j\in\mathcal{J}_h}|e_j(t)|w_j(t)\biggr),
\end{align*}
where we have used again Lemma \ref{AT}, the $W^{\kappa,1}(\R^d)$ regularity of $m(t,x,y,I)v(t,y)$ with respect to the $y$ variable, and the Lipschitz regularity of $m$. Furthermore, Lemma \ref{AT} gives us the bound
\[
\sum\limits_{j\in\mathcal{J}_h}w_j(t)v(t,x_j(t))\leqslant \|v\|_{L^{1}(\R^d)}+Ch^{\kappa}\|v\|_{W^{\kappa,1}(\R^d)},
\]
which together with manipulations similar to those made for $|(I_av)(t,x_i(t))-I_a(t,x_i(t),\nu,w)|$, and the bound \eqref{BWm1}, gives 
\begin{equation}
    |\Delta\mathcal{M}(t,x_i(t))|\leqslant C_T\biggl(h^{\kappa}\|v^0\|_{W^{\kappa,1}(\R^d)}+\sum\limits_{j\in\mathcal{J}_h}|e_j(t)|w_j(t)\biggr).\label{Prox4}
\end{equation}
We denote as $\mathcal{K}$ an arbitrary finite subset of $\mathcal{J}_h$. If we add \eqref{eprime} for all values of $i\in\mathcal{K}$, and use bounds \eqref{Prox1} through \eqref{Prox4}, together with the equation for $w_i$, we get
\begin{align*}
    \left(\sum\limits_{i\in\mathcal{K}}e_{\varepsilon,i}(t)\right)'\leqslant &C_TB(t)\biggl(h^{\kappa}\|v^0\|_{W^{\kappa,1}(\R^d)}+\sum\limits_{j\in\mathcal{J}_h}|e_j(t)|w_j(t)\biggr),
\end{align*}
where
\[
B(t):=\sum\limits_{i\in\mathcal{K}}(\nu_i(t)+|\nabla v(t,x_i(t))|)w_i(t)+\overline{G}+|\mathcal{K}\cap \mathcal{J}^m_h|h^de^{\tilde{a}T}+\tilde{a}.
\]
Given that
\[
\sum\limits_{i\in\mathcal{K}}\nu_i(t)w_i(t)\leqslant \overline{\rho},
\]
and
\[
\sum\limits_{i\in\mathcal{K}}w_j(t)|\nabla v(t,x_j(t))|\leqslant \|\nabla v\|_{L^{1}(\R^d)}+Ch^{\kappa-1}\|\nabla v\|_{W^{\kappa-1,1}(\R^d)},
\]
thanks to \eqref{BWm1}, we conclude that there exists a constant $B_T$, independent of the choice of $\mathcal{K}$ and $h$, such that $B(t)\leqslant B_T$. Consequently, for all values of $t\in[0,T]$, $h$ small enough and any finite subset of $\mathcal{J}_h$,  we have the relation
\[
\left(\sum\limits_{i\in\mathcal{K}}e_{\varepsilon,i}(t)\right)'\leqslant C_T\biggl(h^{\kappa}+\sum\limits_{j\in\mathcal{J}_h}|e_j(t)|w_j(t)\biggr).
\]
Integrating between $0$ and $t$, taking the limit when $\varepsilon$ goes to zero, and using Gr\"onwall's lemma, we obtain that there exists a constant $C_T$, independent of $\mathcal{K}$ such that
\[
\sum\limits_{i\in\mathcal{K}}|e_{i}(t)|w_i(t)\leqslant C_T h^{\kappa}\|v^0\|_{W^{\kappa,1}(\R^d)}.
\]
Being $C_T$ independent of $\mathcal{K}$ and $h$, \eqref{ConvRes} is immediate.
\end{proof}
In other words, we proved in Proposition \ref{Prox} that the piece-wise constant functions that take values $v(t,x_k(t))$ and $v_k(t)$ respectively over the intervals $\Omega_k(t)$ are close in $L^1(\R^d)$.\\
\begin{proof}[Proof of Theorem \ref{Theo6}.]
According to the triangle inequality, \begin{align}
\lVert v-v_\e^h\rVert_{L^1(\R^d)}\leqslant \lV v-{\Pi}_\e^h(t)v\rV_{L^1(\R^d)} + \lV{\Pi}_\e^h(t)v-v_\e^h\rV_{L^1(\R^d)}, 
\label{triangle inequality}
\end{align}
it only remains to bound both terms on the right hand side.
\begin{itemize}
\item[{\emph i)}] 
According to Proposition \ref{u-pi u} with $p=1$, and bound \eqref{BWm1} 
\begin{align*}
\lV v-{\Pi}_\e^h(t)v\rV_{L^1(\R^d)}\leqslant C(\e^r+\big(\frac{\e}{h}\big)^{\kappa})\lV v\rV_{\mu, p}\leqslant C_T(\e^r+\big(\frac{\e}{h}\big)^{\kappa})\lV v^0\rV_{\mu, p}.
\end{align*}
\item[{\emph ii)}]On the other hand, one computes 
\begin{align*}
\lV{\Pi}_\e^h(t)v-v_\e^h\rV_{L^1(\R^d)}&=\int_{\R^d}{\big\lvert  \sum\limits_{i\in\mathcal{J}_h}{w_i(t)\varphi_\e(x-x_i(t))\big(v(t,x_i(t))-\nu_i(t)\big)}  \big\rvert dx} \\
&\leqslant \sum\limits_{i\in\mathcal{J}_h}{\biggl(w_i(t) \lvert v(t,x_i(t))-\nu_i(t) \rvert\int_{\R^d}{\lvert \varphi_\e(x-x_i(t))\rvert dx}\biggr)}.
\end{align*}
According to the definition of $\varphi_\e$, with the change of variable $x'=\frac{x-x_k(t)}{\varepsilon}$, we note that 
\[\int_{\R^d}{\lvert \varphi_\e(x-x_k(t))\rvert dx}= \int_{\R^d}{\lvert \varphi(x)\rvert dx}<+\infty, \]
by hypothesis on $\varphi$. 
We have then, according to Proposition \ref{v-v_k}, that 
\[\lV{\Pi}_\e^h(t)v-v_\e^h\rV_{L^1(\R^d)}\leqslant C_T h^{\kappa} \lV v^0\rV_{\kappa, 1}\leqslant  C h^{\kappa} \lV v^0\rV_{\mu, 1},\]
which concludes the proof of Theorem \ref{Theo6}. 
\end{itemize}
\end{proof}
\subsection{Asymptotic preserving properties}
\label{conv infinite time}
The study of the asymptotic behaviour of the solution for adaptive dynamics models, such as \eqref{Pv}, is one of the main interests often treated in the literature (see \cite{barles2009concentration,bonnefon2015concentration,calsina2013asymptotics,cooney2022long,coville2013convergence,desvillettes2008selection,friedman2009asymptotic,guilberteau2023long,jabin2011selection,lorenzi2020asymptotic}). For this reason, the design of numerical methods which preserve the asymptotic behaviour, or at least, the identification of the problems for which the asymptotics are preserved under a certain numerical scheme, is a priority. In other words, given that $v(t,\cdot)$ converges to a measure $\mu$ when $t$ goes to infinity, we expect to identify the conditions under which $\lim\limits_{h \to 0} \lim\limits_{t \to +\infty} v^h_{\e(h)}(t,\cdot) =  \mu$, that is, ensuring the commutativity of diagram \eqref{diagram}:
\begin{equation}
\begin{matrix}
    v(t,\cdot)&\xyrightarrow[t\rightarrow\infty]{}&\mu\\
    &&\\
   \hspace{-0.6cm}{\scriptstyle  h\rightarrow 0}\xuparrow{0.5cm} &&\hspace{0.6cm}\xuparrow{0.5cm}{\scriptstyle  h\rightarrow 0}\\
    &&\\
    v^h_{\varepsilon(h)(t)}&\xyrightarrow[t\rightarrow\infty]{}&\mu_h
\end{matrix}\label{diagram}
\end{equation}
In what follows, we formally define the concept of an asymptotic preserving approximation, and give examples and counter-examples of this concept.\\

We recall that, according to the Riesz representation theorem,  the space of finite Radon measures can be identified with the topological dual space of $\mathcal{C}_c(\R^d)$. Hence, we say that a sequence of finite Radon measures $\{\mu_n\}_{n\in \mathbb{N}}$ converges weakly to a finite Radon measure $\mu$ (denoted $\mu_n \underset{n\to +\infty}{\rightharpoonup} \mu$) if for all $\phi\in \mathcal{C}_c(\R^d)$, 

\[\int_{\R^d}{\phi(x)d\mu_n(x)}\underset{n\to +\infty}{\longrightarrow}  \int_{\R^d}{\phi(x)d\mu(x)}. \]

This leads us to introduce the following definition:

\begin{definition}
We say that the particle solution $v^{h}_{\varepsilon}$ defined in \eqref{veh} is an \emph{asymptotic preserving approximation} of $v$, the solution to \eqref{Pv}, if for all $\e:(0,1]\to \R_+^*$ which converges to $0$ when $h$ goes to $0$,  and all $\phi \in \mathcal{C}_c(\R^d)$,
\[\underset{t\to +\infty}{\limsup}\left\lvert \int_{\R^d}{\phi(x) v_{\e(h)}^h(t,x)dx}-\int_{\R^d}{\phi(x) v(t,x)dx} \right\rvert \underset{h\to 0}{\longrightarrow}0.  \]
\end{definition}

The following lemma ensures that, in the previous definition, $v_{\e(h)}^h$, introduced in \eqref{veh}, can be replaced by $v^h$, introduced in \eqref{vh}. 

\begin{lem}
The function $v_\e^h$ is an asymptotic preserving approximation of $v$ if and only if 
\[\underset{t\to +\infty}{\limsup}\left\lvert \int_{\R^d}{\phi(x) dv^h(t,x)}-\int_{\R^d}{\phi(x) v(t,x)dx} \right\rvert \underset{h\to 0}{\longrightarrow}0. \]
\label{lem condition asymptotic preserving}
\end{lem}

\begin{proof}
Let us prove that for all $\e:(0,1]\to \R_+^*$ which converges to $0$ as $h$ goes to $0$, and all $\phi \in \mathcal{C}_c(\R^d)$,  
\[ \underset{t\to +\infty}{\limsup} \left\lvert \int_{\R^d}{\phi(x)v^h_{\e(h)}(t,x)dx}-\int_{\R^d}{\phi(x)dv^h(t,x)} \right\rvert \underset{h\to 0}{\longrightarrow}0.  \]

Let $h>0$. According to the definitions of $v^h$ and $v^h_{\e}$, and since $\int_{\R^d}{\varphi_{\e(h)}(x-x_i(t))dx}=1$ for all $i\in \mathcal{J}_h$ we get, for all $t\geqslant 0$, 
\begin{align*}
\left\lvert \int_{\R^d}{\phi(x)v^h_{\e(h)}(t,x)dx}-\int_{\R^d}{\phi(x)dv^h(t,x)} \right\rvert
=& \left\lvert  \int_{\R^d}{\sum\limits_{i\in\mathcal{J}_h}{\nu_i(t)w_i(t)\big(\phi(x)-\phi(x_i(t))\big)\varphi_{\e(h)}(x-x_i(t))}dx}\right\rvert\\
&\leqslant \sum\limits_{i\in\mathcal{J}_h}{\lvert \nu_i(t)w_i(t)\rvert  \int_{\R^d}{\big\lvert \big(\phi(x)-\phi(x_i(t))\big) \varphi_{\e(h)}(x-x_i(t))\big\rvert dx}}.
\end{align*}
With the change of variable `$y=\frac{x-x_i(t)}{\e(h)}$', we get, for all $i\in \mathcal{J}_h$, 
\[\int_{\R^d}{\big\lvert \big(\phi(x)-\phi(x_i(t))\big) \varphi_{\e(h)}(x-x_i(t))\big\rvert dx}= \int_{K}{\big\lvert \big(\phi(\e(h) y+x_i(t))-\phi(x_i(t))\big) \varphi(y) \big\rvert dy},  \]
where $K$ is the support of $\varphi$.
Let $\eta>0$. Since $\phi$ is continuous with a compact support, and thus uniformly continuous, then $\lvert \phi(\e(h)x+x_i(t))-\phi(x_i(t)) \vert \leqslant \eta$ for all $i\in \mathcal{J}_h$, $x\in K$, $t\geqslant 0$ and any $h$ small enough. 
Therefore, for any $h$ small enough, 
\[\left\lvert \int_{\R^d}{\phi(x)v^h_{\e(h)}(t,x)dx}-\int_{\R^d}{\phi(x)dv^h(t,x)} \right\rvert\leqslant \eta \sum\limits_{i\in\mathcal{J}_h}{\nu_i(t)w_i(t)},\]
which concludes the proof, since there exists $\overline{\rho}>0$ such that $0\leqslant \sum\limits_{i\in\mathcal{J}_h}{\nu_i(t)w_i(t)}\leqslant \overline{\rho}$ for all $h>0$, $t\geqslant 0$,  as proved in Theorem \ref{theo bounds particles}. 
\end{proof}

The problem of determining if $v_\e^h$ is an asymptotic preserving approximation of $v$ is generally a difficult question. In what follows, we deal with cases where we are able to determine the asymptotic behaviour of both $v$ and $v^h$, and we check if the necessary and sufficient condition from Lemma \ref{lem condition asymptotic preserving} holds. From now on, we assume that $a$ is local and not time dependent, $\textit{i.e.}$ $a(t,x,I)=a(x)$ and that that the functions $m$, $R$, $\psi_g$ and $\psi_d$ are not time-dependent. We assume as well that the function $(x,y,I)\mapsto m(x,y,I)$ is not only uniformly compactly supported as a function of the $x$ variable, but relative to the $y$ variable as well. That is, there exist two compact sets $K_x$ and $K_{y}$ such that
$\sup\limits_{y,I}m(x,y,I)=0$ for all $y$ outside of $K_x$ and $\sup\limits_{x,I}m(x,y,I)=0$ for all $x$ outside of $K_y$. We denote $K_{xy}:=K_x\cup K_y$. Finally, we assume $\psi_g$ and $\psi_d$ to be compactly supported as functions of the $y$ variable. 
\subsubsection{Necessary conditions of convergence towards a Radon measure}
With the help of necessary conditions, we would be able to rule out those cases where $v$ does not converge towards certain types of Radon measures, which are the object of our interest. We start by giving a general result, involving the necessary conditions of convergence towards any Radon measure.  

\begin{lem}
Let us assume that $v(t, \cdot) \underset{t\to +\infty}{\rightharpoonup} \mu$ in the weak sense in the space of finite Radon measures. Then, for $\gamma \in\{g, d\}$, 
\[I_\alpha(t,x)\underset{t\to +\infty}{\longrightarrow} \int_{\R^d}{\psi_\alpha(x,y)d\mu(y)}=:\overline{I_\alpha}(x)\quad \forall x\in \R^d, \]
and for all $\phi \in \mathcal{C}_0^1(\R^d)$, 

\[\int_{\R^d}{\bigg(a(x)\cdot\nabla \phi(x)+R\left(x,\overline{I_g}(x)\right)\phi(x)\bigg)d\mu(x)}+ \int_{\R^d}{\left(\int_{\R^d}{m\left(x,y, \overline{I_d}(x)\right)}\phi(x)dx\right)}d\mu(y)=0. \]
\label{lem necessary condition conv measure}
\end{lem}

\begin{proof}
Let us assume that $v(t, \cdot) \underset{t\to +\infty}{\rightharpoonup} \mu$.  By definition of the weak convergence in the space of finite Radon measure, for any $\phi \in \mathcal{C}_c^1(\R^d)$, 
\[\int_{\R^d}{\phi(x)v(t,x)dx}\underset{t\to +\infty}{\rightarrow}\int_{\R^d}{\phi(x)d\mu(x)}.\]
The first identity is thus a direct consequence of the definition of the weak convergence, applied with $\phi= \psi_\gamma(x, \cdot)$, $\gamma \in\{g, d\}$. 

Let $\phi \in \mathcal{C}_c(\R^d)$. One computes

\begin{align*}
\frac{d}{dt}\int_{\R^d}{\phi(x)v(t,x)dx}=& -\int_{\R^d}{\phi(x) \nabla \cdot \left(a(x)v(t,x)\right)} + \int_{\R^d}{R(x, I_g(t,x))\phi(x)v(t,x)dx}\\
&+ \int_{\R^d}{\left(\int_{\R^d}{m(x,y,I_d(t,x))v(t,y)}\right)\phi(x)dx}\\
=&\int_{\R^d}{\bigg( a(x)\cdot\nabla \phi(x)+R(x,I_g(t,x))\phi(x)\bigg)v(t,x)dx}\\
&+ \int_{\R^d}{\bigg(\int_{\R^d}{m(x,y,I_d(t,x))\phi(x)dx}\bigg)v(t,y)dy}\\
\underset{t \to +\infty}{\longrightarrow}& \int_{\R^d}{\bigg(a(x)\cdot \nabla \phi(x)+R\left(x,\overline{I_g}(x)\right)\phi(x)\bigg)d\mu(x)}\\
&+ \int_{\R^d}{\left(\int_{\R^d}{m\left(x,y, \overline{I_d}(x)\right)\phi(x)dx}\right)d\mu(y)}.
\end{align*}
Thus, $t\mapsto \int_{\R^d}{\phi(x)v(t,x)dx}$ is a convergent function with a convergent derivative, which ensures that the limit of its derivative is zero, which concludes the proof. 
\end{proof}

The following proposition provides a necessary condition for the convergence to a sum of Dirac  masses. 

\begin{prop}
Let us assume that $v(t,\cdot) \underset{t\to +\infty}{\rightharpoonup}\sum\limits_{i=1}^N{C_i\delta_{x_i}}$, with $x_1,\ldots, x_N\in \R^d$, $C_1,\ldots, C_N>0$. Then, for $\alpha=g,d$, $I_\alpha(t, x)\underset{t\to +\infty}{\rightarrow} \sum\limits_{i=1}^N{C_i\psi_\alpha(x, x_i)}=:\overline{I_\alpha}( x)$. Moreover, for all $i\in \{1,\ldots,N\}$, $a(x_i)=0$, $R\left(x_i, \overline{I_g}(x_i)\right)=0$, and for all $x\in \R^d,$ $m(x, x_i, \overline{I_d}(x))=0$. 
\label{Conditions convergence dirac}
\end{prop}

\begin{proof}
According to the previous lemma, for all $\phi\in \mathcal{C}_c^1(\R^d)$, 
\[\sum\limits_{i=1}^N{C_i a(x_i)\cdot\nabla \phi(x_i)} + \sum\limits_{i=1}^N{C_i R\left(x_i, \overline{I_g}(x_i)\right)\phi(x_i)}+ \sum\limits_{i=1}^N{C_i\int_{\R^d}{m\left(x, x_i, \overline{I_d}(x)\right)\phi(x)dx}}=0.\]

Let $\varepsilon>0$. For any non-negative function  $\phi_\varepsilon\in \mathcal{C}_c^1(\R^d)$ with a support on $ \R^d\backslash \bigcup_{i=1}^N{B(x_i, \varepsilon)} $, we have that
\[\int_{\R^d}{\sum\limits_{i=1}^N{C_i m\left(x, x_i, \overline{I_d}(x)\right) \phi_\varepsilon(x)}dx}=0,\]
which proves that $x\mapsto \sum\limits_{i=1}^N{C_i m\left(x, x_i, \overline{I_d}(x)\right)}$ is $0$ on $ \R^d\backslash \bigcup_{i=1}^N{B(x_i, \varepsilon)} $. Since $m$ is a non-negative function, for all $i\in \{1,... N\}$,  $x \mapsto m\left(x, x_i, \overline{I_d}(x)\right)$ is $0$ as well on $ \R^d\backslash \bigcup_{i=1}^N{B(x_i, \varepsilon)}$, and therefore on $\R^d$, since the result holds for any $\varepsilon>0$. Hence, 
\[\sum\limits_{i=1}^N{C_i a(x_i)\cdot\nabla \phi(x_i)} + \sum\limits_{i=1}^N{C_i R\left(x_i, \overline{I_g}(x_i)\right)\phi(x_i)}=0,\] 
for any $\phi \in \mathcal{C}_c^1(\R^d)$. By choosing $\phi^1_j$ such that 
\[
\phi^1_j(x_i)=\delta_{ij}\mbox{ and } \nabla\phi^1_j(x_i)=0,\ i,j\in \{1,... N\},
\]
where $\delta_{ij}$ represents the Kronecker delta, one proves that $R\left(x_i, \overline{I_g}(x_i)\right)=0$ for all $i\in \{1,... N\}$.\\
Finally, by choosing $\phi^2_{ij}$ such that
\[
\nabla\phi^2_{jl}(x_i)=\delta_{ij} e_l,\ i,j\in \{1,... N\},\ l\in\{1,... d\},
\]
where $\{e_l\}_{l=1}^d$ represents the euclidean basis in $\R^d$, one proves that $a(x_i)=0$, for any $i\in \{1,\ldots,N\}$.  
\end{proof}
\subsubsection{Limit identification and asymptotic preserving approximations}
In some cases, it is possible to guarantee the existence of a limit for $v^h$ and identify it.\\
Assume that there exists $\hat{x}\in \R^d$ an asymptotically stable equilibrium for the ODE `$\dot x =a(x)$' and that there exists $C, \delta >0$  such that
\begin{equation}
\forall y \in \mbox{supp}(n^0), t\geq 0, \quad \lVert X(t,y) -\hat{x}\rVert \leq Ce^{-\delta t}
\label{hypotheses exponential convergence}.
\end{equation}
Moreover, let us assume that there exist positive values $D$, $I^m$ and $I^M$ such that 
\begin{equation}
   R(x, I^m)\geqslant 0 ,\quad  R(x, I^M)\leqslant 0\mbox{ and }\partial_IR(x,I)\leqslant -D,\quad\forall x\in \textrm{supp}(v^0)\label{hypotheses R}.
\end{equation}
Then, we can compute the limit of $v^h$ when $t$ goes to $+\infty$, whatever the value of $m$, as stated in the following proposition.

\begin{prop}\label{conv vh unique equilibrium}
Let us assume that $\textrm{supp}\left(v^0\right)$ is a compact set such that hypotheses \eqref{hypotheses exponential convergence} and \eqref{hypotheses R} hold. 
Then, $v^h$ converges to $\hat{\rho}_h \: \delta_{\hat{x}}$ in the weak sense in the space of Radon measures, where $\hat{\rho}_h$ is the unique solution of \[R(\hat{x}, \psi_g(\hat{x},\hat{x})\hat{\rho}_h)=0.\] 
\end{prop}
The following lemma, proved in Appendix \ref{ODEresults}, is required in the proof of this result.

\begin{lem}
Let $u\in \mathcal{C}^2\left(\R_+,  \R\right)$ be a bounded function, and let us assume that there exist $p_0>0$,  $p:\R_+\to \R_+$ a function which satisfies $p\geqslant p_0$ and $B\in L^1(\R_+)$ an integrable function such that 
\[\ddot u(t)\geqslant - p(t) \dot u(t)+B(t). \]
Then, there exists $u_{\infty}\in\R$ such that $\lim\limits_{t\rightarrow +\infty}u(t)=u_{\infty}$.
\label{convergence negative part}
\end{lem}

\begin{proof}[Proof of Proposition \ref{conv vh unique equilibrium}]
Given that $\textrm{supp}\left(v^0\right)\cup K_{xy}$ is a compact set which is strictly contained in the basin of attraction of $\hat{ x}$ we will have the existence of $\mathcal{J}^0_h\subset \mathcal{J}_h$, with $\lvert \mathcal{J}^0_h\rvert<+\infty$ such that $\nu^h(t)\not\equiv 0$ only for $i\in\mathcal{J}^0_h$.\\
Let us denote, for all $i\in \mathcal{J}^0_h$, $\alpha_i(t):=\nu_i(t)w_i(t)$,
and 
\[\rho_h(t):=\sum\limits_{i\in \mathcal{J}_h}{\alpha_i(t)}=\sum\limits_{i\in \mathcal{J}^0_h}{\alpha_i(t)}.\]
Let us note that, according to the hypotheses on $a$, for all $i\in \mathcal{J}^0_h$, $x_i(t)$ converges to $\hat{x}$. Thus, 
\[v^h(t)-\rho_h(t)\delta_{\hat{ x}}\underset{t\to +\infty}{\rightharpoonup}0. \]
Hence, it only remains to prove that $\rho_h$ converges to the expected limit. 
According to the definition of $\rho_h$, 
\begin{align}
\dot \rho_h(t)=\sum\limits_{i\in \mathcal{J}^0_h}{R(x_i(t), I_g(t,x_i(t)))\alpha_i(t) }+\underbrace{\sum\limits_{i,j \in \mathcal{J}^0_h}{w_i(t)m(x_i(t), x_j(t), I_d(t, x_i(t)))\alpha_j(t)}}_{:=\e(t)}. 
\label{derivative rhoN}
\end{align}
According to hypothesis \ref{hypotheses exponential convergence}, there exist $C, \delta>0$ such that 
\[\left\lVert x_i(t)-\hat{ x} \right\rVert \leqslant C e^{-\delta t},\quad \forall i\in \mathcal{J}^0_h, \quad \forall t\geqslant 0. \]
Thus, for all $i\in \mathcal{J}^0_h$, 
\[w_i(t)=e^{\int_0^t{\divv a(x_i(s)) ds}}w_i(0)\leqslant \underbrace{e^{\int_0^t{\divv a(x_i(s))-\divv a(\hat{ x})ds}}}_{\leqslant \tilde C}e^{\divv a(\hat{ x})t}w_i(0), \]
which proves, since $\divv a(\hat{x})<0$, that there exist $C', \delta'$ such that for all $t\geqslant 0,$
\begin{align}
0\leqslant \underset{i\in \mathcal{J}^0_h}{\max}\; w_i(t)\leqslant C'e^{-\delta' t}. 
    \label{w exponential conv}
\end{align}
In particular, it proves that $t \mapsto \varepsilon(t)$ defined in \eqref{derivative rhoN}, converges to zero with an exponential speed, since 
\begin{align*}
    \lvert \e(t)\rvert\leqslant \underset{i\in \mathcal{J}_h}{\max}\; w_i(t) \lvert \mathcal{J}^0_h \rvert \lVert m \rVert_{L^\infty} \rho_h(t), 
\end{align*}
and $\rho_h$ is bounded, according to Theorem \ref{theo bounds particles}. 

Moreover, according to the hypothesis on $\psi_g$, $\underline{\psi_g}\,\rho(t)\leqslant I_g(t, x_i(t))\leqslant \lVert \psi_g\rVert_\infty \, \rho(t) $. 
Thus, according to the hypotheses \eqref{hypotheses R} on $R$, the relation
\[\dot \rho_h(t)\geqslant \left(\underset{i\in \mathcal{J}^0_h}{\min}\left(R(x_i(t), I_g(t, x_i(t)))\right)\right)\, \rho_h(t)\]
implies that, as soon as $\rho_h$ becomes small, and so does $I_g$, $\rho_h$ becomes increasing, which proves that $\rho_h$ is lower bounded by a positive constant. Moreover, since $\rho_h$ and $\e$ are bounded, and
\begin{align*}
    \|\dot\alpha\|_{\mathcal{l}_1}=\sum\limits_{i\in\mathcal{J}^0_h}|\dot{\alpha}(t)|\leqslant& \sum\limits_{i\in\mathcal{J}^0_h}\left(2|\divv a(x_i(t))|+ |R(x_i(t), I_g(t, x_i(t)))|\right)\alpha_i(t)+\e(t)\\
    \leqslant&\left(2\|a\|_{W^{1,\infty}(\R^d)}+\overline{R}\right)\rho_h(t)+\e(t),
\end{align*}
where
\[
\overline{R}:=\max\limits_{t\in\R,i\in\mathcal{J}^0_h}|R(x_i(t), I_g(t, x_i(t)))|,
\]
then $\|\dot\alpha\|_{\mathcal{l}_1}$ is also bounded.

Now, let us prove that $\rho_h$ satisfies the equality of Lemma \ref{convergence negative part}. First, for $\gamma\in\{g,d\}$, we compute
\begin{align*}
\left\lvert \frac{d}{dt}I_\gamma(t, x_i(t))- \psi_\gamma(\hat{x}, \hat{x})\dot \rho_h(t)\right\lvert \leqslant & \left\lvert\sum\limits_{j\in \mathcal{J}^0_h}{\bigg(a(x_i(t))\partial_x\psi_\gamma(x_i(t), x_j(t))+a(x_j(t))\partial_y\psi_\gamma(x_i(t), x_j(t))\bigg)\alpha_j(t)}\right\rvert\\
&+\left\lvert \sum\limits_{j\in \mathcal{J}^0_h}{\big(\psi_\gamma(x_i(t), x_j(t))-\psi_\gamma(\hat{x}, \hat{x}) \big)\dot \alpha_j(t)}\right\rvert \\
\leqslant& 2\max\limits_{i\in \mathcal{J}^0_h}\lvert a(x_i(t)) \rvert \lVert\psi_\gamma  \rVert_{W^{1,\infty}(\R^d)}  \rho_h(t) + \max\limits_{i,j\in \mathcal{J}^0_h}\lvert \psi_\gamma(x_i(t), x_j(t))-\psi_\gamma(\hat{x}, \hat{x})\rvert  \|\dot\alpha\|_{\mathcal{l}_1}. 
\end{align*}
Since hypothesis \eqref{hypotheses exponential convergence}, is satisfied, the functions $t\mapsto  \max\limits_{i\in \mathcal{J}^0_h}\lvert a(x_i(t)) \rvert$ and $t\mapsto \max\limits_{i,j\in \mathcal{J}^0_h}\lvert \psi_\gamma(x_i(t), x_j(t))-\psi_\gamma(\hat{x}, \hat{x})  \rvert $ converge to zero with an exponential speed. Since $\vert \mathcal{J}^0_h \rvert<+\infty $, this proves that, for $\gamma\in\{g,d\}$,
\begin{align}
    \sum\limits_{i\in \mathcal{J}^0_h}{\left\lvert \frac{d}{dt}I_\gamma(t, x_i(t))-\psi_\gamma(\hat{x}, \hat{x})\dot \rho_h(t)\right\rvert }= O(e^{-\delta t}),
    \label{d/dt I_g(t,xi(t))}
\end{align}
for a certain $\delta >0$. 

Thus, by differentiating \eqref{derivative rhoN}, we get
\begin{align*}
\ddot \rho_h(t)=&\sum\limits_{i\in \mathcal{J}^0_h}{a(x_i(t))\partial_xR(x_i(t), I_g(t, x_i(t)))\alpha_i(t) }+\sum\limits_{i\in \mathcal{J}^0_h}{\left( \frac{d}{dt}I_g(t,x_i(t))\right)\partial_I R(x_i(t), I_g(t, x_i(t))) \alpha_i(t)} \\
&+\sum\limits_{i\in \mathcal{J}^0_h}{\dot \alpha_i(t)R(x_i(t), I_g(t, x_i(t)))}+\sum\limits_{i,j\in \mathcal{J}^0_h}{ w_i(t){m(x_i(t), x_j(t), I_d(t, x_i(t)))\dot\alpha_j(t)}}\\
&+\sum\limits_{i,j\in \mathcal{J}^0_h}{\dot w_i(t){m(x_i(t), x_j(t), I_d(t, x_i(t)))\alpha_j(t)}}
+\sum\limits_{i, j\in \mathcal{J}^0_h}{w_i(t)\frac{d}{dt}m(x_i(t), x_j(t), I_d(t, x_i(t)))\alpha_j(t)},
\end{align*}
where
\begin{itemize}
    \item[\emph{i)}] $\sum\limits_{i\in \mathcal{J}^0_h}{a(x_i(t))\partial_xR(x_i(t), I_g(t, x_i(t)))\alpha_i(t) }=O(e^{-\delta t})$, since $\max\limits_{i\in \mathcal{J}_h} |a(x_i(t))|$ converges to zero with exponential speed, and $\partial_x R$ and $\rho_h$ are bounded, 
    \item[\emph{ii)}] according to \eqref{d/dt I_g(t,xi(t))}$, \sum\limits_{i\in \mathcal{J}^0_h}{\left( \frac{d}{dt}I_g(t,x_i(t))\right)\partial_I R(x_i(t), I_g(t, x_i(t)))\alpha_i(t)}=-p(t)\dot \rho_h(t)+O(e^{-\delta t})$, where $$p(t)=-\psi_g(\hat{x} , \hat{x})\sum\limits_{i\in \mathcal{J}^0_h}{\partial_I R(x_i(t), I_g(t, x_i(t)))\alpha_i(t)}\geqslant D\underline{\psi_g}\min\limits_{t\geqslant 0}\, \rho_h(t)>0,$$
    \item[\emph{iii)}] 
    \begin{align*}
    \sum\limits_{i\in \mathcal{J}^0_h}{\dot \alpha_i(t)R(x_i(t), I_g(t, x_i(t)))}=&\underbrace{\sum\limits_{i\in \mathcal{J}^0_h}{R(x_i(t), I_g(t, x_i(t)))^2}\alpha_i(t)}_{:=P(t)\geqslant 0}\\
    &+\underbrace{\sum\limits_{i,j\in \mathcal{J}^0_h}{w_i(t) R(x_i(t), I_g(t, x_i(t))) \alpha_j(t)m(t, x_i(t), x_j(t), I_d(t, x_i(t) ))}}_{=O(e^{-\delta t})},    
    \end{align*}
    where the relation for the second term was proved thanks to the bound 
    $$\left\lvert \sum\limits_{i,j\in \mathcal{J}^0_h}{w_i(t) R(x_i(t), I_g(t, x_i(t))) \alpha_j(t)m(t, x_i(t), x_j(t), I_d(t, x_i(t) ))} \right\rvert \leqslant \overline{M}\, \overline{R}  \rho_h(t)  \lvert \mathcal{J}^0_h \rvert \max\limits_{i\in \mathcal{J}_h}{w_i(t)}, $$
    the boundedness of $\rho_h$ and inequality \eqref{w exponential conv},
    
    \item[\emph{iv)}] 
    \begin{align*}
        \sum\limits_{i,j\in \mathcal{J}^0_h}{ w_i(t){m(x_i(t), x_j(t), I_d(t, x_i(t)))\dot\alpha_j(t)}}+\sum\limits_{i,j\in \mathcal{J}^0_h}{\dot w_i(t){m(x_i(t), x_j(t), I_d(t, x_i(t)))\alpha_j(t)}}&\\
        +\sum\limits_{i, j\in \mathcal{J}^0_h}{w_i(t)\frac{d}{dt}m(x_i(t), x_j(t), I_d(t, x_i(t)))\alpha_j(t)}&=O(e^{-\delta t }),
    \end{align*} since for all $t\geq 0$,
    \begin{align*}
        &\left\lvert \sum\limits_{i,j\in \mathcal{J}^0_h}{ w_i(t){m(x_i(t), x_j(t), I_d(t, x_i(t)))\dot\alpha_j(t)}}\right\rvert\leqslant \overline{M}\lvert \mathcal{J}^0_h \rvert\|\dot{\alpha}\|_{\mathcal{l}^1} \max\limits_{i\in \mathcal{J}^0_h}w_i(t),
    \end{align*}
    \begin{align*}
        \left\lvert \sum\limits_{i,j\in \mathcal{J}_h}{\dot w_i(t){m(x_i(t), x_j(t), I_d(t, x_i(t)))\alpha_j(t)}} \right\rvert \leqslant \overline{M} \rho_h(t)  \lvert \mathcal{J}^0_h \rvert \lVert a \rVert_{W^{1,\infty}(\R^d)}\max\limits_{i\in \mathcal{J}^0_h}w_i(t), 
    \end{align*}
    and 
    \begin{align*}
         \left\lvert \sum\limits_{i, j\in \mathcal{J}_h}{w_i(t)\frac{d}{dt}m(x_i(t), x_j(t), I_d(t, x_i(t)))\alpha_j(t)} \right\rvert \leqslant   \overline{M} \left(\|a\|_{L^{\infty}(\R)}+\big\lvert\frac{d}{dt}I_g(t, x_i(t))\big\rvert\right) \rho_h(t)  \lvert \mathcal{J}^0_h \rvert \max\limits_{i\in \mathcal{J}^0_h}w_i(t), 
    \end{align*}
where $\frac{d}{dt}I_d(t, x_i(t))$ is bounded thanks to \eqref{d/dt I_g(t,xi(t))}. For these three inequalities, we conclude with \eqref{w exponential conv}. 
\end{itemize}

Hence, $\ddot \rho_h(t)\geqslant -p(t)\dot \rho_h(t)+O(e^{-\delta t})$. According to Lemma \ref{convergence negative part}, $\rho_h$ has a limit when $t$ goes to $\infty$, which we denote $\hat{\rho}_h$. 
Since $\dot \rho_h(t)=\sum\limits_{i\in \mathcal{J}^0_h}{R(x_i(t), I_g(t,x_i(t)))\alpha_i(t) }+\e(t)$, which converges to $R(\hat{x}, \psi_g(\hat{x}, \hat{x})\hat{\rho}_h)\hat{\rho}_h$, we deduce that 
$R(\hat{x}, \psi_g(\hat{x}, \hat{x})\hat{\rho_h})=0$, since $\hat{\rho}_h>0$.
\end{proof}
 
When $m\equiv0$ and under the same assumptions for the remaining coefficients of the problem as in Proposition \ref{conv vh unique equilibrium}, we are able to identify the limit of $v$ and prove that it coincides with the limit of $v^h$. According to Lemma \ref{lem condition asymptotic preserving}, this ensures that $v^h_\e$ is an asymptotic preserving approximation of $v$. 
\begin{theorem}
Let us assume that there exists $\hat{x}\in \R^d$ which is an asymptotically stable equilibrium for the ODE $\dot x =a(x)$ such that hypothesis \eqref{hypotheses exponential convergence} holds. We assume as well that $m\equiv0$ and that hypotheses \eqref{hypotheses R} hold. Then, $v$ converges to $\hat{\rho} \: \delta_{\hat{x}}$ in the weak sense in the space of Radon measures, where $\hat{\rho}$ is the unique solution of \[R(\hat{x}, \psi_g(\hat{x},\hat{x})\hat{\rho})=0.\] 
Consequently, $v^h_\e$ is an asymptotic preserving approximation of $v$.
\label{asymptotic preserving unique equilibrium}
\end{theorem}

\begin{proof}
Let us recall that
\[\rho(t)=\int_{\R^d}{v(t,x)dx}. \]
We recall as well that, the function $a$ being only dependent of $x$, the characteristic lines $X_t(x):=X(t,x)$ satisfy $X^{-1}_t(x):=X^{-1}(t,x)=X_{-t}(x)$ and $X_t(X_s(x))=X_{t+s}(x)$.
By using the fact that, for any $x\in \R^d$,  $$v(t,x)=v^0(X_{-t}(x))e^{\mathcal{G}(0,t,x)},$$
where
\[
\mathcal{G}(s,t,x):=\int_s^t{R(X_{\tau-t}(x), (I_gv)(\tau, X_{\tau-t}(x)))-\divv a(X_{\tau-t}(x)) d\tau},
\]
and that, by hypothesis, $K:=\text{supp}(v^0)$ is a compact set included in the basin of attraction of $\hat{x}$, one proves that $\text{supp}(v(t,\cdot ))$ is the image of $\text{supp}(v^0)$ by $X_t(\cdot)$. Since $X_t(y)$ converges to $\hat{x}$ for all $y\in \text{supp}(v^0)$, we prove that $K_t:=\text{supp}(v(t,\cdot ))=X_t(\text{supp}(v^0))=X_t(K)$ is a compact set included in the basin of attraction of $\hat{x}$, for all $t\geqslant 0$. 
By \eqref{hypotheses exponential convergence}, there exist $C>0$ and $\delta>0$ such that 
\begin{align*}
\left\lVert X_t( y)-\hat{x} \right\rVert \leqslant C e^{-\delta t},\quad \forall y \in K, \quad \forall t\geqslant 0. 
\end{align*}

Let $\phi \in \mathcal{C}_c(\R^d)$. By definition of $\rho$ and by using the change of variable `$x=X_t(y)$' we get
\begin{align*}
\bigg\lvert \int_{\R^d}{\phi(x)v(t,x)}-\rho(t)\phi(\hat{x})\bigg\rvert&\leqslant \int_{K_t}{\lvert \phi(x)-\phi(\hat{x}) \rvert v(t,x)dx}\\
&=\int_{K}{\lvert \phi(X_t(y))-\phi(\hat{x}) \rvert v(t,X_t(y))e^{\int_0^t\divv a(X_t(y))}   dy} \\
&\leqslant \max\limits_{y \in K} \lvert \phi(X_t(y))-\phi(\hat{x}) \rvert \int_{K}{v(t,X_t(y))e^{\int_0^t\divv a(X_t(y))}dy}\\
&=\max\limits_{y \in K} \lvert \phi(X_t(y))-\phi(\hat{x}) \rvert\rho(t).
\end{align*}
Since $\rho$ is bounded, this proves that
\begin{align}
    v(t, \cdot)-\rho(t)\delta_{\hat{x}}\underset{t\to +\infty}{\rightharpoonup} 0.
\label{n-rho delta}
\end{align}
Hence, it only remains to prove that $\rho$ converges to the expected limit. We see that
\begin{align}
\dot \rho(t)=\int_{\R^d}{R(x, (I_gv)(t,x)) v(t,x)}dx.
\label{derivative rho}
\end{align}
First, let us note that $\rho$ has a positive lower bound. Indeed, according to the hypothesis on $\psi_g$, for all $x\in \R^d$,  $\underline{\psi_g}\,\rho(t)\leqslant (I_gv)(t, x)\leqslant \lVert \psi_g\rVert_{L^{\infty}}  \, \rho(t) $. 
Thus,
\[\dot \rho(t)\geqslant \left(\underset{x\in K_t}{\min}\left(R(x, (I_gv)(t, x)\right)\right)\, \rho(t).\]
Hence, as soon as $\rho(t)$ becomes small, and so does $I_g(t,x)$, $\rho$ becomes increasing, which proves that $\rho$ is lower bounded by a positive constant. We get an upper bound for $|\dot{\rho}(t)|$ from the relation
\[
|\dot{\rho}(t)|\leqslant\left(\underset{x\in K_t}{\max}\left(R(x, (I_gv)(t, x)\right) \right)\rho(t),
\]
and the boundedness of $\rho(t)$.\\
We introduce now the function 
\[
\Tilde{v}(t,y)=v(t,X_t(y))e^{\int_0^t\divv a(X_s(y))ds},
\]
which satisfies
\[
\int_{K}\Tilde{v}(t,y)dy=\int_{K_t}v(t,x)dx=\rho(t).
\]
Moreover,
\begin{align}
    \partial_t \Tilde{v}(t,y)=&R(X_t(y),(I_gv)(t,X_t(y)))\Tilde{v}(t,y).\label{d/dt tilde(v)(t,y)}
\end{align}
Before proving that $\rho$ satisfies the equality of Lemma \ref{convergence negative part}, we observe that the following relation holds for all $y\in K$: 
\begin{align}
\left\lvert \frac{d}{dt}(I_{g}v)(t, X_t( y))-\psi_g(\hat{x}, \hat{x})\dot \rho(t)\right\rvert\leqslant& \left\lvert\int_{K_t}{a(X_t(y))\cdot \nabla_x \psi_g(X_t(y), z)v(t,z)dz}\right\rvert\nonumber\\
&+\left\rvert\int_{K_t}{\big(\psi_g(X_t(y), z)-\psi_g(\hat{x}, \hat{x}) \big)\partial_t v(t,z) dz}\right\rvert.\label{dtIgamma}
\end{align}
From the equation satisfied by $v$, we see that
\begin{align*}
    &\int_{K_t}{\big(\psi_g(X_t(y), z)-\psi_g(\hat{x}, \hat{x}) \big)\partial_t v(t,z) dz}\\
    =&\int_{K_t}{\nabla_z\psi_g(X_t(y), z)a(z)v(t,z)dz}\\
    &+\int_{K_t}{\big(\psi_g(X_t(y), z)-\psi_g(\hat{x}, \hat{x}) \big)R(z,(I_gv)(t,z))v(t,z) dz}\\
    =&\int_{K}{\nabla_z\psi_g(X_t(y), X_t(\bar z))a(X_t(\bar z))\Tilde{v}(t,\bar z)d\bar z}\\
    &+\int_{K}{\big(\psi_g(X_t(y), X_t(\bar z))-\psi_g(\hat{x}, \hat{x}) \big)R(X_t(\bar z),(I_gv)(t,X_t(\bar z)))\Tilde{v}(t,\bar z) d\bar z},
\end{align*}
which allows us to conclude that
\begin{align*}
    \left\lvert\int_{K_t}{\big(\psi_g(X_t(y), z)-\psi_g(\hat{x}, \hat{x}) \big)\partial_t v(t,z) dz}\right\rvert\leqslant& \max\limits_{\bar z\in K}\|a(X_t(\bar z))\|\|\psi_g\|_{W^{1,\infty}(\R^d)}\rho(t)\\
    &+\max\limits_{y,\bar z\in K}\left\lvert\psi_g(X_t(y), X_t(\bar z))-\psi_g(\hat{x}, \hat{x})\right\rvert\overline{R}\rho(t).
\end{align*}
Using this relation in \eqref{dtIgamma} gives
\begin{align*}
    \left\lvert \frac{d}{dt}(I_{\gamma}v)(t, X_t( y))-\psi_\gamma(\hat{x}, \hat{x})\dot \rho(t)\right\rvert\leqslant& 2\max\limits_{\bar z\in K}\|a(X_t(\bar z))\|\|\psi_\gamma\|_{W^{1,\infty}(\R^d)}\rho(t)\\
    &+\max\limits_{y,\bar z\in K}\left\lvert\psi_\gamma(X_t(y), X_t(\bar z))-\psi_\gamma(\hat{x}, \hat{x})\right\rvert\overline{R}\rho(t),
\end{align*}
which proves, according to hypothesis \eqref{hypotheses exponential convergence}, that for all $y\in K$
\begin{align}
\frac{d}{dt}I_g(t, X(t, y))=\psi_g(\hat{x}, \hat{x})\dot \rho(t)+O(e^{-\delta t})), 
\label{d/dt I_g(t, X(t, y))}
\end{align}
for a certain $\delta>0$.\\ 
Differentiating \eqref{derivative rho} we get 
\begin{align*}
\ddot \rho(t)=&\frac{d}{dt}\left(\int_{K_t}{R(x, (I_gv)(t,x)) v(t,x)}dx\right)\\
=&\frac{d}{dt}\left(\int_{K}{R(X_t(y), (I_gv)(t,X_t(y))) \tilde{v}(t,y)}dy\right)\\
=&\int_{K}{a(X_t(y))\cdot\nabla_x R(X_t(y), (I_gv)(t,X_t(y)))\tilde{v}(t,y)}dy\\
&+\int_{K}{\frac{d}{dt}(I_{g}v)(t, X_t( y))\partial_IR(X_t(y), (I_gv)(t,X_t(y))) \tilde{v}(t,y)}dy\\
&+\int_{K}{R(X_t(y), (I_gv)(t,X_t(y)))\partial_t\Tilde{v}(t,y)dy}.
\end{align*}
Let us note that
\begin{itemize}
    \item[\emph{i)}] Since $a(X_t(y))$ converges uniformly to zero with an exponential speed, then
    $$\int_{\R^d}{a(X_t(y))\cdot\nabla_x R\left(X_t(y), (I_gv)(t, X_t(y))\right)\tilde{v}(t, y))dy}=O(e^{-\delta t}), $$
    thanks to the boundedness of $\nabla_x R$ and $\rho(t)$.
    \item[\emph{ii)}] According to \eqref{d/dt I_g(t, X(t, y))},
    $$\int_{K}{\frac{d}{dt}(I_{g}v)(t, X_t( y))\partial_IR(X_t(y), (I_gv)(t,X_t(y))) \tilde{v}(t,y)}dy=-p(t)\dot \rho(t) + O(e^{-\delta t}),$$ 
    with 
    \begin{align*}
        p(t)&:=-\psi_g(\hat{x}, \hat{x}) \int_{K}{\partial_IR(X_t(y), (I_gv)(t,X_t(y))) \tilde{v}(t,y)}dy\geqslant \psi_g(\hat{x}, \hat{x}) D\min\limits_{t\geqslant 0}\rho(t)>0.
    \end{align*}
    \item[\emph{iii)}] Directly from \eqref{d/dt tilde(v)(t,y)}, 
    \begin{align*}
        \int_{K}{R(X_t(y), (I_gv)(t,X_t(y)))\partial_t\Tilde{v}(t,y)dy}
        &=\int_{K}{\Big(R(X_t(y), (I_gv)(t,X_t(y)))\Big)^2  \tilde{v}(t, y)dy}\\
        &=:P(t)\geqslant 0.
        \end{align*}
\end{itemize}
Thus, $\ddot \rho(t)\geqslant -p(t)\dot \rho(t)+O(e^{-\delta t})$, hence, $\rho$ converges, thanks to Lemma \ref{convergence negative part}. 

Recalling \eqref{n-rho delta}, $v(t, \cdot)$ thus converges to $\hat{\rho} \delta_{\hat{x}}$, where $\hat{ \rho}$ is the limit of $\rho$. We conclude, according to Proposition \ref{Conditions convergence dirac}, that $\hat{ \rho}$ satisfies the expected equality. \\
Having proved that $v$ and $v^h$ share the same limit is enough then to conclude, thanks to Lemma \ref{lem condition asymptotic preserving}, that $v_\e^h$ is an asymptotic preserving approximation.
\end{proof}

If $m$ is not $0$, under very specific hypotheses over its support, we can extend the result of Theorem \ref{asymptotic preserving unique equilibrium}. The explanation behind this is simple: as long as the population is composed of traits that are not prone to mutatations, it will evolve as in the case where mutations are not possible at all.
\begin{theorem}
Let us assume that there exists $\hat{x} \in \R^d$ which is an asymptotically stable equilibrium for the ODE $\dot x =a(x)$ such that \eqref{hypotheses exponential convergence} holds. Moreover, let us assume that $\textrm{supp}\left(v^0\right)\cup K_{x}$ is a compact set such that there exist $C', \delta'>0$ such that
\begin{equation*}
\forall y \in  \textrm{supp}\left(v^0\right)\cup K_{x}, \quad t\geq 0, \quad \lVert X(t,y) -\hat{x}\rVert \leq C'e^{-\delta' t}, 
\end{equation*}
that $\bigcup\limits_{s\geqslant 0}\left(X_s(\text{supp}(v^0)\cup K_x)\right)\cap K_y=\emptyset$ and that hypothesis \eqref{hypotheses R} holds. Then, $v$ converges to $\hat{\rho} \: \delta_{\hat{x}}$ in the weak sense in the space of Radon measures, where $\hat{\rho}$ is the unique solution of \[R(\hat{x}, \psi_g(\hat{x},\hat{x})\hat{\rho})=0.\] 
Consequently, $v^h_\e$ is an asymptotic preserving approximation of $v$.
\label{asymptotic preserving unique equilibrium1}
\end{theorem}
\begin{proof}
By using the fact that, for any $x\in \R^d$,  $$v(t,x)=v^0(X_{-t}(x))e^{\mathcal{G}(0,t,x)}+\int_0^t\int_{\R^d}m(X_{s-t}(x), z, (I_dv)(s, X_{s-t}(x)))v(s, z)dze^{\mathcal{G}(s,t,x)}ds,$$
where
\[
\mathcal{G}(s,t,x):=\int_s^t{R(X_{\tau-t}(x), (I_gv)(\tau, X_{\tau-t}(x)))-\divv a(X_{\tau-t}(x)) d\tau},
\]
we observe that $\text{supp}(v(t,x))\subset X_t(\text{supp}(v^0))\cup\bigcup\limits_{0\leqslant s\leqslant t}X_s(K_x)\subset\bigcup\limits_{s\geqslant 0}\left(X_s(\text{supp}(v^0)\cup K_x)\right)$, therefore
\begin{align*}
    v(t,x)=&v^0(X_{-t}(x))e^{\mathcal{G}(0,t,x)}+\int_0^t\int_{\R^d}m(X_{s-t}(x), z, (I_dv)(s, X_{s-t}(x)))v(s, z)dze^{\mathcal{G}(s,t,x)}ds\\
    =&v^0(X_{-t}(x))e^{\mathcal{G}(0,t,x)}+\int_0^t\int_{\text{supp}(v(t,x))\cap K_y}m(X_{s-t}(x), z, (I_dv)(s, X_{s-t}(x)))v(s, z)dze^{\mathcal{G}(s,t,x)}ds\\
    =&v^0(X_{-t}(x))e^{\mathcal{G}(0,t,x)}.
\end{align*}
Therefore, we can replicate the proof for the case $m\equiv 0$.
\end{proof}

The result of Theorem \ref{asymptotic preserving unique equilibrium} does not generalize when $\text{supp} \left(v^0\right)$ is not strictly contained in the basin of attraction of $x_s$, as shown in the following result:
\begin{prop}
Let us consider the one-dimensional PDE
\begin{align}
    \begin{cases}
    \partial_t v(t,x)+\nabla_x \cdot \left(a(x)v(t,x)\right)=(r(x)-\rho(t))v(t,x),\\
    \rho(t)=\int_{\R}{v(t,x)dx},\\
    n(0,\cdot)=n^0(\cdot),
    \end{cases}
    \label{PDE adv-growth}
\end{align}
which is a particular case of \eqref{Pv}, and let us assume that there exist $x_u<x_s$ such that $a(x_u)=a(x_s)=0$, $a'(x_u)>0$, $a'(x_s)<0$, $\mathrm{supp}(n^0)\subset [x_u, x_s]$, $n^0(x_u)=0$, and that there exists $\alpha>0$ such that ${n^0}'(x)=\underset{x\to x_u^+}{O}((x-x_u)^\alpha)$ and that $r(x_u)-(1+\alpha)f'(x_u)>r(x_s)$. Then, $v_\e^h$ is not an asymptotic preserving approximation of $v$. \label{non asymptotic preserving}
\end{prop}

\begin{proof}
The long-time behaviour of the solution of \eqref{PDE adv-growth} has been studied in detail in \cite{guilberteau2023long}, and it has been proved, under the hypotheses of Proposition \ref{non asymptotic preserving}, that $v$ converges to a function in $L^1$. 
Let us now compute the limit of $v^h$. Since $n^0(x_u)=0$, we can assume, without loss of generality, that for all $i\in \mathcal{J}_h$, $x_i^0\in (0,1]$. Thus, since $a>0$ on $(x_u, x_s)$, for all $t\geqslant 0$, $x_i(t)$ converges to $x_s$. As seen in the proof of Proposition \ref{conv vh unique equilibrium}, $v^h$ therefore converges to $r(x_s)\delta_{x_s}$, and $v_\e^h$ is therefore not a asymptotic preserving approximation of $v$. 
\end{proof}

\section{Simulations}
In this section, we present some simulations obtained with the particle method developed throughout the article. In Figure \ref{fig: Friedman}, we deal with the non-local advection equation presented in \cite{friedman2009asymptotic}, which writes
\begin{align}
\partial_t v(t,x)+\nabla_x \left( a(t, I_1v(t,x), I_2 v(t,x) \right)=0,\quad x\in \R^d, t\geq 0,
\label{equation Friedman}
\end{align}
with $I_jv(t,x)=\int_{\R^2}{x_j v(t,x)dx}$, for $j\in \{1,2\}$. Note that this equation is not exactly a particular case of \eqref{eq intro}, since there are two non-local terms involved for advection, but the particle method can straightforwardly be adapted to this case. As in this paper, we show that, depending on the parameters, the solution of this PDE can converge to a single Dirac mass, to a sum of two Dirac masses, or to the sum of four Dirac masses. The parameters used for the simulations are the same as the one detailed in Figures $9$, $10$ and $11$ of \cite{friedman2009asymptotic}. 

Figure \ref{fig: advection-growth} illustrates different scenarios for the equation
\begin{align}
\begin{cases}
\partial_t v(t,x) +\nabla_x \left(a(x)v(t,x)\right)=\left(r(x)-\rho(t)\right)n(t,x),\\
\rho(t)=\int_{\R}{v(t,x)dx},\\
n(0,x)=n^0(x),
\end{cases}
\label{equation advection-selection}
\end{align}
with $a(x)=x(1-x)$, and an initial solution supported in $[0,1]$. This equation has been studied in \cite{guilberteau2023long} where it has been proved that its solution can either converge to a function in $L^1$, (which depends on the initial condition), or to a Dirac mass on $1$, depending on the functions $r$ and $n^0$.

\begin{figure}[H]
\begin{subfigure}{\textwidth}
\includegraphics[width=0.33\textwidth]{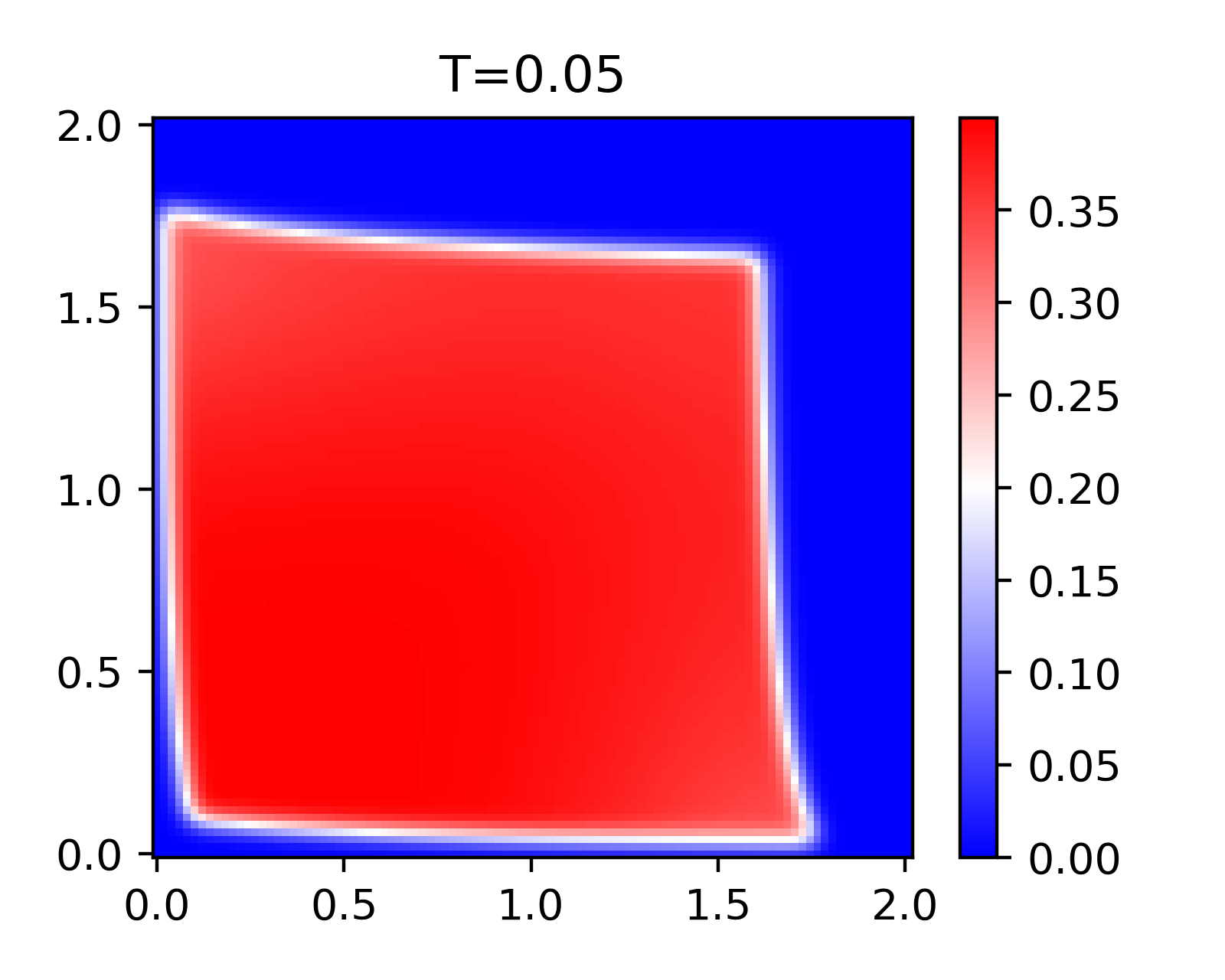}
\includegraphics[width=0.33\textwidth]{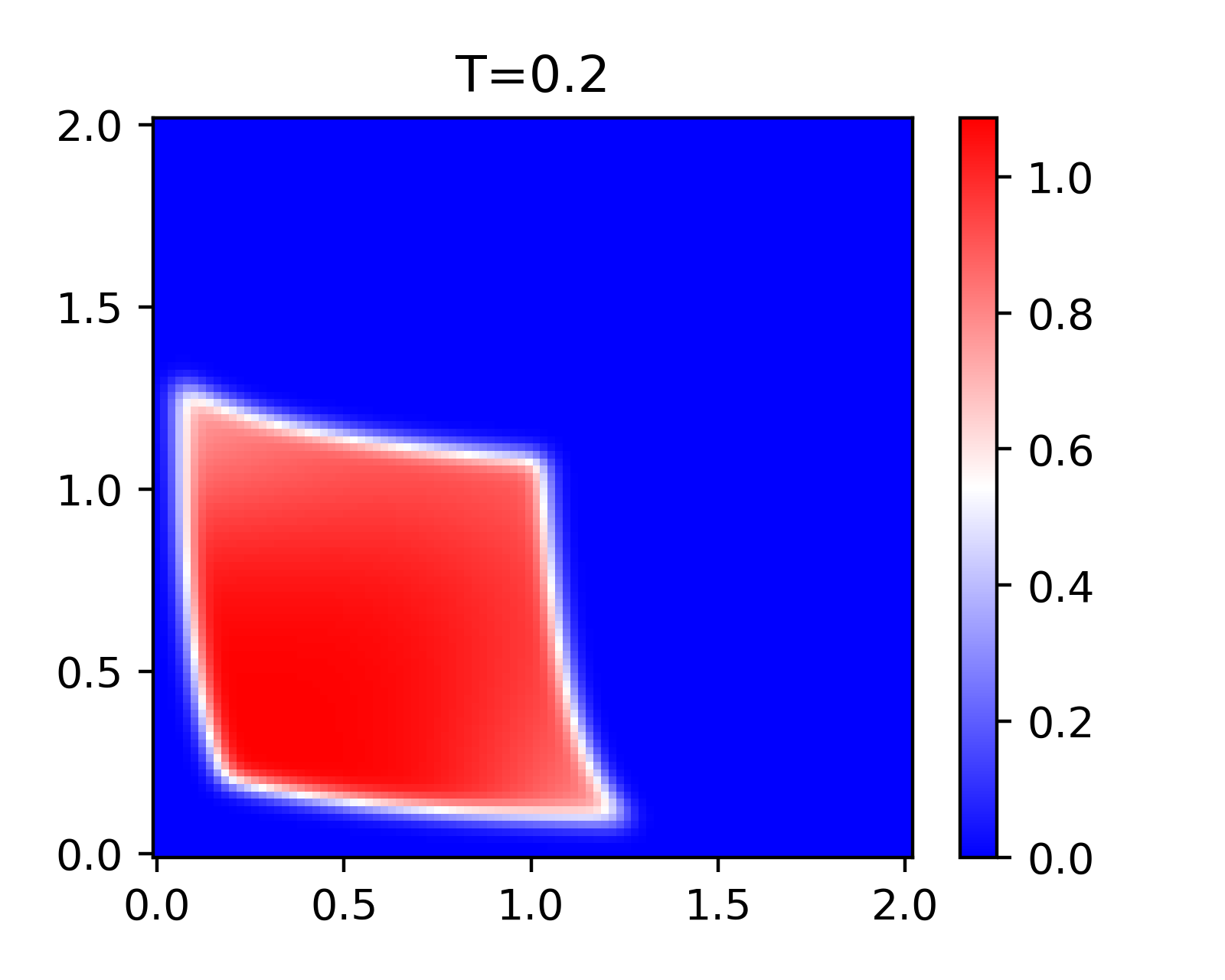}
\includegraphics[width=0.33\textwidth]{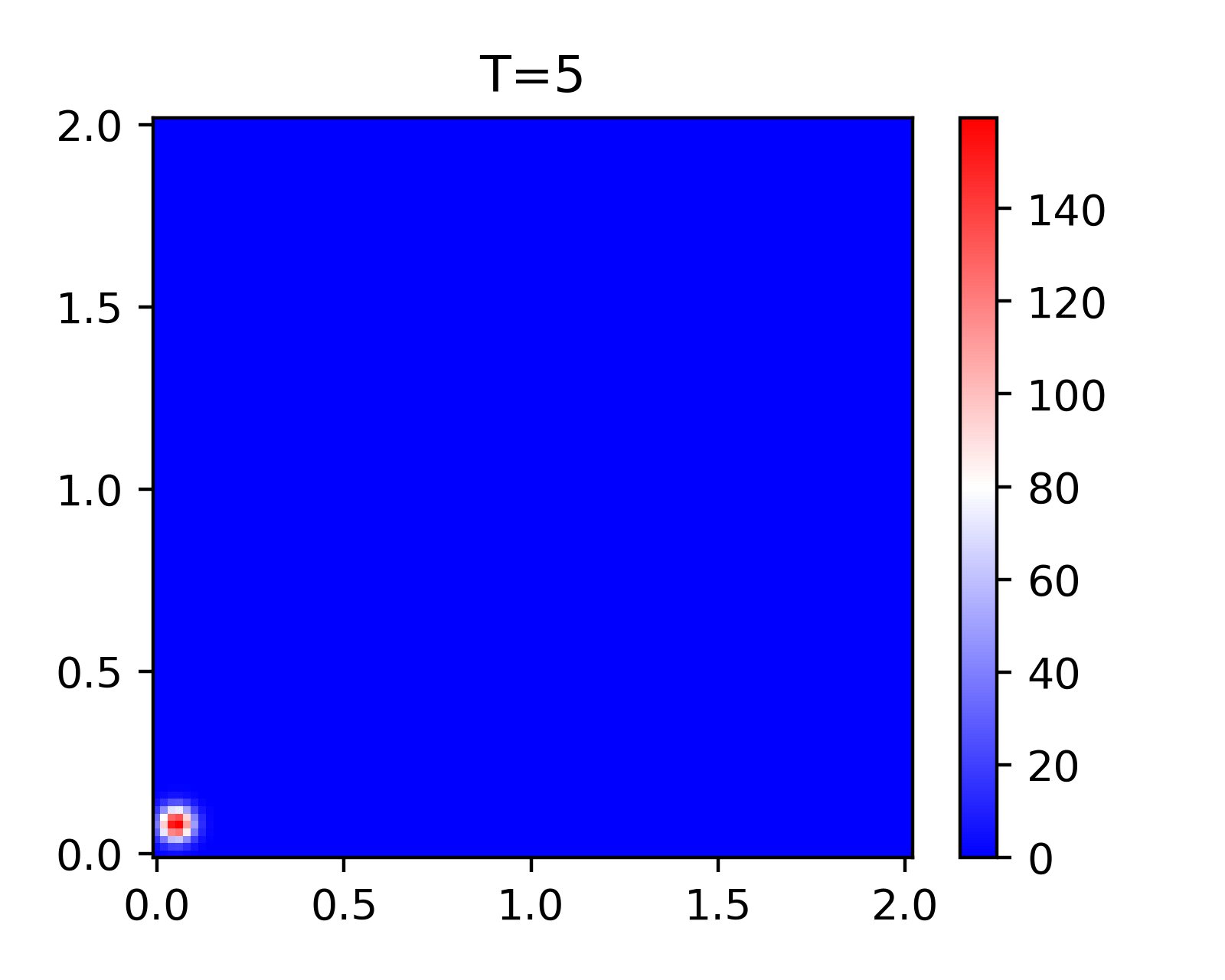}
\label{fig: monostable}
\caption{Monostability}
\end{subfigure}

\begin{subfigure}{\textwidth}
\includegraphics[width=0.33\textwidth]{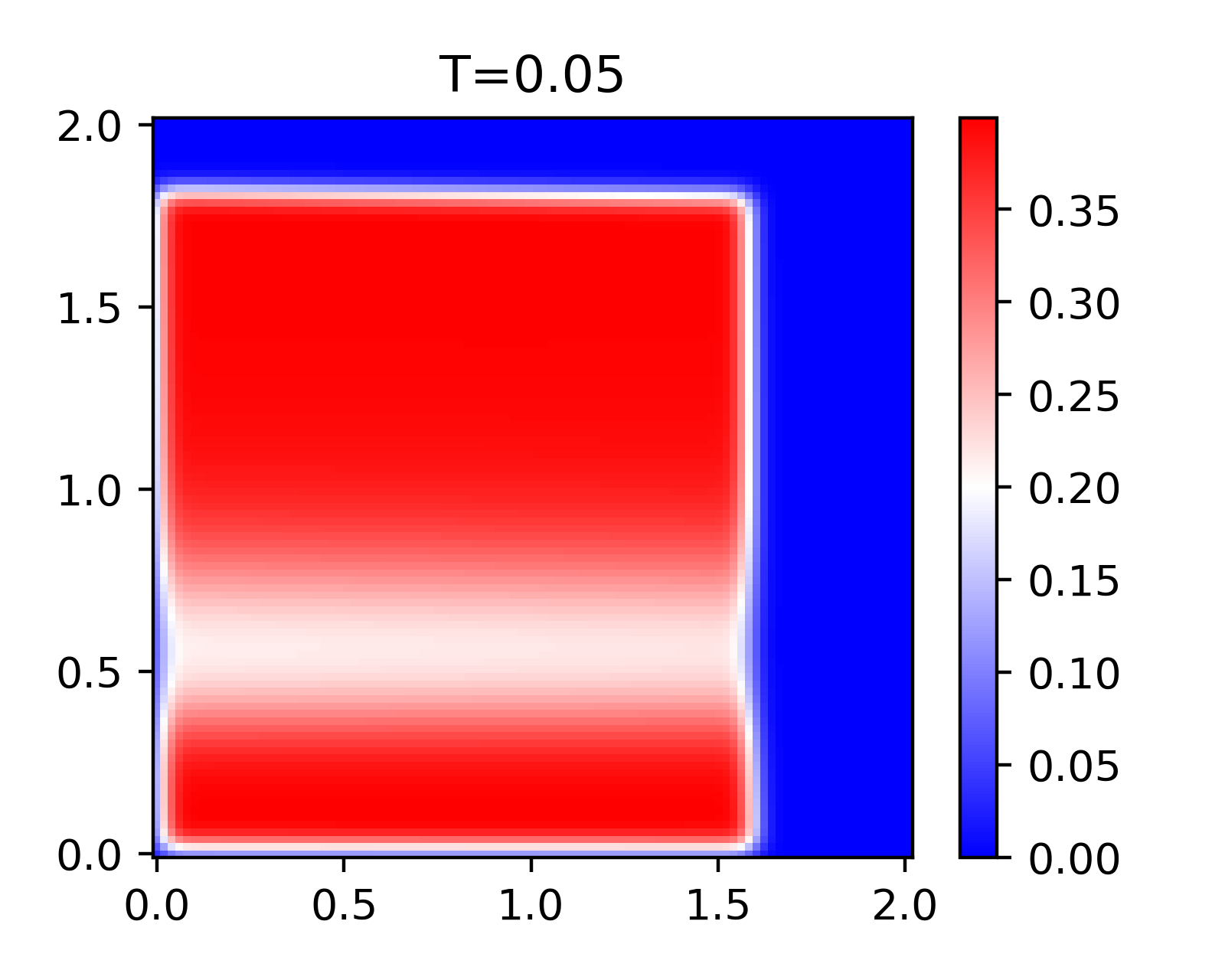}
\includegraphics[width=0.33\textwidth]{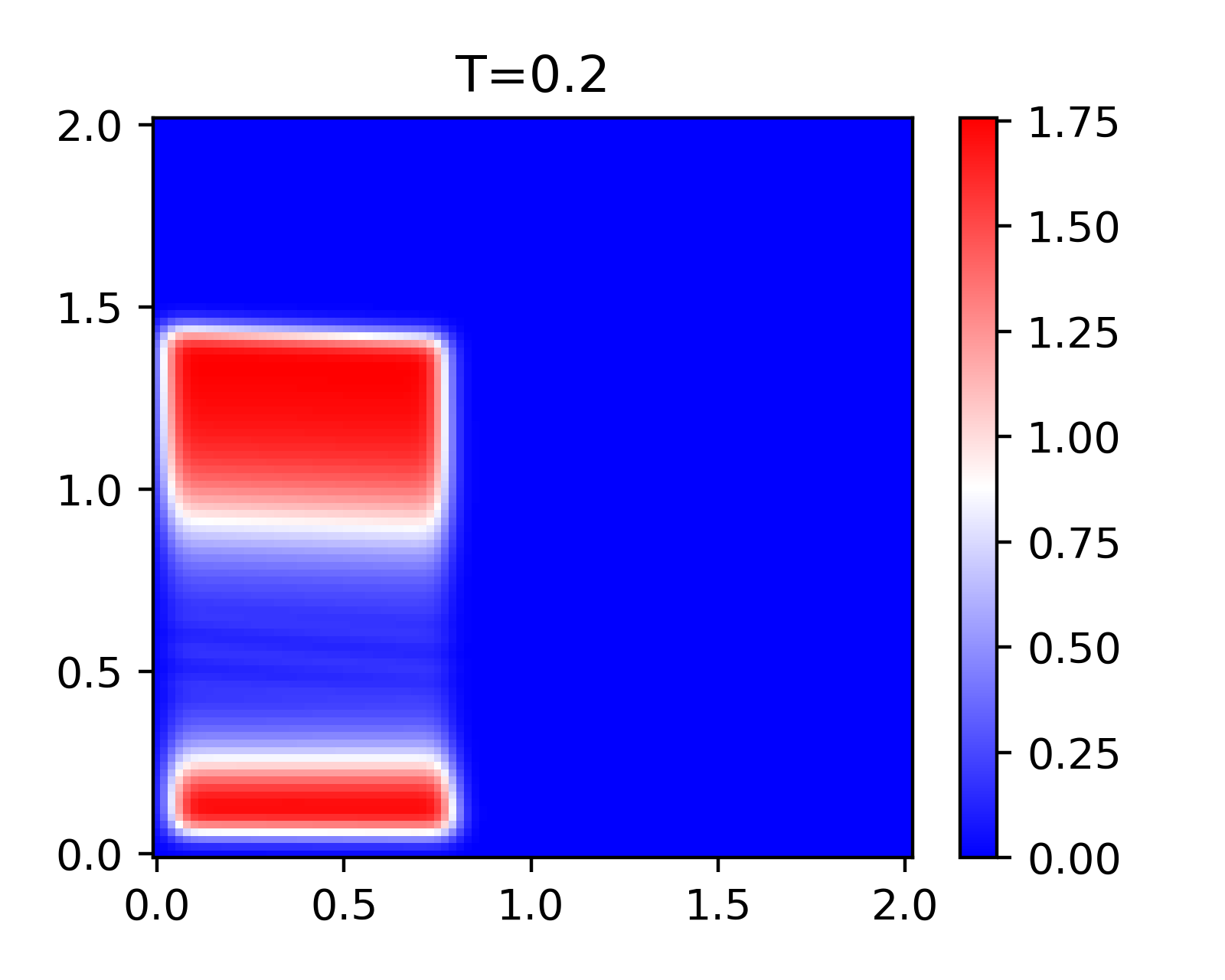}
\includegraphics[width=0.33\textwidth]{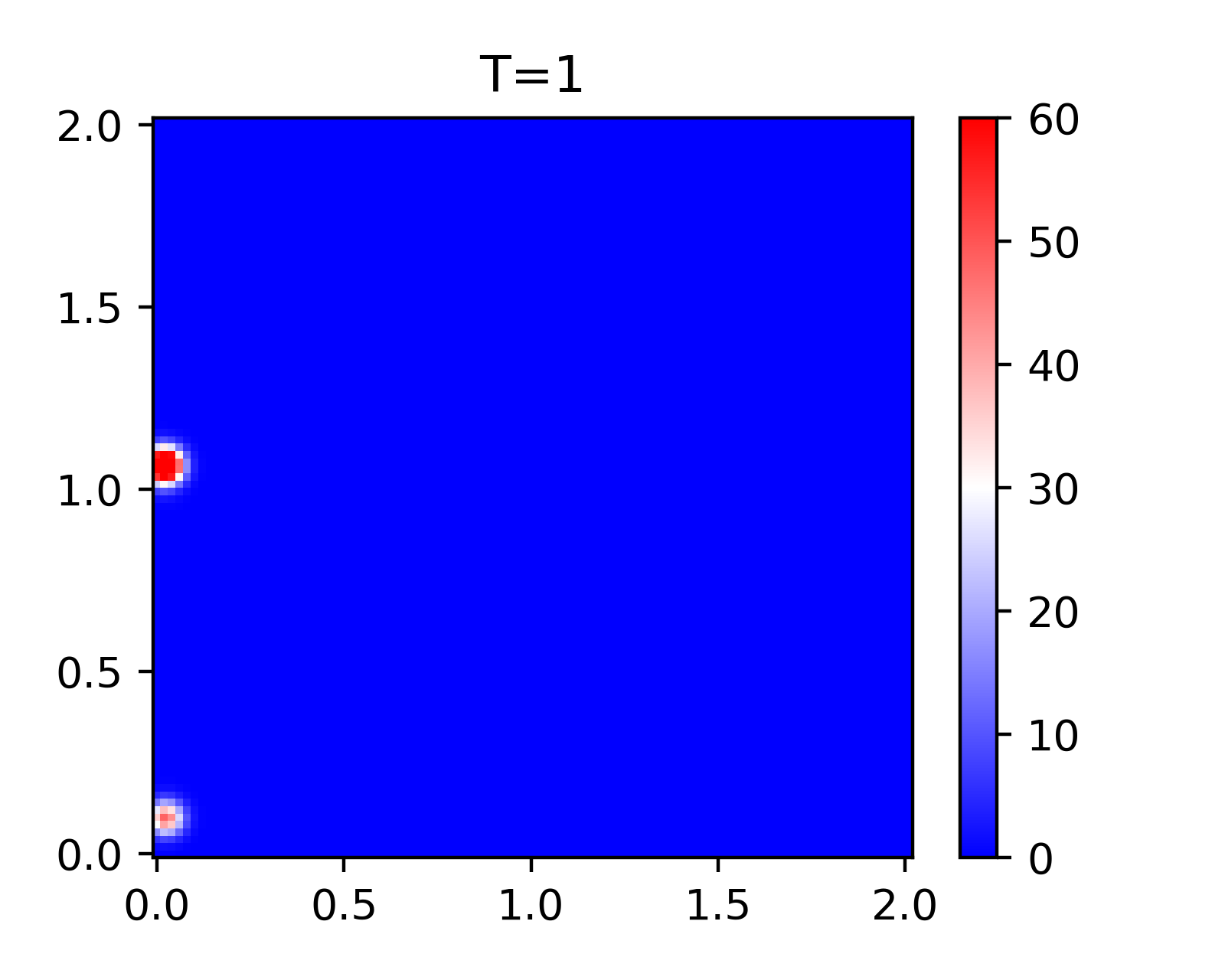}
\label{fig: bistable}
\caption{Bistability}
\end{subfigure}

\begin{subfigure}{\textwidth}
\includegraphics[width=0.33\textwidth]{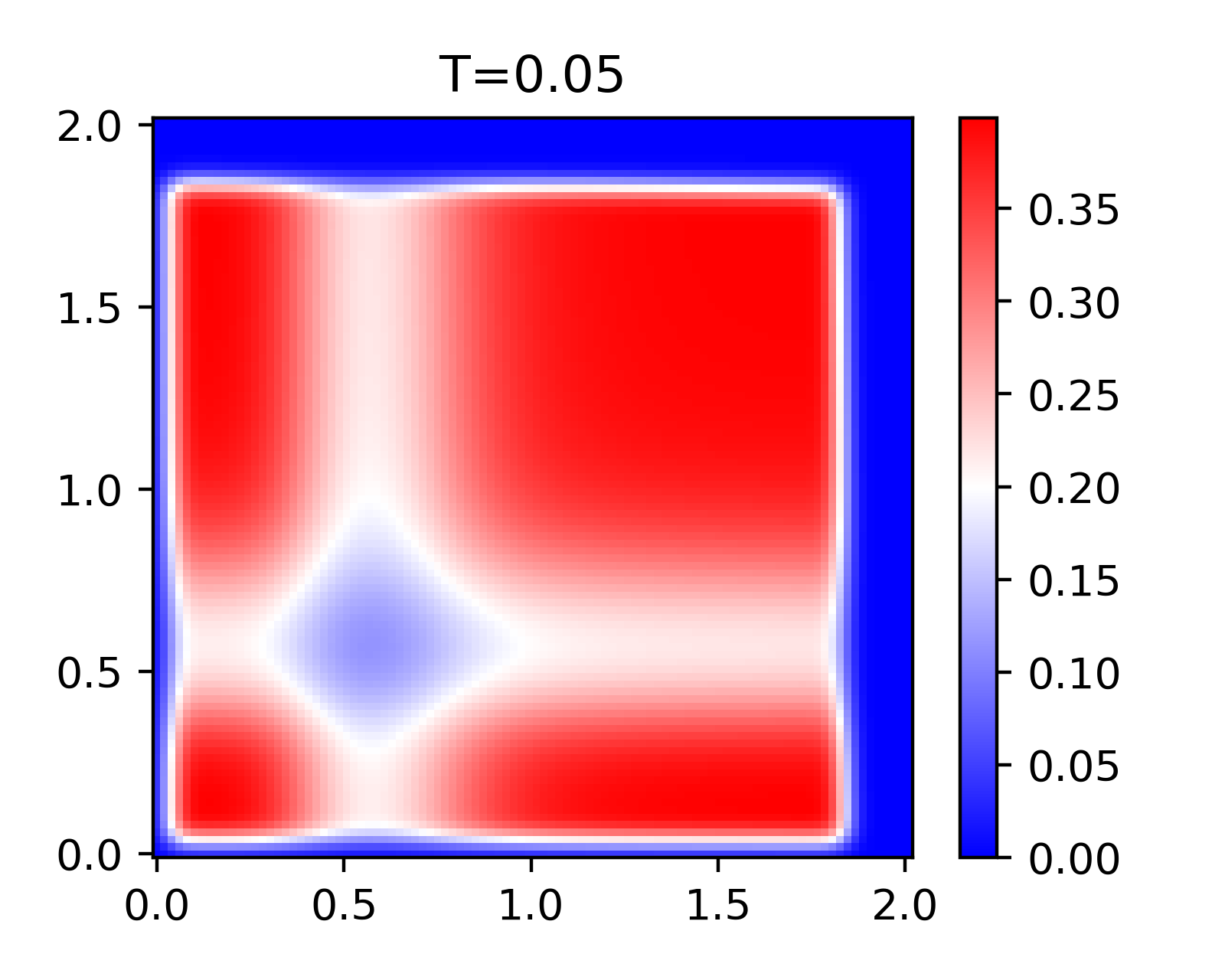}
\includegraphics[width=0.33\textwidth]{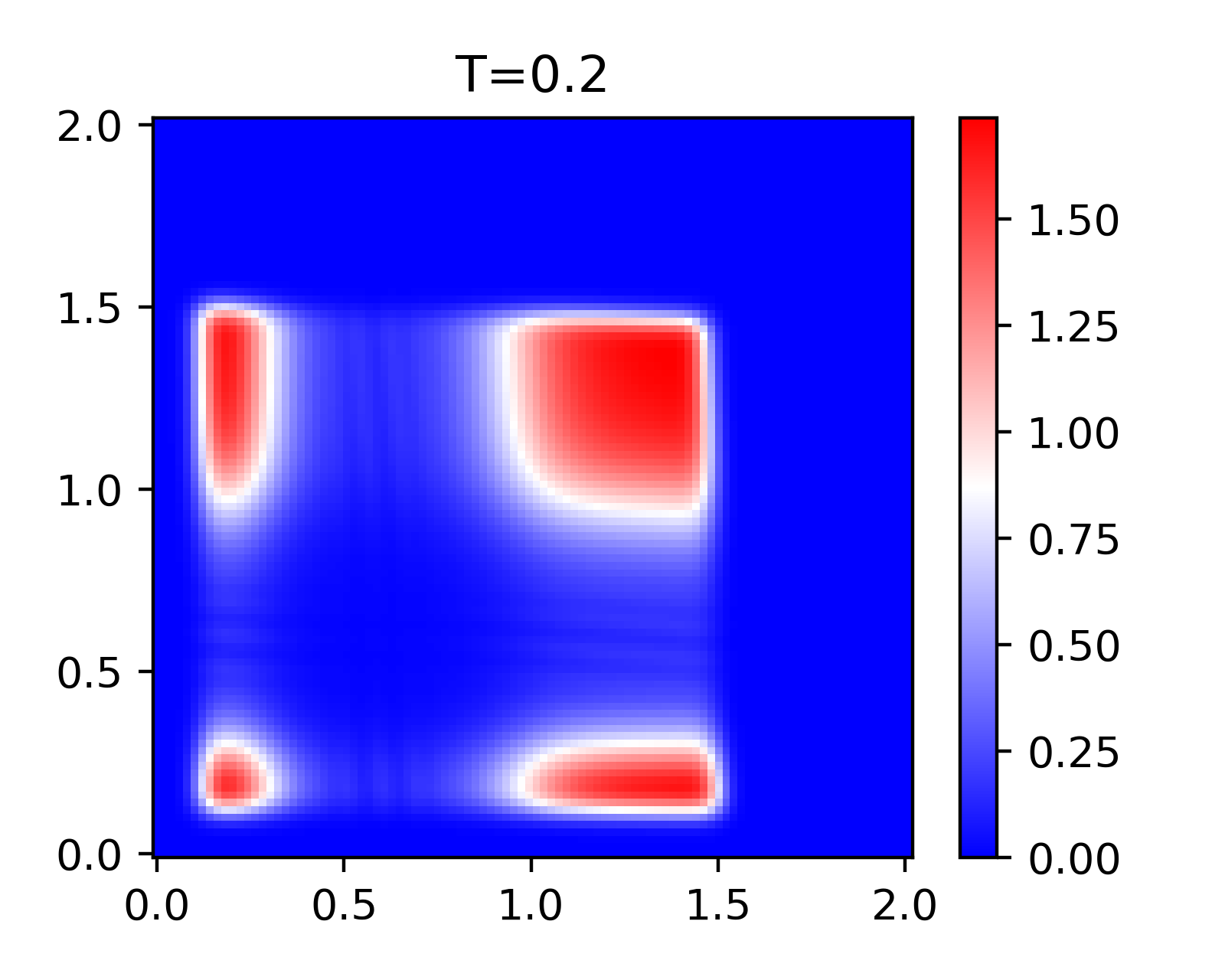}
\includegraphics[width=0.33\textwidth]{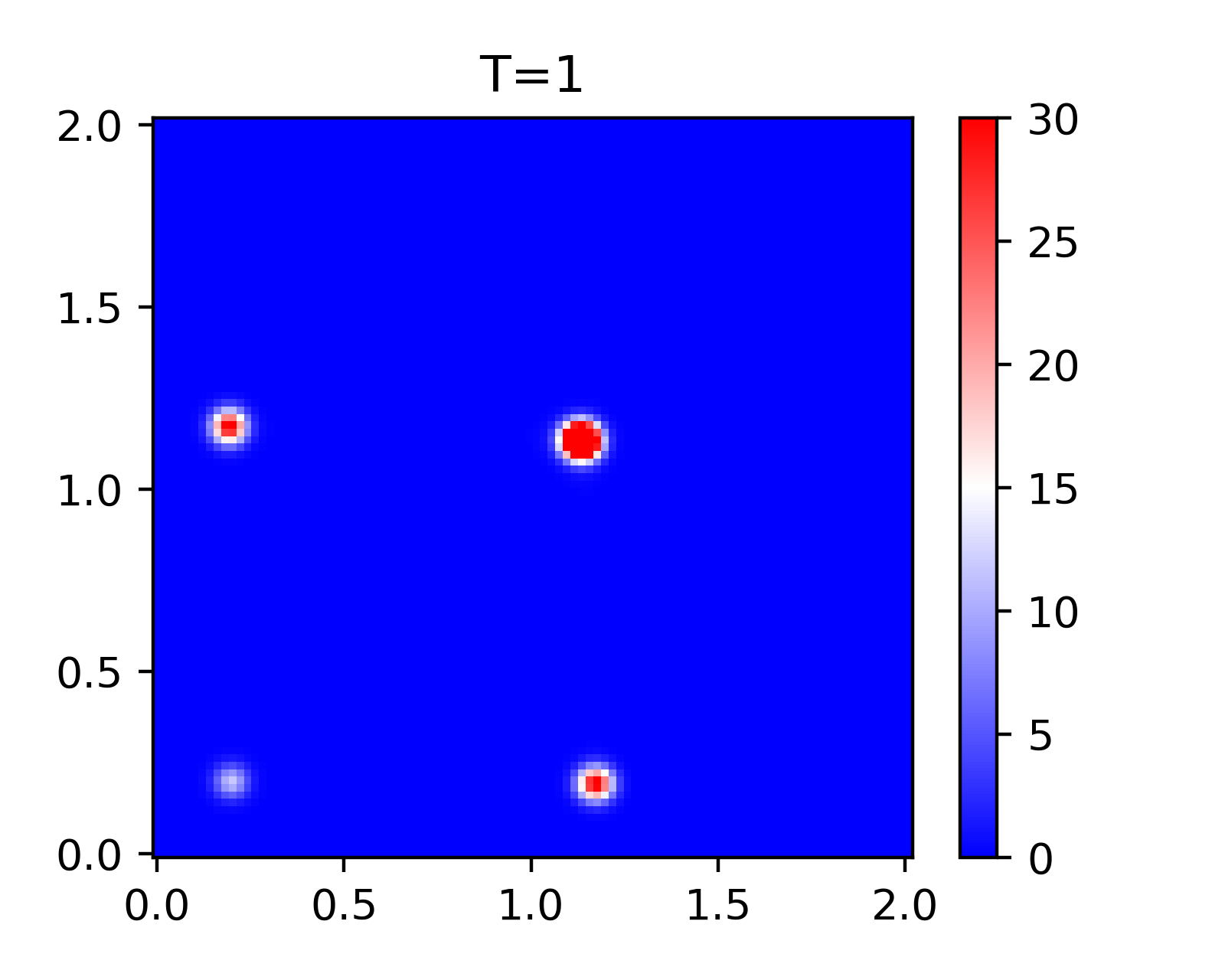}
\label{fig: quadristable}
\caption{Quadstability}
\end{subfigure}
\caption{The three possible regimes of convergence for equation \eqref{equation Friedman}, obtained with the particle method. The lines $(a)$, $(b)$ and $(c) $ respectively show the convergence to a single Dirac mass, two Dirac masses and four Dirac masses, and have been obtained by choosing the parameters of Figures 9, 10 and 11 of \cite{friedman2009asymptotic}. In the three cases, we have chosen $N=100$, $h=2./100$, $\e=h^{0.8}$, and the cut off function $\varphi$ is a Gaussian.}
\label{fig: Friedman}
\end{figure}

\begin{figure}[H]
\centering
\begin{subfigure}{\textwidth}
\includegraphics[width=0.33\textwidth]{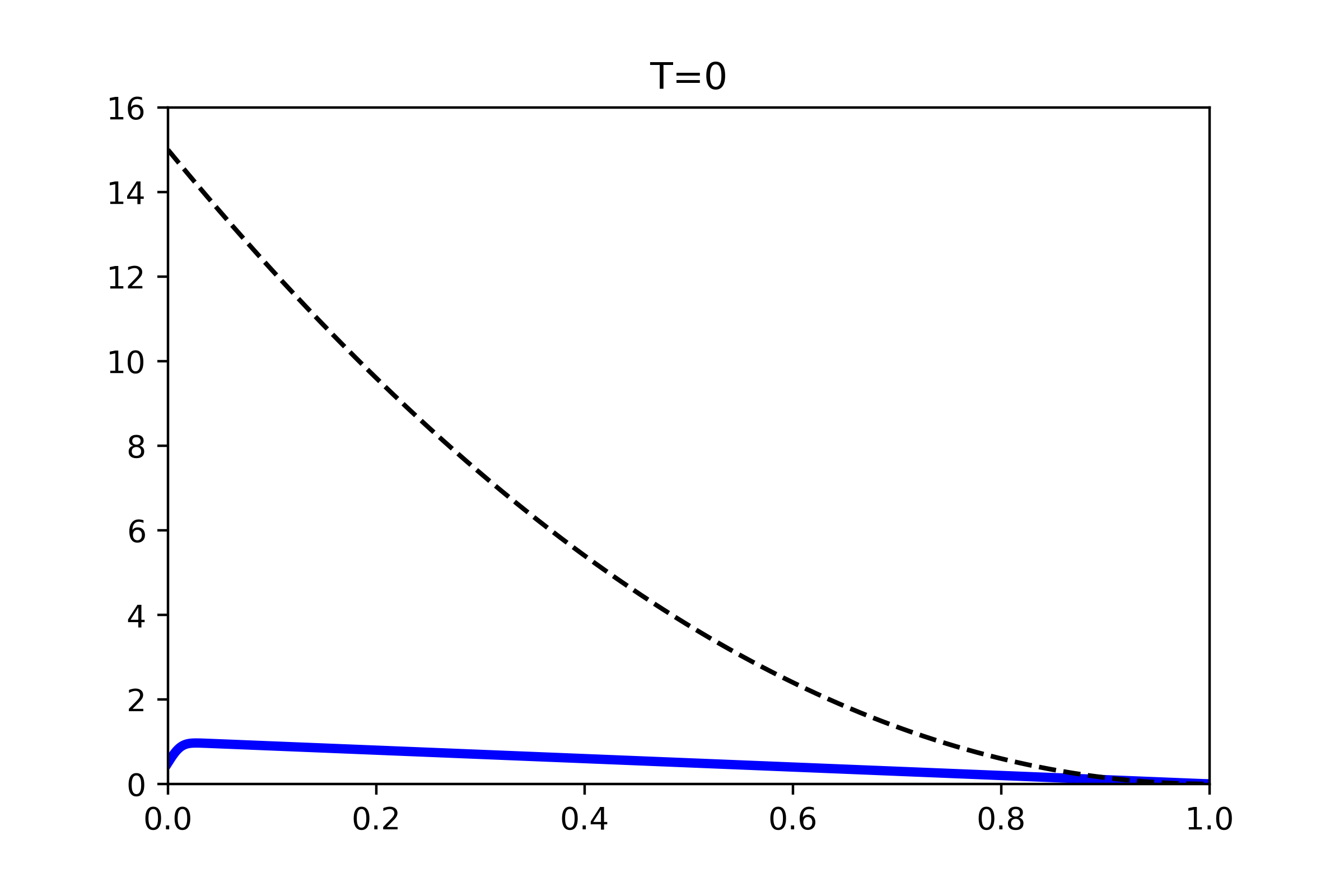}
\includegraphics[width=0.33\textwidth]{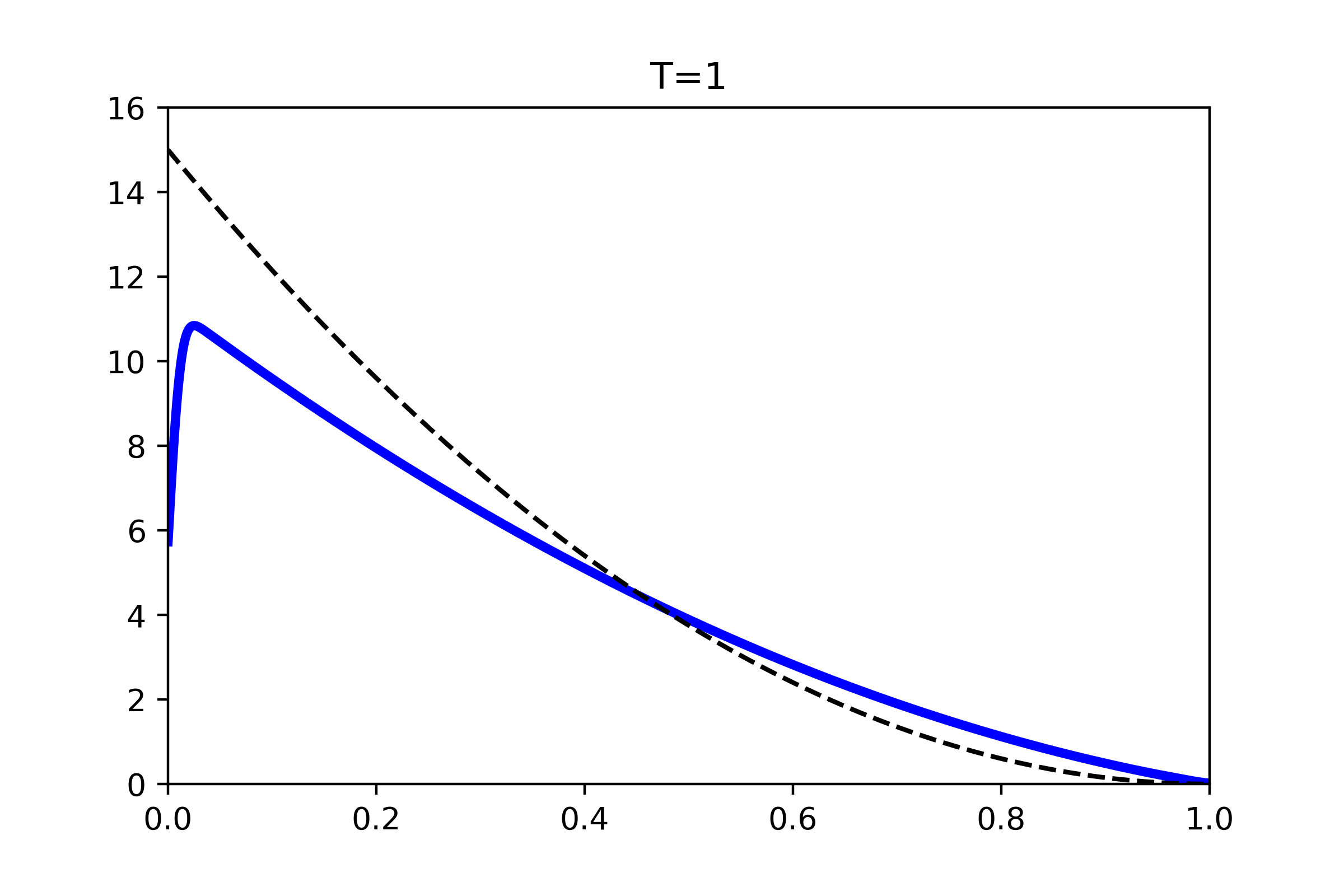}
\includegraphics[width=0.33\textwidth]{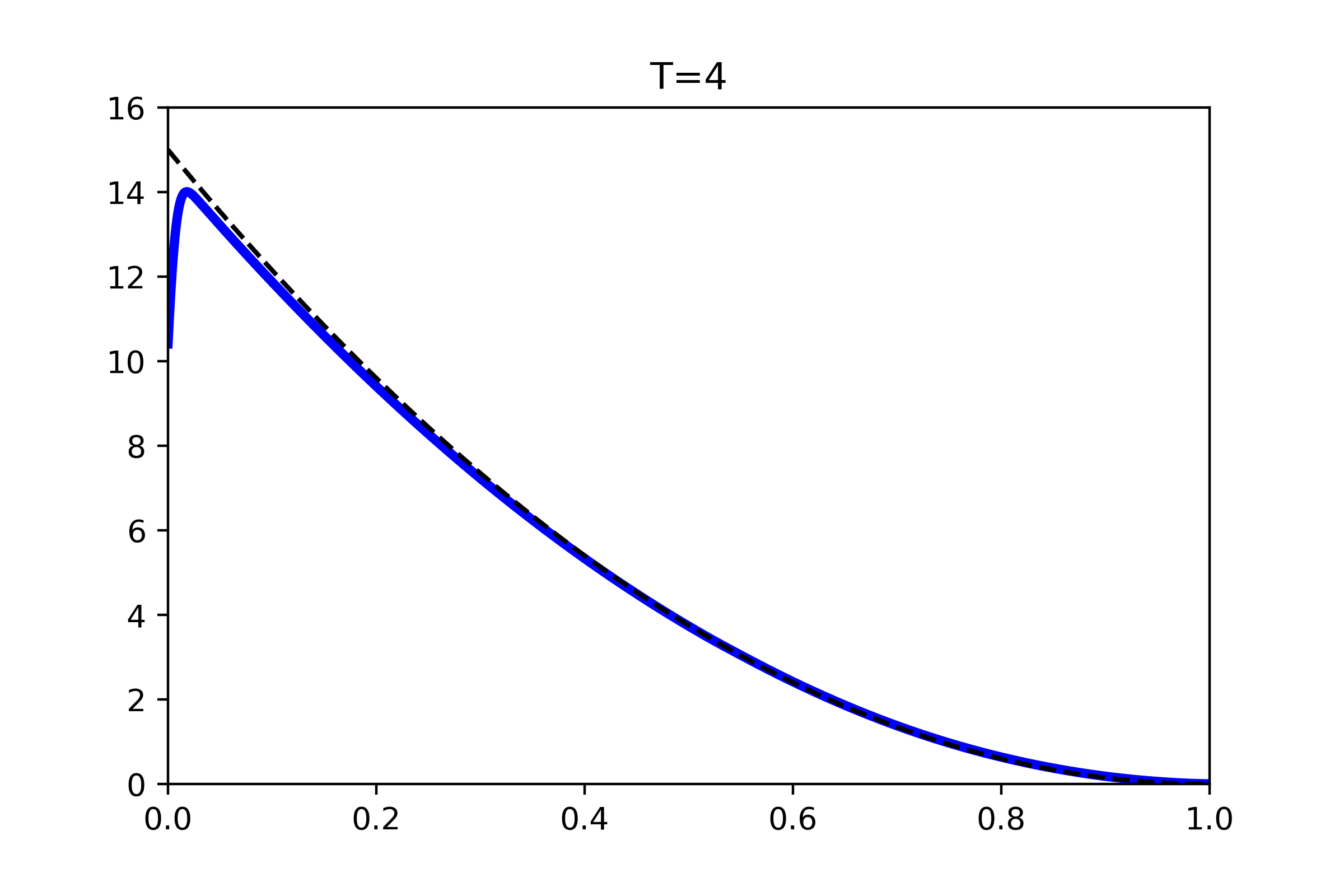}
\caption{Solution of \eqref{equation advection-selection} at different time steps, with $n^0(x)=1-x, r(x)=6-4x$.}
\label{alpha 0}
\end{subfigure}
\begin{subfigure}{\textwidth}
\includegraphics[width=0.33\textwidth]{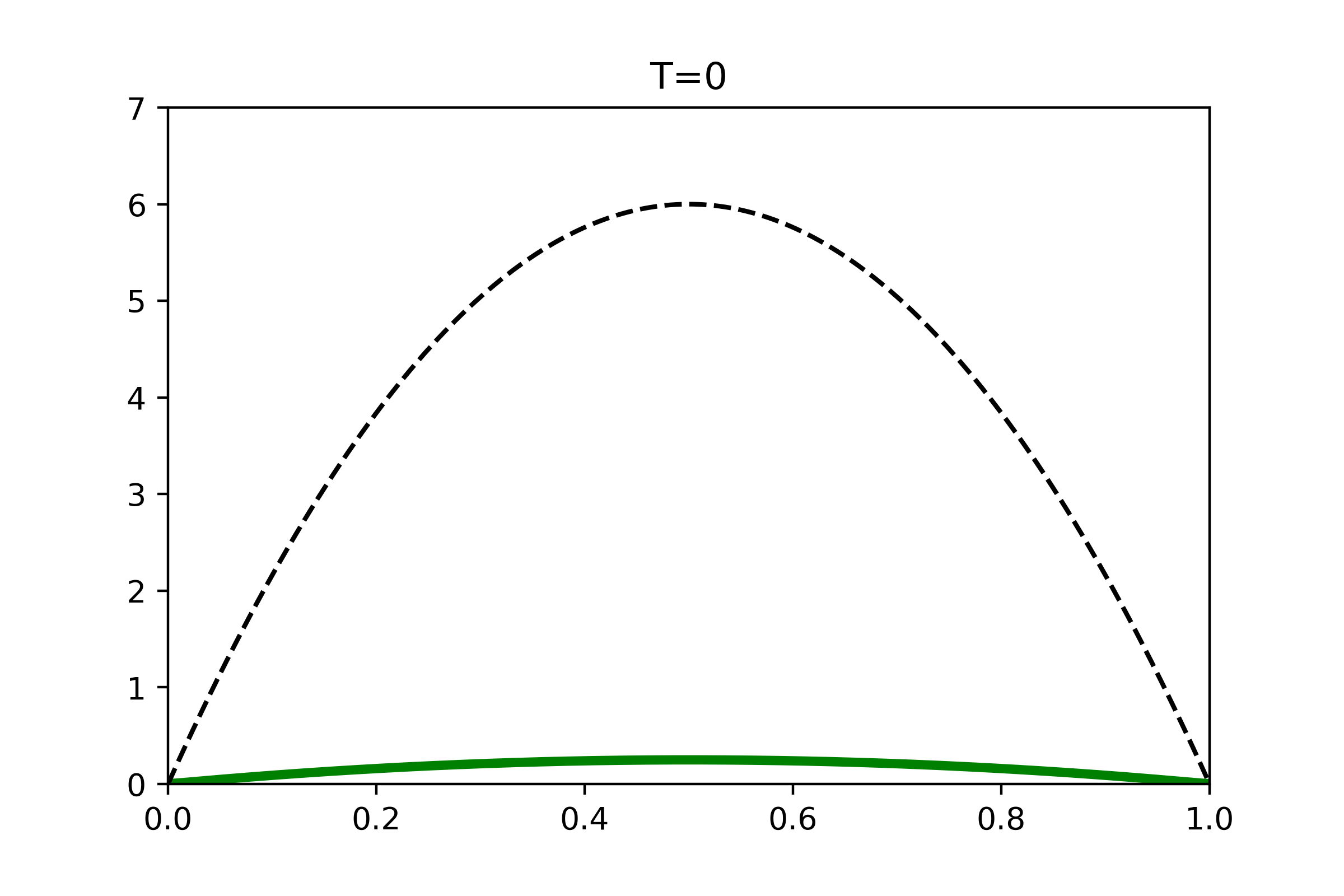}
\includegraphics[width=0.33\textwidth]{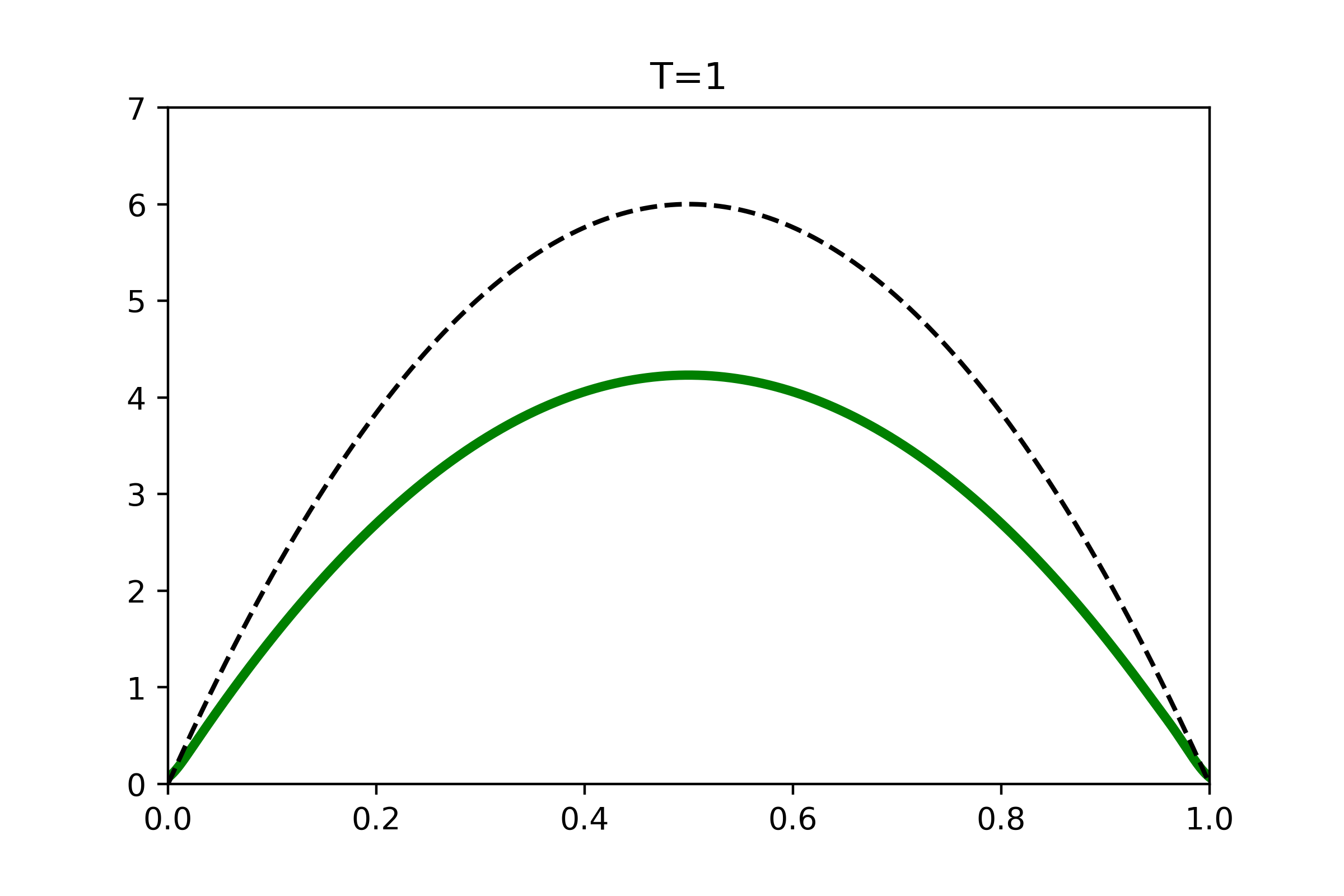}
\includegraphics[width=0.33\textwidth]{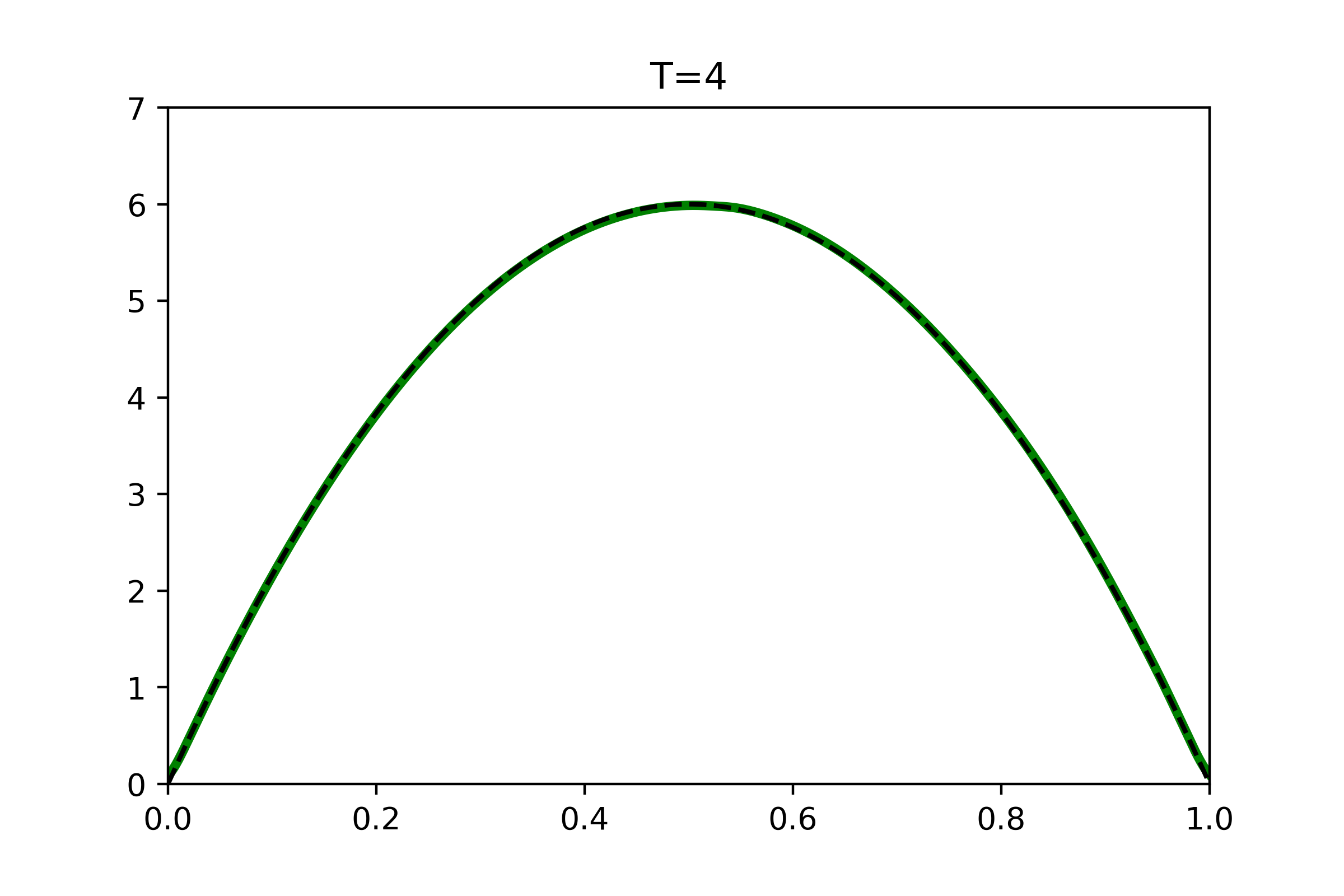}
\caption{Solution of \eqref{equation advection-selection} at different time steps, with $n^0(x)=x(1-x), r(x)=6-4x$.}
\label{alpha 1}
\end{subfigure}
\begin{subfigure}{\textwidth}
\includegraphics[width=0.33\textwidth]{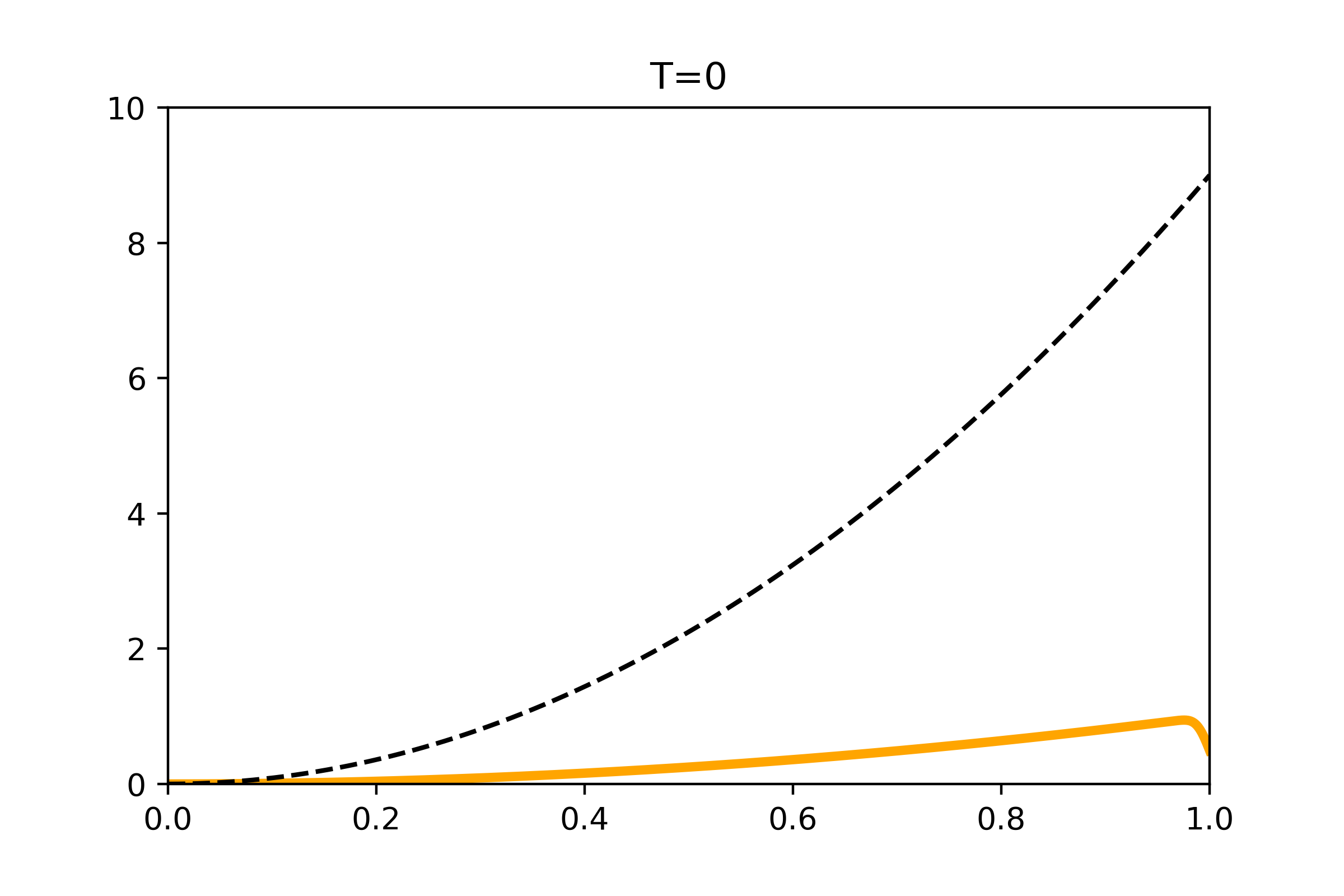}
\includegraphics[width=0.33\textwidth]{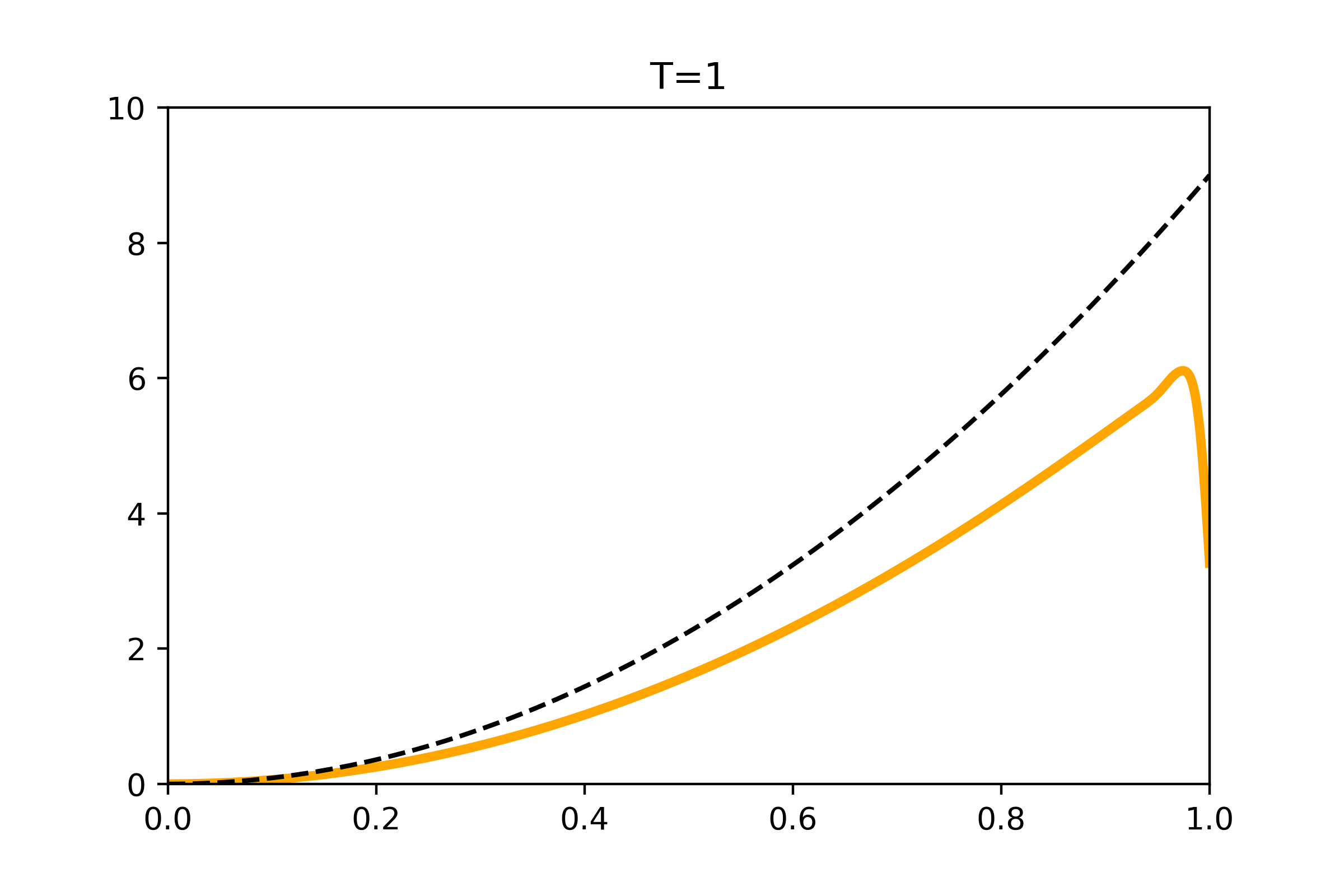}
\includegraphics[width=0.33\textwidth]{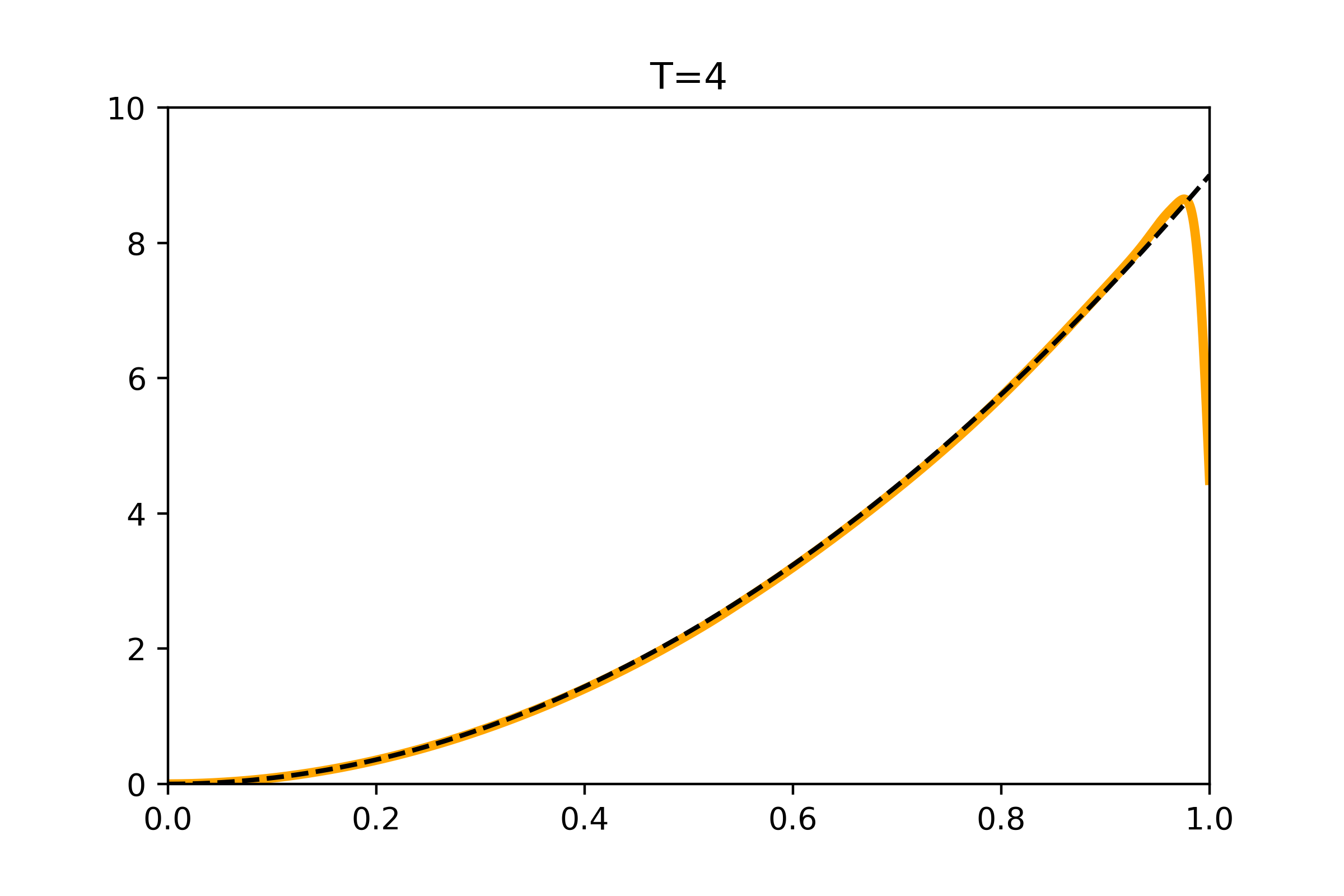}
\caption{Solution of \eqref{equation advection-selection} at different time steps, with $n^0(x)=x^2, r(x)=6-4x$.}
\label{alpha 2}
\end{subfigure}
\begin{subfigure}{\textwidth}
\includegraphics[width=0.33\textwidth]{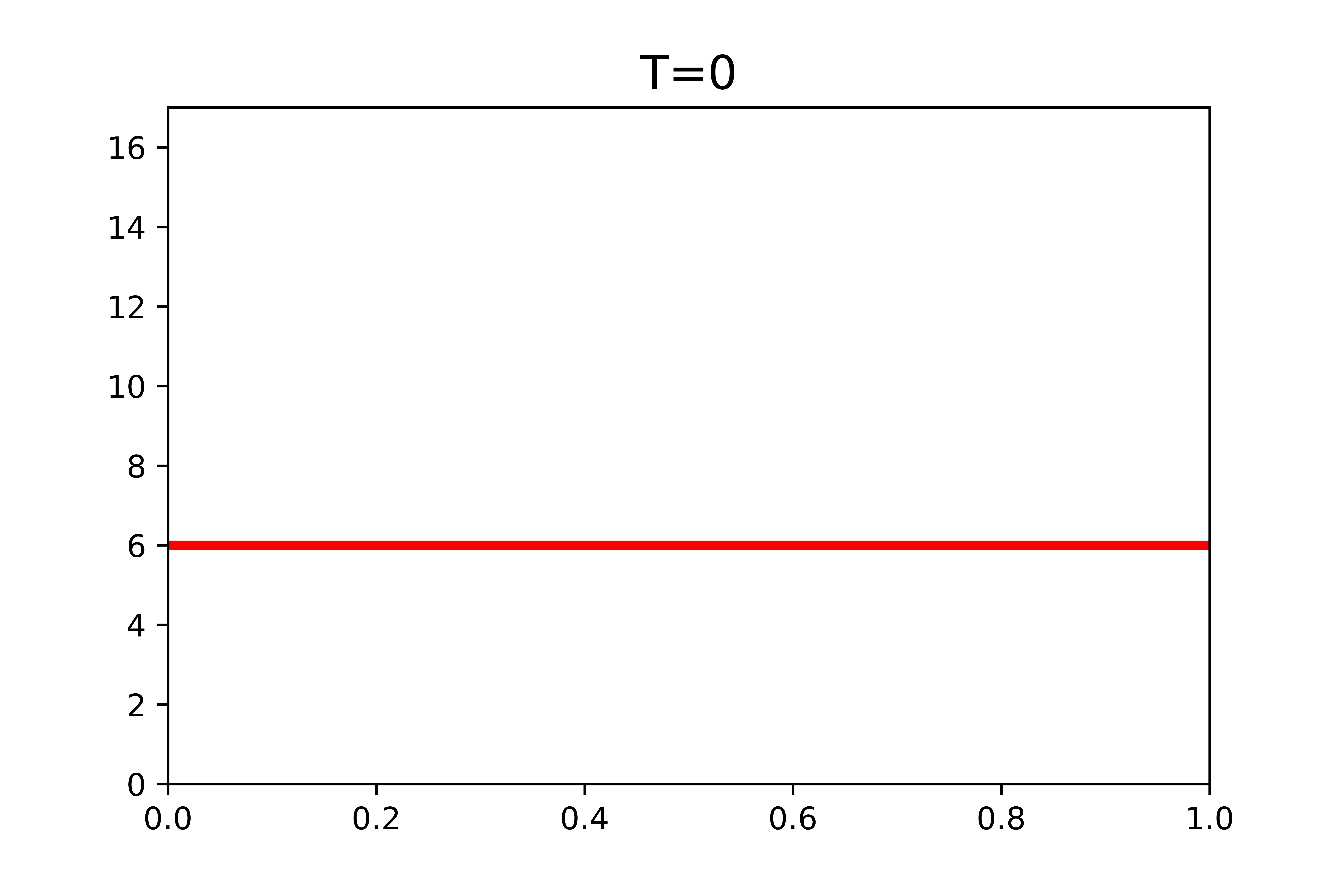}
\includegraphics[width=0.33\textwidth]{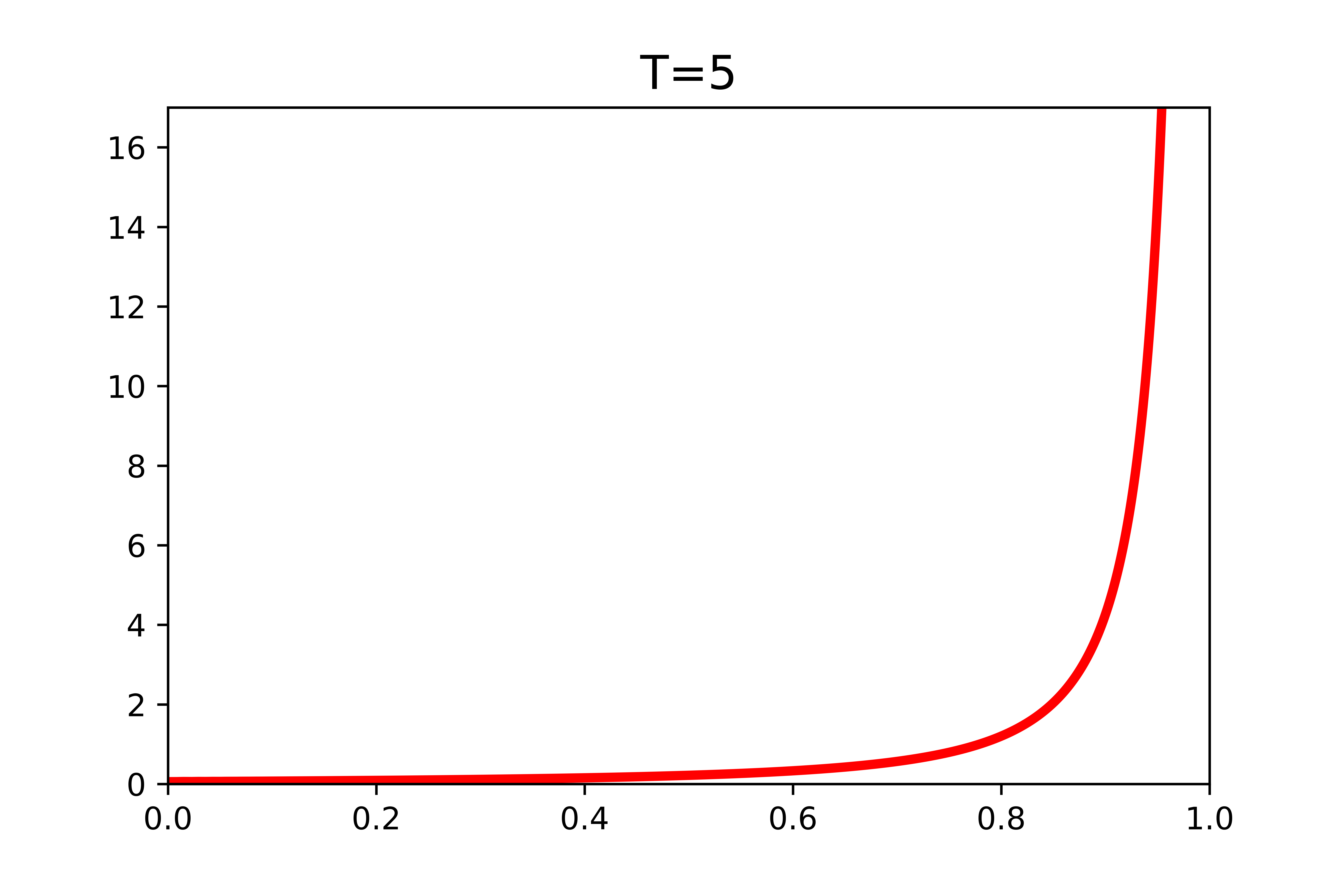}
\includegraphics[width=0.33\textwidth]{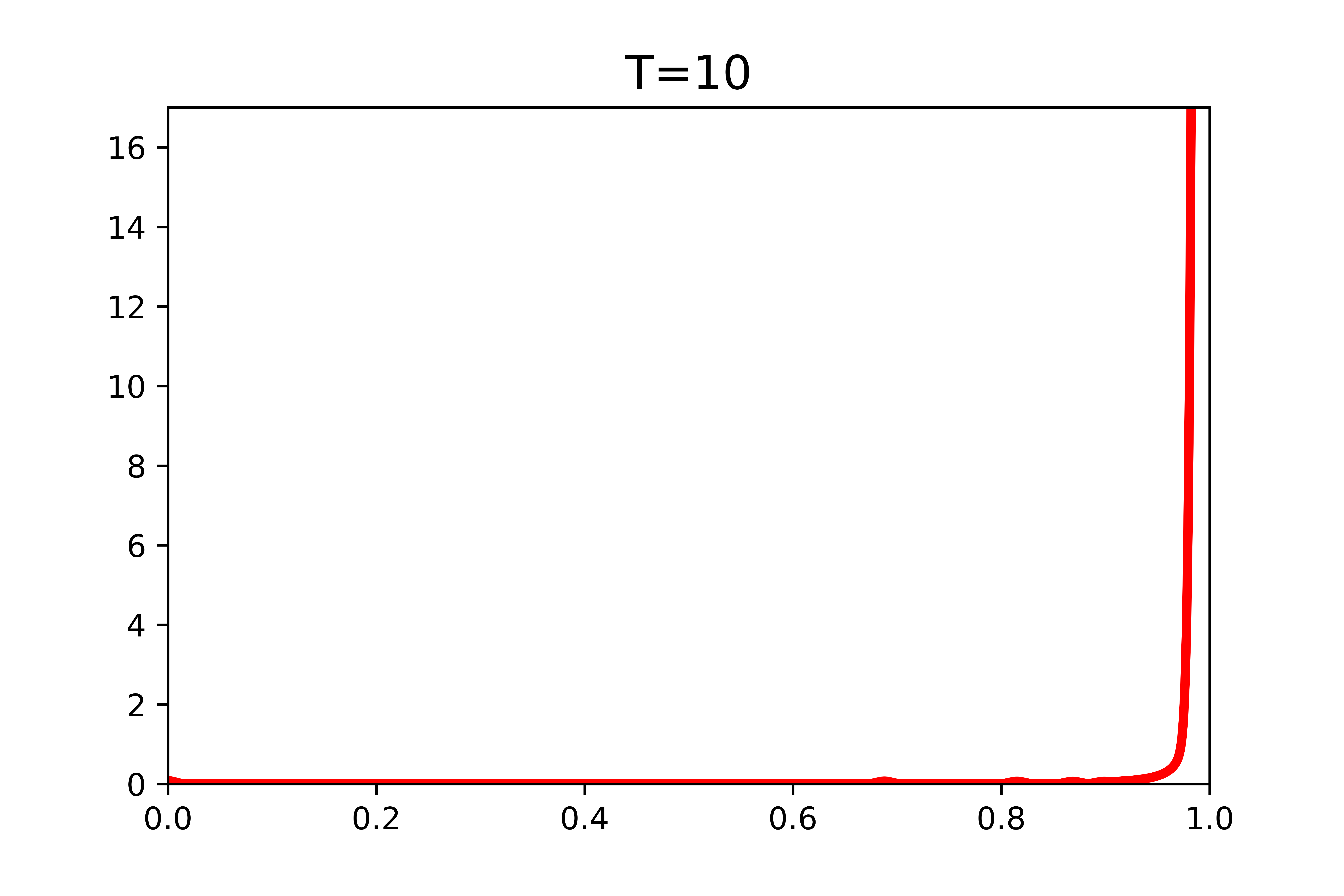}
\caption{Solution of \eqref{equation advection-selection} at different time steps, with $n^0(x)=6, r(x)=6-0.5x$.}
\label{Dirac}
\end{subfigure}
\caption{Different possible regimes of convergence for the solution of \eqref{equation advection-selection}. The first three lines (green, blue and orange curves), show the convergence to a function in $L^1$, which can be explicitly computed (see \cite{guilberteau2023long}), and is represented by a black dashed line. Note that the limit function is different when the initial condition changes. The last  line (red curves) shows the convergence to a Dirac mass in $1$. In the four cases, we have chosen $a(x)=x(1-x)$, $N=5000$, $h=\frac{1}{N}$, $\e=\sqrt{h}$ and the cut-off function $\varphi$ is a Gaussian.} 
\label{fig: advection-growth}
\end{figure}

\appendix
\section{Proof of the results over the characteristics}\label{AnnA}
In order to prove results which involve the use of absolute values, we introduce a smooth re-normalizing sequence of functions. Consider a sequence of smooth positive functions $\beta_{\varepsilon}$ satisfying $\beta_{\varepsilon}(0)=0$, $\beta_{\varepsilon}(s)>0$ for all $s\neq 0$, $\beta_{\varepsilon}(s)\leqslant |s|$, $\beta_{\varepsilon}(s)\rightarrow |s|$ almost everywhere, $|\dot \beta_{\varepsilon}(s)|\leqslant 1$  and $s \dot \beta_{\varepsilon}(s)\rightarrow |s|$ almost everywhere. For example we may choose
\begin{equation}
 \beta_{\varepsilon}(s)=\left\{
\begin{matrix}
-s-\varepsilon(1-\frac{2}{\pi})&\mbox{ if }&s\leqslant -\varepsilon,\\
&&\\
\frac{2\varepsilon}{\pi}\left(1-\cos(\frac{\pi}{2\varepsilon}s)\right)&\mbox{ if }& -\varepsilon<s<\varepsilon,\\
&&\\
s-\varepsilon(1-\frac{2}{\pi})&\mbox{ if }&s\geqslant \varepsilon.
\end{matrix}
\right.\label{renorm}   
\end{equation}

\begin{proof}[Proof of  Lemma \ref{ChardX2}.]
We introduce the notation
\[
\Delta X_j(t):=X^j_{u_1}(t,y_1)-X^j_{u_2}(t,y_2).
\]
For all $t\in [0, T]$, the function
\[
U_{\varepsilon}(t):=\sum\limits_{j=1}^d\beta_{\varepsilon}(\Delta X_j(t))
\]
satisfies then the relation
\begin{align*}
    \dot U_{\varepsilon}(t)=&\sum\limits_{j=1}^d \dot \beta_{\varepsilon}(\Delta X_j(t))\left(a_j(t,X_{u_1}(t,y_1),(I_au_1)(t,X_{u_1}(t,y_1)))-a_j(t,X_{u_2}(t,y_2),(I_au_2)(t,X_{u_2}(t,y_2)))\right)\\
    \leqslant& \sum\limits_{j=1}^d\left(\|a_j\|_{W^{1,\infty}_x}\sum_{i=1}^d|\Delta X_i(t)|+\|a_j\|_{W^{1,\infty}_I}|(I_au_1)(t,X_{u_1}(t,y_1))-(I_au_2)(t,X_{u_2}(t,y_2))|\right)\\
    \leqslant& d \|a\|_{W^{1,\infty}_{x,I}}\left(\sum_{i=1}^d|\Delta X_i(t)|+|(I_au_1)(t,X_{u_1}(t,y_1))-(I_au_2)(t,X_{u_2}(t,y_2))|\right)\\
    \leqslant& d \|a\|_{W^{1,\infty}_{x,I}}\left((1+\|\psi_a\|_{W^{1,\infty}_{x} L^{\infty}_y  }\|u_1\|)\sum_{i=1}^d|\Delta X_i(t)|+\|\psi_a\|_{L^{\infty}}\|u_1-u_2\|_{L^1(\R^d})\right).
\end{align*}
Integrating between $0$ and $t$ we get
\begin{align*}
    &U_{\varepsilon}(t,y)-U_{\varepsilon}(0,y)\\
    &\leqslant d \|a\|_{L^{\infty}_tW^{1,\infty}_{x,I}}\left((1+\|\psi_a\|_{L^{\infty}_{t,y}W^{1,\infty}_{x}   }\|u_1\|)\int_0^t\sum_{i=1}^d|\Delta X_i(t)|ds+\|\psi_a\|_{L^{\infty}}\int_0^t\|u_1-u_2\|_{L^1(\R^d)}ds\right).
\end{align*}
Taking the limit when $\varepsilon$ goes to $0$ and applying Gr\"onwall's lemma we get the desired result.
\end{proof}

\begin{proof}[Proof of Lemma \ref{ChardX}.]
We explicitly give the proof for $k=1$. The proof for higher values of $k$ follows the same ideas.\\
Thanks to the hypothesis over $a$, the function $X_u$ is one time differentiable with respect to $y$, and directly from \eqref{CharL} we get the system of equations
\begin{align}
\begin{cases}
\dot{\partial_{y_i}X_u(t,y)}=J_a(t,X_u(t,y)) \partial_{y_i}X_u(t,y),\quad t\in [0,T],\\
\partial_{y_i}X_u(0,y)=e_i,
\end{cases}\label{dyCharL}
\end{align}
for all values of $i\in\{1,\ldots,d\}$, where 
\begin{equation}
    \left[J_a(t,x)\right]_{ij}:=\partial_{x_i} a_j(t,x,(I_au)(t,x))+\partial_Ia_j(t,x,(I_au)(t,x))\int_{\R^d}\partial_{x_i}\psi_a(t,x,y)u(t,y)dy,\label{Ja}
\end{equation}
is the Jacobian matrix of the function $a(t,x,(I_au)(t,x))$ and the $e_i$ represent the canonical basis of $\R^d$.\\
The function 
\[
V_{\varepsilon}(t,y):=\sum\limits_{i,j=1}^n\beta_{\varepsilon}(\partial_{y_i}X^j_u(t,y))
\]
satisfies then
\[
\dot V_{\varepsilon}(t,y)=\sum\limits_{i,j=1}^n \dot \beta_{\varepsilon}(\partial_{y_i}X^j_u(t,y))\sum\limits_{k=1}^d\left[J_a(t,X_u(t,y))\right]_{kj} \partial_{y_i}X^k_u(t,y)
\]
and consequently
\[
\dot V_{\varepsilon}(t,y)\leqslant d\|a\|_{W^{1,\infty}_{x,I}}(1+\|\psi_a\|_{W^{1,\infty}_xL^{\infty}_y}\|u\|)\sum\limits_{i,j=1}^n|\partial_{y_i}X^j_u(t,y)|.
\]
Integrating between $0$ and $t$, we obtain the relation
\[
  V_{\varepsilon}(t,y)-V_{\varepsilon}(0,y)\leqslant d\,\tilde{\alpha}_1\int_0^t\sum\limits_{i,j=1}^n|\partial_{y_i}X^j_u(s,y)|ds
\]
with $\tilde{\alpha}_1:=\sup\limits_{t\in [0,T]}\|a\|_{W^{1,\infty}_{x,I}}(1+\|\psi_a\|_{W^{1,\infty}_xL^{\infty}_y}\|u\|)$, which after taking the limit when $\varepsilon$ goes to $0$ leads to
\[
 \sum\limits_{i,j=1}^n|\partial_{y_i}X^j_u(t,y)|\leqslant d+d\tilde{\alpha}_1\int_0^t\sum\limits_{i,j=1}^n|\partial_{y_i}X^j_u(s,y)|ds.
\]
We obtain \eqref{bpX} thanks to Gr\"onwall's lemma.\\
In order to prove \eqref{ContpX} we adopt the notation
\[\Delta \partial_{y_k}X^j(t):=\partial_{y_k}X^j_{u_1}(t,y_1)-\partial_{y_k}X^j_{u_2}(t,y_2)\]
and define
\[
[J_{X_u}(t,y)]_{jk}:=\partial_{y_k}X^j_u(t,y),
\]
which satisfies the relation
\[
\dot{J_{X_u}}(t,y)=J_a(t,X_u(t,y))J_{X_u}(t,y),\ J_{X_u}(0,y)=I_d,
\]
with $J_a(t,x)$ as defined on \eqref{Ja}. Consequently, we have that
\[
D_{\varepsilon}(t):=\sum\limits_{j=1}^{d}\sum\limits_{k=1}^{d}\beta_{\varepsilon}\left(\Delta \partial_{y_k}X^j(t)\right)
\]
satisfies for all $t\in [0, T]$
\begin{align*}
    \dot{D_{\varepsilon}}(t)=&\sum\limits_{j=1}^{d}\sum\limits_{k=1}^{d} \dot \beta_{\varepsilon}\left(\Delta \partial_{y_k}X^j(t)\right)\sum\limits_{i=1}^d\left([J_{a}(t,X_{u_1}(t,y_1))]_{ij}\partial_{y_k}X^i_{u_1}(t,y_1)-[J_{a}(t,X_{u_2}(t,y_2))]_{ij}\partial_{y_k}X^i_{u_2}(t,y_2)\right)\\
    \leqslant& \sum\limits_{j=1}^{d}\sum\limits_{k=1}^{d}\sum\limits_{i=1}^d\left|[J_{a}(t,X_{u_1}(t,y_1))]_{ij}\partial_{y_k}X^i_{u_1}(t,y_1)-[J_{a}(t,X_{u_2}(t,y_2))]_{ij}\partial_{y_k}X^i_{u_2}(t,y_2)\right|\\
    \leqslant& \sum\limits_{j=1}^{d}\sum\limits_{k=1}^{d}\sum\limits_{i=1}^d\left|[J_{a}(t,X_{u_1}(t,y_1))]_{ij}\right|\left|\Delta \partial_{y_k}X^i(t)\right|\\
    &+\sum\limits_{j=1}^{d}\sum\limits_{k=1}^{d}\sum\limits_{i=1}^d\left|[J_{a}(t,X_{u_1}(t,y_1))]_{ij}-[J_{a}(t,X_{u_2}(t,y_2))]_{ij}\right|\left|\partial_{y_k}X^i_{u_2}(t,y_2)\right|.
\end{align*}
From the hypothesis over $a$ and $\psi_a$ we see that
\[
\left|[J_{a}(t,X_{u_1}(t,y_1))]_{ij}\right|\leqslant \|a\|_{W^{1,\infty}_{x,I}}(1+\|\psi_a\|_{W^{1,\infty}_{x}L^{\infty}_y}\|u_1\|).
\]
Furthermore, from the definition of $J_a(t,x)$ we conclude that there exists a constant $C$, depending only on $\|a\|_{W^{2,\infty}_{x,I}}$, $\|\psi_a\|_{W^{2,\infty}_{x}L^{\infty}_y}$ and $\|u_i\|$ such that
\begin{align*}
    \left|[J_{a}(t,X_{u_1}(t,y_1))]_{ij}-[J_{a}(t,X_{u_2}(t,y_2))]_{ij}\right|&\leqslant C(\sum\limits_{j=1}^{d}|X^j_{u_1}(t,y_1)-X^j_{u_2}(t,y_2)|+\|u_1-u_2\|_{L^1(\R^d)})\\
    &\leqslant C(|y_1-y_2|+\|u_1-u_2\|_1+\|u_1-u_2\|_{L^1(\R^d)}),
\end{align*}
where we have used the results from Lemma \ref{ChardX2} on the second line.\\ 
Putting all estimates together, we conclude that there exist constants $C_1$ and $C_2$ only depending on $\|a\|_{W^{2,\infty}_{x,I}}$, $\|\psi_a\|_{W^{2,\infty}_{x}L^{\infty}_y}$ and $\|u_i\|$, such that
\[
\dot{D_{\varepsilon}}(t)\leqslant C_1 \sum\limits_{j=1}^{d}\sum\limits_{k=1}^{d}|\partial_{y_k}X^j_{u_1}(t,y_1)-\partial_{y_k}X^j_{u_2}(t,y_2)|+C_2 (|y_1-y_2|+\|u_1-u_2\|_1+\|u_1-u_2\|_{L^1(\R^d)}).
\]
Integrating in time, using Gr\"onwall's lemma and taking the limit when $\varepsilon$ goes to zero, we obtain \eqref{ContpX}.
\end{proof}

\begin{proof}[Proof of Lemma \ref{ChardXinv}.]
We explicitly give the proof for $k=1$. The proof for higher values of $k$ follows the same ideas.\\
Differentiating once each component of the equality $X_u(t,X^{-1}_u(t,x))=x$ with respect to each of the variables $x_k$, we obtain the family of relations
\[\sum_{i=1}^d\partial_{y_i}X^j_u(t,X^{-1}_u(t,x))\partial_{x_k}\left(X^{-1}_{u}\right)^i(t,x)=\delta_{jk},\ j,k=1,\ldots,d,\]
where $\delta_{jk}$ represents the Kronecker's delta. Written in matrix form, this equality reads
\[
J_{X_u}(t,X^{-1}_u(t,x))J_{X^{-1}_u}(t,x)=I_d.
\]
It is known that the matrix $J_{X_u}(t,y)$ is invertible for all values of $x$, furthermore, its determinant is given by the expression 
\[
det(J_{X_u}(t,y))=e^{\int_0^t\nabla_x\cdot a(s,y,(I_au)(s,y))+\partial_Ia(s,y,(I_au)(s,y))\cdot\int\limits_{\R^d}\nabla_x\psi_a(s,y,z)u(s,z)dzds}\geqslant c_T>0
\] 
for all values of $t\in[0,T]$ and $y\in\R^d$.\\
We conclude by writing
\[
J_{X^{-1}_u}(t,x)=J^{-1}_{X_u}(t,X^{-1}_u(t,x)),
\]
and noticing that all of the components of $J^{-1}_{X_u}(t,X^{-1}_u(t,x))$ are a combination of sums and multiplications of the components of $J_{X_u}(t,X^{-1}_u(t,x))$, divided by $det(J_{X_u}(t,y))$. The bound \eqref{bpX} from Lemma \ref{ChardX}, together with the lower bound for the determinant of $J_{X_u}(t,y)$ gives the bound \eqref{bpXinv} over the components of $J_{X^{-1}_u}(t,x)$.
\end{proof}

\begin{proof}[Proof of Lemma \ref{ChardXinv2}.]
We explicitly give the proof for $k=1$. The proof for higher values of $k$ follows the same ideas.\\
Differentiating with respect to $t$ the relation $X_{u_i}(t,X^{-1}_{u_i}(t,x))=x$, for $i=1,2$, we see that
\[
a(t,x,(I_au_i)(t,x))+J_{X_{u_i}}(t,X^{-1}_{u_i}(t,x))\dot X^{-1}_{u_i}(t,x)=0,
\]
which gives 
\begin{equation}
   \dot X^{-1}_{u_i}(t,x)=-J^{-1}_{X_{u_i}}(t,X^{-1}_{u_i}(t,x))a(t,x,(I_au_i)(t,x)). 
\end{equation}
From now on we adopt the notations
\begin{align*}
A^i(t,x)&:=a(t,x,(I_au_i)(t,x)),\\
K^i(t,x)&:=-J^{-1}_{X_{u_i}}(t,X^{-1}_{u_i}(t,x)),\\
\Delta X^{-1}_j(t,x)&:=\left(X^{-1}_{u_1}\right)^{j}(t,x)-\left(X^{-1}_{u_2}\right)^{j}(t,x).
\end{align*}
The function
\[
W_{\varepsilon}(t,x):=\sum\limits_{j=1}^d\beta_{\varepsilon}(\Delta X^{-1}_j(t,x))
\]
satisfies the relation
\begin{align*}
    \dot W_{\varepsilon}(t,x)&=\sum\limits_{j=1}^d\sum\limits_{k=1}^d \dot \beta_{\varepsilon}(\Delta X^{-1}_j(t,x))\left(K^1_{jk}(t,x)A^1_k(t,x)-K^2_{jk}(t,x)A^2_k(t,x)\right)\\
    &\leqslant \sum\limits_{j=1}^d\sum\limits_{k=1}^d|K^1_{jk}(t,x)A^1_k(t,x)-K^2_{jk}(t,x)A^2_k(t,x)|.
\end{align*}
From Lemma \ref{ChardXinv} we know that all components of $K^i$ are uniformly bounded by a constant only depending on $T$ and $\|u_i\|$. We deduce from the hypothesis over $a$ that the components of $A^i$ are uniformly bounded by $\tilde{a}:=\|a\|_{L^{\infty}}$. Therefore
\[
 \dot W_{\varepsilon}(t,x)\leqslant d\tilde{C}(T,\|u_1\|)\sum\limits_{k=1}^d |A^1_k(t,x)-A^2_k(t,x)|+\tilde{a}\sum\limits_{j=1}^d\sum\limits_{k=1}^d|K^1_{jk}(t,x)-K^2_{jk}(t,x)|.
\]
The function $a$ being $L$-Lipschitz with respect to the $I$ variable, we have that, for all values of $k$
\[
|A^1_k(t,x)-A^2_k(t,x)|\leqslant L\|\psi_a\|_{L^{\infty}}\|u_1-u_2\|_{L^1(\R^d)}.\]
On the other hand, from the definition of $K^i$ and Lemma \ref{ChardX} we conclude that there exists $C$, depending on $T$, $\|a\|_{W^{2,\infty}_{x,I}}$, $\|\psi_a\|_{W^{2,\infty}_{x}L^{\infty}_y}$ and $\|u_i\|$ such that
\[
|K^1_{jk}(t,x)-K^2_{jk}(t,x)|\leqslant \left\{\begin{matrix}
    0,&\mbox{ if }&\partial_I a=0,\\
    &&\\
C\left(\sum\limits_{j=1}^{d}|\left(X^{-1}_{u_1}\right)^j(t,x)-\left(X^{-1}_{u_2}\right)^j(t,x)|+\|u_1-u_2\|_1\right),&\mbox{ if }&\partial_I a\neq 0.
\end{matrix}\right.
\]
Putting everything together, integrating between $0$ and $t$, taking the limit when $\varepsilon$ goes to $0$ and applying Gr\"onwall's lemma we get \eqref{ContXinv}.
\end{proof}
\section{Existence of solution for a system of ODEs with infinitely many unknowns and equations}\label{AnnB}
\begin{proof}[Proof of Lemma \ref{CLinf}.]
 For all $u\in X^T_h$, there exists a sequence of elements $u^{\delta}\in X^T_h $ such that:
 \begin{enumerate}
     \item[1)] $\mathcal{K}^{\delta}_h:=\{k\in\mathcal{J}_h: u^{\delta}_k\neq 0\}$ has a finite number of elements.
     \item[2)]
     \[
     \lim\limits_{\delta\rightarrow 0}\|u-u^{\delta}\|_{1,h}=0.
     \]
 \end{enumerate}
 We denote $K^{\delta}:=|\mathcal{K}^{\delta}_h|$ and notice that the system
 \begin{equation}
     \dot{x^{\delta}_k}(t)=A_{u^{\delta},w}(t,x^{\delta}_k),\ x^{\delta}_k(0)=x^0_k,\label{infapp}
 \end{equation}
 is composed of a coupled system of $K^{\delta}$ equations and unknowns (corresponding to those $k\in \mathcal{K}^{\delta}_h$), and an uncoupled infinite number of equations, corresponding to those $k\not\in \mathcal{K}^{\delta}_h$. Therefore, thanks to the classic Cauchy-Lipschitz theory, the system \eqref{infapp} has a unique solution $x^{\delta}_k\in\mathcal{C}^1([0,T])$, $k\in\mathcal{J}_h$.\\
 We claim that for all values of $k$, the sequence $x^{\delta_1}_k-x^{\delta_2}_k$ is a Cauchy sequence in $\mathcal{C}^1([0,T])$, therefore is has a limit that we will call $x_k(t)$, which is solution to \eqref{inftot}.\\
We first remark that $x^{\delta_1}-x^{\delta_2}\in Y^T_h$ due to the fact that $|x^{\delta}_k(t)-x^0_k|\leqslant \|a\|_{L^{\infty}}T$ for all values of $k$ and $\delta$.\\ 
 Consider now $\beta_{\varepsilon}$ as defined in \eqref{renorm}, then
 \begin{align*}
     \dot{\beta_{\varepsilon}}(x^{\delta_1}_k-x^{\delta_2}_k)&\leqslant |A_{u^{\delta_1},w}(t,x^{\delta_1}_k)-A_{u^{\delta_2},w}(t,x^{\delta_2}_k)|\\
     &\|a\|_{W^{1,\infty}_{x}}|x^{\delta_1}_k-x^{\delta_2}_k|+\|a\|_{W^{1,\infty}_{I}}|I_a(t,x^{\delta_1}_k,u^{\delta_1},w)-I_a(t,x^{\delta_1}_k,u^{\delta_2},w)|.
 \end{align*}
 Noticing that
 \begin{align*}
     |I_a(t,x^{\delta_1}_k,u^{\delta_1},w)-I_a(t,x^{\delta_1}_k,u^{\delta_2},w)|&=|\sum\limits_{j\in\mathcal{J}_h}\biggl(u^{\delta_1}_j(t)\psi_a(t,x^{\delta_1}_k,x^{\delta_1}_j(t))-u^{\delta_2}_j(t)\psi_a(t,x^{\delta_2}_k,x^{\delta_2}_j(t))\biggr)w_j(t)|\\
     &\leqslant (\|\psi_a\|_{L^{\infty}}\|u^{\delta_1}-u^{\delta_2}\|_{1,h}+\|u^{\delta_2}\|_{1,h}\|\psi_a\|_{W^{1,\infty}_{x,y}}\|x^{\delta_1}-x^{\delta_2}\|_{\infty,h})\|w\|_{\infty,h},
 \end{align*}
 we deduce the existence of two constants\footnote{Notice that in order to obtain the estimate over $I_a$, we used the hypothesis of differentiability over both variables on $\psi_a$}, $C_1$ and $C_2$, only depending on $a$, $\psi_a$, $u$ and $w$, such that
 \[
 \dot{\beta_{\varepsilon}}(x^{\delta_1}_k-x^{\delta_2}_k)\leqslant C_1\|x^{\delta_1}-x^{\delta_2}\|_{\infty,h}+C_2\|u^{\delta_1}-u^{\delta_2}\|_{1,h}.
 \]
 Integrating between $0$ and $t$, taking the maximum over $k$ and $t$ and using Gr\"onwall's lemma, we conclude that there exists a constant $C_T$, only depending on $T$ and the coefficients of the problem, such that
 \[
 \|x^{\delta_1}-x^{\delta_2}\|_{\infty,h}\leqslant C_T\|u^{\delta_1}-u^{\delta_2}\|_{1,h}.
 \]
 Proceeding in a similar way with the absolute value of $\dot{x}^{\delta_1}-\dot{x}^{\delta_2}$ we obtain that
 \[
 \|\dot{x}^{\delta_1}-\dot{x}^{\delta_2}\|_{\infty,h}\leqslant C_T\|u^{\delta_1}-u^{\delta_2}\|_{1,h}.
 \]
 Recalling that $u^{\delta}$ is a Cauchy sequence on $X^T_{h}$, then so it is $x^{\delta}_k$ on $\mathcal{C}^1([0,T])$, for each $k$.\\
 Let $x:=\{x_k\}_{k\in\mathcal{J}_h}$ be the limit of $x^{\delta}$ when $\delta$ goes to $0$. With a simple continuity argument we conclude that $x$ is a solution of \eqref{inftot} over $[0,T]$. The uniqueness can be obtained by assuming the existence of two solutions, deriving the equation satisfied by the difference and using Gr\"onwall's lemma to conclude that they have to be equal.
\end{proof}
\section{A result from approximation theory}\label{AppTheo}
As mentioned before, Lemma \ref{AT} is a direct corollary of Lemma 8 in \cite{mas1987particle}, that we recall here
\begin{lem}\label{L8fromRav}
Let $k>d$ an integer. Assume that 
\begin{align*}
    a\in (L^\infty(0,T;W^{k+1, \infty}(\R^d)))^d. 
\end{align*}
Then, there exists a constant $C>0$ such that for all functions $\varphi \in W^{k,p}(\R^d) $, $1\leqslant p\leqslant +\infty$, and $t\in [0,T]$, 

\begin{align*}
    \lVert \varphi-\underset{i\in\mathcal{J}_h}{\sum}{w_i(t)\varphi(x_i(t))\delta(\cdot -x_i(t))}\rV_{W^{-k, p}(\R^d)}\leqslant Ch^k\lVert \varphi\rVert_{W^{k,p}(\R^d)}.
\end{align*}
\label{negative Sobolev}
\end{lem}

Given that for fixed functions $\nu\in X^T_h $ and $w\in Y^T_h$ we have the inclusion
\[
A_{\nu,w}:(t, x)\mapsto a(t,x,I_a(t,x,\nu,w))\in (L^\infty(0,T; W^{k+1, \infty}(\R^d)))^d,
\]
then Lemma \ref{L8fromRav} holds true as well for the values of $x_i$ obtained in Section \ref{TPM}.\\

\begin{proof}[Proof of Lemma \ref{AT}.]
We recall that $W^{-k, 1}(\R^d)$ is the dual space of $W^{k, \infty}(\R^d)$. Thus, for any $\psi\in W^{-k, 1}(\R^d)$ we have
\begin{align*}
    \lV \psi \rV_{-k, 1} = \underset{f\in W^{k, \infty}(\R^d)}{\sup}\frac{\lvert \langle \psi, f  \rangle \rvert }{\quad \lVert f \rVert_{k, \infty }}. 
\end{align*}
Since the function $f\equiv 1$ belongs to $W^{k, \infty}(\R^d)$, and has norm equal to $1$ in this space, we get for all $\varphi\in W^{k, p}(\R^d)$ 
\begin{align*}
    \bigg\lvert \int_{\R^d}{\varphi(x)dx}-\underset{i\in\mathcal{J}_h}{\sum}{w_i(t)\varphi(x_i(t))}\bigg\rvert &=\bigg\lvert \langle \varphi-\underset{i\in\mathcal{J}_h}{\sum}{w_i(t)\varphi(x_i(t))\delta(\cdot -x_i(t))}, 1 \rangle \bigg\rvert\\ &\leqslant  \lVert \varphi-\underset{i\in\mathcal{J}_h}{\sum}{w_i(t)\varphi(x_i(t))\delta(\cdot -x_i(t))}\rV_{-k, 1}. 
\end{align*}
We conclude by applying Lemma \ref{negative Sobolev} with $p=1$. 
\end{proof}
\section{Proofs of convergence results from ODE theory}\label{ODEresults}
This appendix is dedicated to the proofs of lemma \ref{convergence negative part}, used in subsection \ref{conv infinite time}. In order to prove this lemma, we use the following result: 

\begin{lem}
Let $\alpha>0$ and $B\in L^1(\mathbb{R}_+)$. Then, all the solutions of the ODE 
\[\dot u(t)=-\alpha u(t)+ B(t)\]
are in $L^1(\mathbb{R}^+).$
\label{lem ODE L1}
\end{lem}

\begin{proof}[Proof of lemma \ref{lem ODE L1}]
The solution of this ODE is explicitly given by 
\[u(t)=u(0)e^{-\alpha t}+\int_0^t{e^{-\alpha(t-s)}}B(s)ds.\]
Hence,
\begin{align*}
\int_0^{+\infty}{\lvert u(t)\rvert dt}&\leqslant \lvert u(0)\rvert \int_0^{+\infty}{e^{-\alpha t}}dt + \iint_{\mathbb{R}_+^2}{e^{-\alpha(t-s)}\lvert B(s)\rvert \mathbbm{1}_{\{s\leqslant t\}}dsdt}.
\end{align*}
With the change of variables $y=s, z=t-s$, we get 
\begin{align*}
\iint_{\mathbb{R}_+^2}{e^{-\alpha(t-s)}\lvert B(s)\rvert \mathbbm{1}_{\{s\leqslant t\}}dsdt}\leqslant\int_0^{+\infty}{\lvert B(y)\rvert  dy}\int_0^{+\infty}{e^{-\alpha z} dz},
\end{align*}
which concludes the proof. 

\end{proof}

\begin{proof}[Proof of Lemma \ref{convergence negative part}]
First, let us note that if $\dot u$ is a BV function, \textit{i.e.} if $\int_{0}^{+\infty}{\lvert \dot u(t)\rvert dt}<+\infty$, then $u$ is a Cauchy function, and thus converges. Let us denote $v:=\dot u$.  Since $u$ is assumed to be bounded, and $\lvert v\rvert =v+2{v}^-$, where ${v}^-$ denotes the negative part of $v$, it is enough to prove that ${v}^-\in L^1(\R^d)$. 
By hypothesis, 
\[\dot v(t)= -p(t)v(t)+P(t)+B(t), \]
which implies that
\[\dot {v^-}(t)\leqslant -p_0\: v^-(t)+B(t).\]
We conclude, according to lemma \ref{lem ODE L1}, that $v^-\in L^1(\R^d)$, which implies that $u$ converges. 
\end{proof}

\section{Acknowledgements}
The authors would like to acknowledge Camille Pouchol and Nastassia Pouradier for their invaluable guidance and insightful comments.\\
This research was supported by the European Union's Horizon 2020 research and innovation programme under the Marie Skłodowska-Curie grant agreements No 754362 and No 955708.

\bibliography{References.bib}
\bibliographystyle{unsrt}

\end{document}